%% file: paper.tex
\pdfoutput=1
\documentclass[sn-basic,Numbered]{sn-jnl}
\usepackage{apxproof}
\usepackage{graphicx}%
\usepackage{enumerate}%
\usepackage{multirow}%
\usepackage{bbm}%
\usepackage{amsmath,amssymb,amsfonts}%
\usepackage{amsthm}%
\usepackage{mathrsfs}%
\usepackage[title]{appendix}%
\usepackage{xcolor}%
\usepackage{textcomp}%
\usepackage{manyfoot}%
\usepackage{booktabs}%
\usepackage{algorithm}%
\usepackage{algorithmicx}%
\usepackage{algpseudocode}%
\usepackage{listings}%
\usepackage{caption}
\usepackage{subcaption}
\usepackage{mathtools}
\usepackage{svg}
\usepackage{tabularx}
\usepackage{stmaryrd}

\makeatletter
\newcommand{\settitle}{\@maketitle}
\makeatother

\theoremstyle{thmstyleone}%
\newtheoremrep{theorem}{Theorem}%
\newtheoremrep{lemma}{Lemma}%
\newtheoremrep{proposition}[theorem]{Proposition}%

\theoremstyle{thmstyletwo}%

\theoremstyle{thmstylethree}%

\newcommand{\setto}{\rightrightarrows}
\newcommand{\zer}[1]{\mathrm{zer}\left( #1 \right)}

\newcommand{\R}{\mathbb R}
\newcommand{\Rm}{\mathbb{R}^{m \times m}}
\newcommand{\Sp}{\mathbb{S}_+}
\newcommand{\I}{\mathrm{Id}}
\newcommand{\1}{\mathbbm{1}}
\newcommand{\norm}[1]{\left\lVert#1\right\rVert}

\newcommand{\mk}[1]{\mathbf{#1}}
\newcommand{\diag}{\text{diag}}
\newcommand{\HH}{\mathcal{H}}
\renewcommand{\v}[1]{\mathbf{#1}}
\DeclareMathOperator{\prox}{prox}
\DeclareMathOperator*{\find}{Find}
\DeclareMathOperator{\tr}{tr}
\DeclareMathOperator*{\argmin}{argmin}
\DeclareMathOperator*{\minimize}{minimize}
\DeclarePairedDelimiter\abs{\lvert}{\rvert}%

\DeclarePairedDelimiter\floor{\lfloor}{\rfloor}

\newenvironment{notation*}
  {\par\vspace{\abovedisplayskip}\noindent
   \tabularx{\columnwidth}{>{$}l<{$} @{${}:{}$} >{\raggedright\arraybackslash}X}}
  {\endtabularx\par\vspace{\belowdisplayskip}}

\begin{document}

\title[Optimal Design of Resolvent Splitting Algorithms]{Optimal Design of Resolvent Splitting Algorithms}

\author{\fnm{Robert L} \sur{Bassett}}\email{robert.bassett@nps.edu}
\author{\fnm{Peter} \sur{Barkley}}\email{peter.barkley@nps.edu}

\affil{\orgdiv{Operations Research Department}, \orgname{Naval Postgraduate School}, \orgaddress{\street{1 University Circle}, \city{Monterey}, \postcode{93943}, \state{CA}, \country{U.S.A.}}}

\abstract{In this paper, we introduce a novel semidefinite programming framework for designing custom frugal resolvent splitting algorithms which find a zero in the sum of $n$ monotone operators. This framework features a number of design choices which facilitate creating resolvent splitting algorithms with specific communication structure. We illustrate these design choices using a variety of constraint sets and objective functions, as well as the use of a mixed-integer SDP to minimize time per iteration or required number of communications between the resolvents which define the splitting. Using the Performance Estimation Problem (PEP) framework, we provide parameter selections, such as step size, which for high dimensional problems provide optimal contraction factors in the algorithms. Among the algorithm design choices we introduce, we provide a characterization of algorithm designs which provide minimal convergence time given structural properties of the monotone operators, resolvent computation times, and communication latencies.}

\maketitle

\begin{toappendix}
\begin{center}
{\Large Appendix: Optimal Design of Resolvent Splitting Algorithms}

\vspace{1cm}

{\large Robert L Bassett  and  Peter Barkley}
\vspace{1cm}
\end{center}
\end{toappendix}

\section{Introduction}\label{Sec:Introduction}

\input{introduction.tex}

\section{Preliminaries}\label{Sec:Preliminaries}

\input{formulation_background.tex}

\section{Semidefinite Program for Resolvent Splitting Design}\label{Sec:Formulation}

\input{main_result.tex}

\input{operator_dual.tex}

\subsection{Constructing M from W}\label{Sec:MfromW}

\input{M_from_W.tex}

\subsection{A Graph Theoretic Interpretation}\label{Sec:GraphTheory}

\input{graph_theory.tex}

\section{Sparsity and Parallelism}\label{Sec:Examples}

\input{examples.tex}

\subsection{Maximally Dense and Maximally Sparse Designs}

\input{max_dense_sparse.tex}

\subsection{Parallelization via Block Designs}\label{Sec:block_min}

\input{general_min_cycle.tex}

\section{Objective Functions}\label{Sec:Objectives}

\input{objectives}

\subsection{Spectral Objective Functions}

\input{spectral.tex}

\subsection{Edge-based Objective Functions}\label{Sec:mip_formulations}

\input{edge.tex}

\subsubsection{Minimum Iteration Time}

\input{mi_sdp.tex}

\subsubsection{Mixed Integer Formulation}\label{Sec:mip}

\input{mip.tex}

\section{Convergence Rates and Optimal Tuning}\label{Sec:PEP}

\input{pep}

\subsection{Optimal Step Size}\label{Sec:MPEP}

\input{convergence.tex}

\subsection{Optimal Matrix Selection}\label{Sec:WPEP}

\input{W_PEP_dual.tex}

\section{Numerical Experiments}\label{Sec:Convergence}

\input{experiments.tex}

\input{convergence_rate_results.tex}

\input{compare_mincycle.tex}

\section{Conclusion}

\input{conclusion}

\bibliography{thebib}%

\nosectionappendix
\begin{toappendix}

\input{objective_formulations.tex} 
\end{toappendix}

\end{document}

%% file: introduction.tex
In this paper, we introduce a novel method for designing custom resolvent splitting algorithms to find a zero in the sum of $n$ maximal monotone operators. The algorithms can be described using only vector addition, scalar multiplication, and the resolvents of each operator. The flexible framework we introduce allows the creation of decentralized, distributed resolvent splitting algorithms by customizing the flow of information between the subproblems which compute each resolvent. This permits customized splitting algorithms designed for particular applications where it may be necessary to specify the communication structure of the algorithms to facilitate computational efficiency or respect real-world constraints.

The work presented here builds on recent work in $n$-operator splittings. Motivated by the success of the Alternating Direction Method of Multipliers (ADMM), a proximal splitting method for minimizing the sum of two convex functions, researchers tried to extend the framework to the sum of $n$ convex functions after it was demonstrated that the direct extension of ADMM to 3 convex functions does not converge \cite{chen2016direct}. Despite this initial discouraging result, recent work on resolvent splitting methods for monotone inclusion provides reason for optimism. In \cite{ryu2020uniqueness}, Ryu showed that Douglas-Rachford Splitting, the resolvent splitting method underlying ADMM, is unique among splittings satisfying a set of mild assumptions unless the original problem is \emph{lifted} by embedding it in a higher dimensional space. Working in this lifted space, Ryu provided a resolvent splitting algorithm for monotone inclusion problems which are the sum of three maximal monotone operators. Shortly thereafter, Malitsky and Tam \cite{malitsky2023resolvent} provided a lifted resolvent splitting algorithm for monotone inclusion problems which are the sum of $n$ maximal monotone operators. Tam \cite{tam2023frugal} followed up on this work by generalizing the example provided in \cite{malitsky2023resolvent} to a set of possible resolvent splitting algorithms.

A number of other recent works also analyze frugal resolvent splitting methods. \cite{eckstein2009general} establishes a class of projective splitting methods which includes the option of using resolvent values computed in either the current or the previous iteration. \cite{JCM-37-778}, \cite{campoy2022product}, and \cite{condat2023proximal} all provide product space generalizations of resolvent splitting methods which can accommodate $n$ maximal monotone operators, although these do not include all available communication structures between resolvents. \cite{bredies2022degenerate} generalizes these approaches, analyzing resolvent splitting algorithms in terms of a degenerate (PSD) preconditioner and a lifted operator. Both \cite{bredies2024graph} and \cite{aragon2024forward} build on the preconditioner analysis and recognize the unweighted graph interpretation of these splitting algorithm by providing new designs based on specific graphs.

This paper shares a common $n$-operator resolvent splitting theme with these recent articles. We provide a practical resolvent splitting framework capable of designing algorithms for particular applications. Our primary contribution is the introduction of a family of semidefinite programming problems whose solutions give $n$-operator resolvent splitting algorithms. This family of semidefinite programming problems can be customized for a variety of practical considerations. In practice, desirable characteristics of a resolvent splitting method can depend on many different aspects of a problem---such as communication latency, convergence rate, per-iteration computation, and how amenable the algorithm is to parallelization. We examine each of these in turn, providing theoretical justification and methodological recommendations for design decisions frequently encountered in practice. In addition to providing a framework for creating resolvent splitting algorithms, we provide several theorems which characterize them, extending and generalizing results in \cite{ryu2020uniqueness, malitsky2023resolvent, tam2023frugal}. We also apply the Performance Estimation Problem (PEP) framework from \cite{drori2014performance}, \cite{ryu2020operator}, and \cite{pepit2022} to provide convergence rate guarantees for these algorithms, and apply our optimization framework to design algorithms which minimize the total time required to converge for a given problem class. We then use the dual of the PEP to determine the optimal step size and other parameters for a given algorithm design and set of assumptions on the monotone operators, in addition to conducting a variety of numerical experiments to provide convergence rates for specific problems, compare formulation choices, and provide a demonstration of algorithm design choices. An accompanying implementation of our contributions can found at \href{https://github.com/peterbarkley/oars/}{github.com/peterbarkley/oars/}. 

The rest of this paper proceeds as follows. In the next section, we introduce the required notation and background.
In Section \ref{Sec:Formulation} we formulate the family of semidefinite programming problems which generate custom resolvent splitting algorithms. Sections \ref{Sec:Examples} and \ref{Sec:Objectives} then provide a large set of examples for generating resolvent splitting algorithms under various practical considerations using the constraints and objective function of the SDP, respectively. Section \ref{Sec:PEP} describes a set of PEP formulations and their duals which provides convergence rate results and the ability to optimize step size and a portion of the algorithm for specific problem classes. Section \ref{Sec:Convergence} then conducts a set of numerical experiments, in which we use the previous results to determine rate of convergence under various assumptions, and we compare the resulting splitting algorithms to those proposed by \cite{malitsky2023resolvent} and \cite{tam2023frugal}. Theorems and lemmas not proven in the main document are postponed to the appendix.

%% file: formulation_background.tex
Consider $n$ proper closed convex functions $f_1, \dots , f_{n}$ where each $f_{i}: \HH \to \mathbb{R} \cup \{\infty\}$ for some Hilbert space $\HH$ that contains the problem's decision variable. Our goal is to solve the problem
\begin{equation}\label{min_sum_func}
\min_{x \in \HH} \; \sum_{i=1}^{n} f_{i}(x)
\end{equation}
using a \emph{proximal splitting algorithm}, i.e. only interacting with each of the $f_{i}$ through the proximal operators $\prox_{f_{1}}, \dots, \prox_{f_{n}}$ and linear combinations thereof, where
\begin{equation}\label{Def:Prox}
\prox_{f_i}(x) := \argmin_{w \in \HH} \; f_{i}(w) + \frac{1}{2} \|w - x\|^{2}.
\end{equation}

If the relative interiors of the domain of each $f_{i}$ has nonempty intersection, then finding $x \in \HH$ such that 
\begin{equation}\label{zero_subdiff}
0 \in \sum_{i=1}^n \partial f_{i}(x),
\end{equation}
where $\partial f_{i}$ denotes the subdifferential of the convex function $f_{i}$, is necessary and sufficient for minimizing \eqref{min_sum_func} \cite[Theorems 23.8 and 23.5]{rockafellar1970convex}.

Because the subdifferential of a proper closed convex function $f_{i}$, considered as a set-valued operator $\partial f: \HH \setto \HH$, is a maximal monotone operator, a generalization of \eqref{zero_subdiff} is the monotone inclusion problem: for $n$ maximal monotone operators $A_{1}, \dots, A_{n}$, where each $A_{i}: \HH \setto \HH$, find $x \in \HH$ such that 
\begin{equation}\label{zero_in_monotone}
0 \in \sum_{i=1}^{n} A_{i}(x).
\end{equation}
Throughout the rest of this paper, we assume that the solution set of \eqref{zero_in_monotone}, denoted $\zer{\sum_{i=1}^n A_i}$, is nonempty. Though subdifferentials of convex functions are maximally monotone, maximally monotone operators are only the subdifferential of some convex function if they have additional structure (see e.g. \cite{rockafellar1970maximal}). For this reason, problem \eqref{zero_in_monotone}, which we focus on in the remainder of the paper, contains problems which cannot be stated in terms of convex functions in \eqref{zero_subdiff}, despite the fact that minimizing the sum of convex functions is our primary motivation.

A \emph{resolvent splitting} is a generalization of a proximal splitting which solves \eqref{zero_in_monotone} and is constructed using only scalar multiplication, addition, and the resolvents of each operator $A_{i}$. The resolvent of an operator $A$ is defined as $J_{A} = (\I + A)^{-1}$. When $A$ is monotone, $J_{A}$ is single-valued, and when $A$ is maximal monotone $J_{A}$ has domain $\mathcal{H}$ \cite[Prop. 23.7]{bauschke_combettes}. A resolvent splitting is a \textit{proximal} splitting when the monotone operators in \eqref{zero_in_monotone} are subdifferentials of convex functions, in which case the resolvent of the subdifferential $\partial f_{i}$ is the proximal operator of the convex function $f_{i}$. A resolvent splitting is called \textit{frugal} if each resolvent $J_{A_{i}}$ is used at most once in each iteration. Many existing splitting algorithms can be framed as frugal resolvent splittings. The oldest is Douglas-Rachford splitting \cite{douglas1956numerical, eckstein1992douglas}, where the solution of $0 \in A_1(x) + A_2(x)$ can be found by lifting $x$ to $\v{x}=(x_1, x_2) \in \HH^2$, choosing $\gamma \in (0, 2)$, and iterating 
\begin{subequations}\label{douglas-rachford}
\begin{align}
x_1 &= J_{A_1}(z^k)\\
x_2 &= J_{A_2}(2x_1-z^k)\\
z^{k+1} &= z^k + \gamma (x_2 - x_1).
\end{align}
\end{subequations}
At a fixed point for $z$, $x_1 = x_2 = x^*$ provides the zero for the sum. In \citep{ryu2020uniqueness}, Ryu develops a frugal splitting over three maximal monotone operators, where for $\gamma \in (0, 1)$, the solution for $0 \in A_1(x) + A_2(x) + A_3(x)$ can be found by iterating
\begin{subequations}\label{ryu}
\begin{align}
x_1 &= J_{A_1}(z^k_1)\\
x_2 &= J_{A_2}(x_1 + z^k_2)\\
x_3 &= J_{A_3}(x_1 - z^k_1 + x_2 - z^k_2)\\
z_1^{k+1} &= z_1^k + \gamma (x_3 - x_1)\\
z_2^{k+1} &= z_1^k + \gamma (x_3 - x_2).
\end{align}
\end{subequations}
Similarly, at a fixed point for $\v{z} = (z_{1}, z_{2})$, $x_1 = x_2 = x_3 = x^*$ provides the zero for the monotone inclusion. Ryu's algorithm was extended by Tam \citep{tam2023frugal} to split across $n$ operators, solving $0 \in \sum_i A_i(x)$ by iterating
\begin{subequations}\label{ryu-ext}
\begin{align}
x_i &= J_{A_i}\left(\frac{2}{n-1} \sum_{j=1}^{i-1}x_j + \sqrt{\frac{2}{n-1}}z^k_i\right) \quad &\forall i \in \left\{ 1, \dots, n-1\right\} \label{ryu-ext-begin}\\
x_n &= J_{A_n}\left(\frac{2}{n-1} \sum_{j=1}^{n-1}x_j - \sqrt{\frac{2}{n-1}}\sum_{j=1}^{n-1}z^k_j\right)\\
z^{k+1}_i &= z^k_i + \gamma \sqrt{\frac{2}{n-1}}\left(x_n-x_i\right) \quad &\forall i \in \left\{ 1, \dots, n-1\right\}, \label{ryu-ext-end}
\end{align}
\end{subequations}
where $\gamma$ can take any value in $(0,1)$. In \cite{malitsky2023resolvent}, Malitsky and Tam develop another splitting over $n$ operators given by
\begin{subequations}\label{mt}
\begin{align}
x_1 &= J_{A_1}\left(z^k_1\right) \label{mt-begin}\\
x_i &= J_{A_i}\left(x_{i-1} + z^k_i - z^k_{i-1}\right) \quad &\forall i \in \left\{ 2, \dots, n-1\right\}\\
x_n &= J_{A_n}\left(x_1 + x_{n-1} - z^k_{n-1}\right)\\
z_i^{k+1} &= z_i^k + \gamma \left(x_{i+1} - x_i \right) \quad &\forall i \in \left\{ 1, \dots, n-1\right\}, \label{mt-end}
\end{align}
\end{subequations}
where $\gamma$ is similarly permitted to be any value in $(0,1)$. In both \eqref{ryu-ext} and \eqref{mt}, the $x_i$ are equal at any fixed point for $\v{z}$, and provide the solution to \eqref{zero_in_monotone}.

This recent flurry of $n$-operator resolvent splitting methods is a result of an observation of the authors in \cite{ryu2020uniqueness} (for three operators) and \cite{malitsky2023resolvent} (for $n$ operators), which introduces a parametrization of frugal resolvent splitting methods. These authors' primary use of this parametrization is to prove a lower bound on the dimension of the vector $\v{z}^{k}$, the variable updated in each iteration of the frugal resolvent splitting method. By altering the values used in Ryu's parametrization, subsequent authors were able to quickly generate new frugal resolvent splitting methods. In \cite{tam2023frugal}, Tam characterizes a large set of these parameters which yield convergent frugal splitting algorithms. In this paper, our Theorem \ref{main_theorem} provides a generalization which demonstrates convergence of an even larger set of these parameters.

Because it will be useful in the remainder, we describe this parametrization next. Since the resolvent splittings are assumed to be frugal, each resolvent is evaluated at most once per iteration. Assume without loss of generality that, in each iteration, the ordering of resolvent evaluation is $J_{A_{1}}, J_{A_{2}}, \dots , J_{A_{n}}$, so that when $i < j$ the input to resolvent $J_{A_{i}}$ does not depend on the evaluation of resolvent $J_{A_{j}}$. We let $\v{x} = (x_{1}, \dots x_{n})\in \HH^{n}$ be the concatenated vector of resolvent outputs, and $\v{z} = (z_{1}, \dots, z_{d}) \in \HH^{d}$ be the vector which is updated in each iteration. 
The three ingredients of one iteration of a frugal resolvent splitting are (a) form $y_{i}$, the input to resolvent $A_{i}$, for all $i \in [n]$, (b) evaluate the resolvent $A_{i}$ at $y_{i}$ and (c) update the vector $\v{z}$. At convergence, all $x_i$ will be equal to the solution of \eqref{zero_in_monotone}. Recall that a resolvent splitting can be described using only scalar multiplication, addition, and resolvent evaluation, so these update steps can be written using matrices $B \in \mathbb{R}^{n \times d}$, $T \in \mathbb{R}^{d \times d}$, $R \in \mathbb{R}^{d \times n}$ and lower triangular $L \in \mathbb{R}^{n \times n}$ as follows.
\begin{subequations}\label{scalar_general}
\begin{align}
y_{i} &= \sum_{j=1}^{d} B_{ij}\, z^{k}_{j} + \sum_{j \leq i} L_{ij}\, x_{j} \quad & \forall i \in \{1, \dots, n\} \\
\text{Solve} \quad x_{i} &= J_{A_{i}}(y_{i}) \quad &\forall i \in \{1, \dots, n\}\\
z^{k+1}_{i} &= \sum_{j=1}^{d} T_{ij} \, z^{k}_{j} + \sum_{i=1}^{n} R_{ij} \, x_{j} \quad &\forall i \in \{1, \dots, d\}.
\end{align}
\end{subequations}

The updates \eqref{scalar_general} can be written more concisely using matrices constructed via the Kronecker product, which we denote by $\otimes$. Let $\I$ denote the identity operator on $\mathcal{H}$, and define $\v{B} = B \otimes \I$, $\v{L} = L \otimes \I$, $\v{T} = T \otimes \I$, and $\v{R} = R \otimes \I$. Denote by $\v{A}:\HH^{n} \setto \HH^{n}$ the maximal monotone operator formed by applying the monotone operators $A_{i}$ elementwise, so
$$\v{A}\begin{pmatrix} x_{1} \\ \vdots\\ x_{n}\end{pmatrix} = \begin{pmatrix}A_{1}(x_{1}) \\ \vdots\\ A_{n}(x_{n}) \end{pmatrix},$$
where each $x_{i} \in \HH$. Then, eliminating the $y_{i}$ variables for conciseness, the updates \eqref{scalar_general} can be written in terms of the concatenated vectors $\v{z}$ and $\v{x}$ as follows.
\begin{subequations}\label{matrix_general}
\begin{align}
\v{x} &= J_{\v{A}}(\v{B} \v{z}^{k} + \v{L} \v{x}) \label{general_form1}\\
\v{z}^{k+1} &= \v{T} \v{z}^{k} + \v{R} \v{x}\label{general_form2}.
\end{align}
\end{subequations}

In this paper, we focus on a restriction of the very general parametrization in \eqref{matrix_general} that allows us to prove convergence under a set of mild conditions. Specifying the matrices $L$, $M$, and the value $\gamma$, we let $\v{R} = \gamma \v{M}$, $\v{B} = -\v{M}^{T}$, and $\v{T}$ be the identity. The frugal resolvent splitting algorithm in \eqref{matrix_general} then becomes
\begin{subequations}\label{d_iteration}
\begin{align}
\mk{x} &= J_{\v{A}}\left(-\mk{M}^{T} \mk{z}^{k} + \mk{L} \mk{x}\right) \label{d_itr1}\\
\mk{z}^{k+1} &= \mk{z}^{k} + \gamma \mk{M} \mk{x}. \label{d_itr2}
\end{align}
\end{subequations}
Each of the aforementioned resolvent splitting algorithms in \eqref{douglas-rachford}, \eqref{ryu}, \eqref{ryu-ext}, and \eqref{mt} can be written in the form \eqref{d_iteration} with appropriate choice of $L$, $M$, and range of $\gamma$. 

In \citep{tam2023frugal}, the author notes that when $d > n$, the dimension of the variable $\v{z}$ in the iteration \eqref{d_iteration} can be reduced by performing the substitution $\v{v} = -\v{M}^{T} \v{z}$, defining $\v{W} = \v{M}^{T} \v{M}$, and multiplying the equation \eqref{d_itr2} by $-\v{M}^{T}$ to obtain the iteration
\begin{subequations}\label{n_iteration}
\begin{align}
\mk{x} &= J_{\mk{A}}\left(\mk{v}^{k} + \mk{L} \mk{x}\right)\label{n_itr1}\\
\mk{v}^{k+1} &= \mk{v}^{k} - \gamma \v{W} \mk{x}.\label{n_itr2}
\end{align}
\end{subequations}
In \eqref{n_iteration}, the variable $\v{v} \in \HH^{n}$ is updated each iteration, which requires less memory than $\v{z} \in \HH^{d}$ when $d > n$. It is worth noting that this substitution implies that $\v{v}^0$ must be in the range of $\v{M}^T$.

In the remainder of the paper, we will derive convergent algorithms for both iterations \eqref{d_iteration} and \eqref{n_iteration} under a variety of practical assumptions. We provide a framework to design new resolvent splitting algorithms that accommodate decentralized computation. We also investigate the performance of these algorithms by providing theoretical results and practical recommendations to guide the generation of these algorithms while prioritizing fast convergence rates, amenability to parallelization, and minimal per-iteration computation time.

Before proceeding, we establish some notation that we will be useful in the rest of the paper. $\Sp^n$ will be used to represent the set of positive semi-definite matrices of dimension $n$ and $\mathbb{S}^{n}$ the set of symmetric matrices of dimension $n$. $\succeq$ provides the Loewner ordering of the set of symmetric matrices, so $Z \succeq W  \implies Z-W \in \Sp^n$. $\I$ is the identity in the appropriate space in which it is used. Given a matrix $K \in \mathbb{S}^{n}$, $\lambda_{1}(K), \dots, \lambda_{n}(K)$ denotes the eigenvalues of $K$ listed in increasing order. $\1$ is a ones vector in the appropriate space. All bold vectors and matrices are lifted, so bolded vectors contain multiple copies of $\HH$, and bolded operators operate on these lifted vectors. For example, $\v{M}: \HH^n \to \HH^d$ is $\v{M} = M \otimes \I$ for $M \in \R^{d \times n}$. We also define the transpose of a lifted operator as a lifting of the transposed matrix, so $\v{M}^T = M^T \otimes \I$. For a natural number $n$, we denote by $[n]$ the set of natural numbers $\{1, \dots , n\}$. We denote by $\mathcal{S}^{C}$ the complement of a set $\mathcal{S}$. $\text{Tr}(\cdot)$ defines the trace of a matrix. $\norm{\cdot}$ denotes the Euclidean norm when applied to vectors and the spectral norm when applied to matrices. $e_i$ refers to the unit vector with 1 in position $i$, and $L_{\cdot i}$ and $L_{i \cdot}$ refer to the $i$th column and row of $L$ respectively. We denote by $\iota_{C}$ the indicator function on a set $C$ and for a set-valued operator $A$ we denote by $A^{-\ovee}$ the operator with $A^{-\ovee}(x) = - A^{-1} (-x)$.

%% file: main_result.tex
In this section we formulate an optimization problem that allows one to solve for the matrices $\v{M}$ and $\v{L}$ in \eqref{d_iteration}, or the matrices $\v{W}$ and $\v{L}$ if using iteration \eqref{n_iteration}. We will call a specific choice of the parameters in algorithms \eqref{d_iteration} or \eqref{n_iteration} a \emph{design} for either of these algorithms. By considering various objective functions and additional constraints, we show that we can recover several of the proximal splitting algorithms from the previous section. Our next theorem is the primary contribution of this paper. 

\input{proof_z.tex}
Theorem \ref{main_theorem}, which proves algorithmic convergence for the zero of any sum of $n$ monotone operators, provides particular insight in the case of $\mu$-strongly monotone operators. When all operators are $\mu$-strongly monotone (as is the case if they are the subdifferentials of $\mu$-strongly convex functions), the permissible range of $\gamma$ extends beyond one (the bound established in \cite{malitsky2023resolvent,tam2023frugal}). In section \ref{Sec:PEP}, where we provide a technique for selecting a $\gamma$ which yields the best contraction factor for certain classes of monotone operators, we observe that $\gamma$ should often be chosen to be greater than or equal to one. We next discuss the implications of Theorem \ref{main_theorem}, postponing its proof to the appendix.

The first aspect of problem \eqref{main_prob} worth addressing is its computational difficulty. When the set $\mathcal{C}$ and the function $\phi$ are convex, problem \eqref{main_prob} is convex. Excepting the general purpose objective $\phi$ and constraint set $\mathcal{C}$, the problem can be written using only affine and semidefinite cone constraints. Constraint \eqref{con2} is perhaps the only constraint that is not immediately recognizable as convex, but we recall that the sum of the $k$ smallest eigenvalues of a PSD matrix is a concave function. Constraint \eqref{con2} can be written using the semidefinite cone as by introducing auxiliary variables $s \geq 0$ and $Y \in \mathbb{S}_{+}^{n}$ and using the constraints
\begin{align}
W + Y - s \mathrm{I} &\succeq 0\\
2 s - \mathrm{Tr}(Y) & \geq c,
\end{align}
see e.g. \cite{ben2001lectures}. Therefore, given appropriate choice of objective function $\phi$ and additional constraint set $\mathcal{C}$, problem \eqref{main_prob} is a semidefinite program which can be solved at scale using a variety of possible algorithms \cite{vandenberghe1996semidefinite, yurtsever2021scalable}.

One crucial aspect of problem \eqref{main_prob} is that the combination of constraints \eqref{con1} and \eqref{con2} combine to define the null space of $W$ exactly. Because $W \1 = 0$, we have $\mathrm{span}(\1) \subseteq \mathrm{Null}(W)$ and $\lambda_{1} = 0$. Constraint \eqref{con2} then gives that $\lambda_{2} > 0$, so that the dimension of the eigenspace with eigenvalue $0$ is $1$ and $\mathrm{span}(\1) = \mathrm{Null}(W)$. Constructing $W$ so that $\mathrm{Null}(W) = \mathrm{span}(\mathbbm{1})$ is an important part of the proof of Theorem \ref{main_theorem}. Constraint \eqref{con3} implies $Z \succeq 0$. Constraint \eqref{con4} ensures $\1^T L \1 = n$. Constraints \eqref{con1}-\eqref{con4} taken together imply that $\mathrm{Null}(Z) = \mathrm{span}(\1)$. Constraints \eqref{con5}-\eqref{con6} give that $Z$ has a constant diagonal with value strictly between 0 and 4.

Utilizing the matrices generated by solving \eqref{main_prob} in the iterations \eqref{d_iteration} or \eqref{n_iteration} requires solving a fixed point equation in $\v{x}$ in each iteration, equations \eqref{d_itr1} and \eqref{n_itr1}, respectively. The lower triangular nature of $L$ is a critical part of this computation, and makes solving for $\v{x}$ straightforward. When $L$ is \emph{strictly} lower triangular, $\v{x}$ can be easily found by sequentially computing $\v{x}_{i} = J_{A_{i}}(\v{v}^{k}_{i} + \sum_{j < i} L_{ij} \v{x}_{j})$ for each $i \in [n]$. When $L$ is lower triangular but not strictly so and $L_{ii} \ne 1$, the elementwise fixed point equation can be written
\begin{align}\label{standard_n_iter}
\v{x}_{i} &= J_{A_{i}}\left(\v{v}^{k}_{i} + \sum_{j < i} L_{ij} \v{x}_{j} + L_{ii} \v{x}_{i}\right).
\end{align}
Manipulating this expressing using the definition of the resolvent, this fixed point equation is satisfied if and only if
\begin{align}
\v{x}_{i} = J_{\frac{1}{1 - L_{ii}} A_{i}}\left(\frac{1}{1-L_{ii}} \v{v}_{i}^{k} + \sum_{j < i} \frac{L_{ij}}{1 - L_{ii}} \v{x}_{j}\right).
\end{align}
Since \eqref{main_prob} requires $L_{ii}$ to be constant and less than 1, $\sum \frac{1}{1-L_{ii}}A_i(x) = 0 \implies \sum A_i(x) = 0$. In this way, the fixed point equation can again be solved directly as in the strictly lower triangular case, with the exception that the scaled resolvent must now be evaluated on a scaled version of the original input. This is equivalent to executing \eqref{standard_n_iter} with $\tilde{\v{A}} = \frac{1}{1 - L_{ii}}\v{A}$, $\tilde{W}= \frac{1}{1 - L_{ii}}W$ and $\tilde{L}= \frac{1}{1 - L_{ii}}(L-L_{ii}\I)$. Therefore when $L$ is lower triangular without being strictly so, the updates \eqref{d_itr1} and \eqref{n_itr1} can still be easily computed.

%% file: proof_z.tex
\begin{theoremrep}\label{main_theorem}

Let $\phi: \mathbb{S}_{+}^{n} \times \mathbb{S}^{n} \to (-\infty, \infty]$ be any proper lower semicontinuous function. Let $\gamma \in (0,1)$, $\mathcal{C} \subseteq \mathbb{S}_{+}^{n} \times \mathbb{S}^{n}$, $c > 0$, and $\varepsilon \in [0, 2)$. Consider the following semidefinite programming problem,
\begin{subequations}\label{main_prob}
\begin{align}
\minimize_{W, Z} \quad &\phi(W, Z) \label{obj}\\
\mathrm{subject\, to} \quad& W \1 = 0 \label{con1}\\
& \lambda_{1}(W) + \lambda_{2}(W) \geq c \label{con2}\\
& Z - W \succeq 0 \label{con3}\\
& \1^{T} Z \1 = 0  \label{con4}\\
&\mathrm{diag}(Z) = Z_{11}\1 \label{con5}\\
& 2 -\varepsilon \leq Z_{11} \leq 2 + \varepsilon \label{con6}\\
& (W, Z) \in \mathcal{C} \label{con7}\\
& W \in \mathbb{S}^{n}_{+}, \; Z \in \mathbb{S}^{n}. \label{con8}
\end{align}
\end{subequations}
 
Any solution $W, Z$ to \eqref{main_prob} produces a convergent resolvent splitting algorithm for the iteration \eqref{n_iteration} by solving for the lower triangular matrix $L$ in $Z = 2 \I - L - L^{T}$. 
This $L$, paired with any $M$ for which $W = M^{T} M$, also produces a convergent resolvent splitting algorithm for the iteration \eqref{d_iteration}. In both iterations, $\v{x} \in \HH^{n}$ converges weakly to a vector for which $x_{i} = x^*$ for all $i \in [n]$, where $x^* \in \HH$ solves the monotone inclusion \eqref{zero_in_monotone}. Moreover, when the operators $A_{i}$ in \eqref{zero_in_monotone} are all $\mu$-strongly monotone for some $\mu > 0$, the valid range of $\gamma$ can be extended to $(0, 1 + 2\mu/\|W\|)$.
\end{theoremrep}
\begin{proof}

The proof proceeds as follows: we first demonstrate that the operator $T_{\v{A}}(\v{z}) = \v{z} + \gamma \v{M} \v{x}$ where $\v{x} = J_{\v{A}}\left(-\v{M}^T \v{z} + \v{L}\v{x}\right)$ is $\gamma$-averaged non-expansive for $\mu$-strong maximal monotone operator $\v{A} = \left(A_1(x_1), \dots, A_n(x_n)\right)$.  We then show that the existence of a zero for $\sum_{i=1}^n A_i$ is equivalent to the existence of a fixed point of $T_{\v{A}}$, and therefore by the $\gamma$-averaged nonexpansivity of $T_{\v{A}}$ we have weak convergence of $(\v{z}^k)$ to a fixed point of $T_{\v{A}}$. Finally, we then show that iterates $(\v{x}^k)$ converge to a unique weak cluster point, which is $\1 \otimes x^*$, where $0 \in \sum_{i=1}^n A_i(x^*)$.

Let $W$, $Z$, and $L$ be feasible solutions to \eqref{main_prob}, and $M \in \R^{d \times n}$ such that $M^T M = W$. Let $\v{M}, \v{W}$, $\v{Z}$, $\v{L}$ be the lifted Kronecker products of $M$, $W$, $Z$, and $L$, respectively. 
We include the $\mu$-strong monotonocity case directly in the main proof by allowing $\mu \geq 0$, though $\mu = 0$ is not included in conventional definitions of strong monotonicity.
For $\mu \geq 0$, we therefore define maximal $\mu$-strongly monotone $\v{A}$ as
$$\v{A}\begin{pmatrix} x_{1} \\ \vdots\\ x_{n}\end{pmatrix} = \begin{pmatrix}A_{1}(x_{1}) \\ \vdots\\ A_{n}(x_{n}) \end{pmatrix},$$
where each $x_i \in \HH$.
For $\v{z} \in \HH^d$, define $T_{\v{A}}: \HH^d \to \HH^d$ as 
$$T_{\v{A}}(\v{z}) = \v{z} + \gamma \v{M} \v{x} , \quad \text{where} \quad \v{x} = J_{\v{A}}\left(-\v{M}^T \v{z} + \v{L} \v{x}\right)$$
By the definition of the resolvent and the maximal monotonicity of $\v{A}$, we know that for $\v{z}^1, \v{z}^2 \in \HH^d$ and $\v{x}^1, \v{x}^2 \in \HH^n$, 
\begin{align}
    &\v{x}^i = J_{\v{A}}(-\v{M}^T \v{z}^i + \v{L} \v{x}^i)\\
    \implies& \v{A}\left(\v{x}^i\right) \ni -\v{M}^T \v{z}^i + (\v{L} - \I) \v{x}^i.
\end{align}
Let $\v{z} = \v{z}^1 - \v{z}^2$, $\v{x} = \v{x}^1 - \v{x}^2$, and $\v{z}^+ = T_{\v{A}}(\v{z}^1)-T_{\v{z}}(\v{z}^2)$.
By the $\mu$-strong monotonicity of $\v{A}$ we have:
\begin{align} 
    \left\langle \v{x}^1 - \v{x}^2, -\v{M}^T \v{z}^1 + (\v{L} - \I) \v{x}^1 - \left(-\v{M}^T \v{z}^2 + (\v{L} - \I) \v{x}^2\right) \right\rangle &\geq \mu \norm{\v{x}^1 - \v{x}^2}^2\\
\implies \left\langle \v{x}, -\v{M}^T \v{z} + (\v{L} - \I) \v{x} \right\rangle & \geq \mu \norm{\v{x}}^2\label{first}
\end{align}
Considering just the left-hand side of the inequality \eqref{first}, symmetrizing the quadratic form $\left\langle \v{x}, (\v{L} - \I) \v{x} \right\rangle$ in light of the definition of $L$, and noting that $W \preceq Z$ by \eqref{con3}, we have the following simplifications,
\begin{align}
    =& \left\langle \v{x}, -\v{M}^T \v{z} \right\rangle +  \left\langle \v{x}, (\v{L} - \I) \v{x} \right\rangle\\
    =& \left\langle \v{x}, -\v{M}^T \v{z} \right\rangle - \frac{1}{2}\left\langle \v{x}, \v{Z} \v{x} \right\rangle\\
    \leq& \left\langle \v{x}, -\v{M}^T \v{z} \right\rangle - \frac{1}{2}\left\langle \v{x}, \v{W} \v{x} \right\rangle
\end{align}
Considering the right-hand side of \eqref{first}, we note that $\norm{\v{x}}^2 \geq \frac{1}{\norm{\v{W}}}\left \langle \v{x}, \v{W} \v{x} \right \rangle$, where $\|\v{W}\|$ is the operator norm in $\mathcal{H}^{n}$. Noting that $\|\v{W}\| = \|W\|$, we have 
\begin{align}
\left\langle \v{x}, -\v{M}^T \v{z} \right\rangle - \frac{1}{2}\left\langle \v{x}, \v{W} \v{x} \right\rangle \geq& \frac{\mu}{\norm{W}}\left \langle \v{x}, \v{W} \v{x} \right \rangle\\
\left\langle \v{x}, -\v{M}^T \v{z} - \left(\frac{1}{2}+\frac{\mu}{\norm{W}}\right) \v{W} \v{x} \right\rangle \geq& 0\\
\left\langle -\v{M}\v{x}, \v{z}\right\rangle - \left(\frac{1}{2}+\frac{\mu}{\norm{W}}\right)\left\langle \v{M}\v{x}, \v{M} \v{x} \right\rangle \geq& 0\label{second}\\
\end{align}
By definition of $\v{z}^{+}$, $\v{M}\v{x} = \frac{\v{z}^+ -\v{z}}{\gamma}$. Therefore, \eqref{second} implies
$$
    \frac{1}{\gamma}\left\langle \v{z} - \v{z}^+, \v{z} \right\rangle - \left(\frac{1}{2}+\frac{\mu}{\norm{W}}\right)\frac{1}{\gamma^2}\norm{\v{z}^+ -\v{z}}^2\geq 0\label{third}
$$
Applying the parallelogram law to the left side yields
\begin{align}
    \frac{1}{2\gamma}\left[\norm{\v{z}}^2 + \norm{\v{z} - \v{z}^+}^2 - \norm{\v{z}^+}^2\right]- \left(\frac{1}{2}+\frac{\mu}{\norm{W}}\right)\frac{1}{\gamma^2}\norm{\v{z} - \v{z}^+}^2 \geq& 0\\
    \frac{1}{2\gamma}\left[\norm{\v{z}}^2 + \frac{\gamma - 1 - \frac{2\mu}{\norm{W}}}{\gamma}\norm{\v{z} - \v{z}^+}^2 - \norm{\v{z}^+}^2\right] \geq& 0
\end{align}
and we therefore have:
$$\norm{\v{z}}^2 + \frac{\gamma - 1 - \frac{2\mu}{\norm{W}}}{\gamma}\norm{\v{z} - \v{z}^+}^2 \geq \norm{\v{z}^+}^2$$ and $T_{\v{A}}$ is $\gamma$-averaged for $\gamma \in (0, 1 + \frac{2\mu}{\norm{W}})$.
We now show 
$$\zer{\sum_{i=1}^{n} A_{i}} \ne \emptyset \implies \mathrm{Fix}(T_{\v{A}}) \ne \emptyset$$
If $\bar{x} \in \mathrm{zero}(\sum_{i=1}^{n} A_{i})$, then there exists a $\v{w} \in \HH^n$ such that $w_i \in A_i(\bar{x})$ and $\sum_{i=1}^n w_i = 0$.
Let $\overline{\v{x}} = \1 \otimes \bar{x}$. 

Since $M^{T} M = W$ and $\mathrm{Null}(W) = \mathrm{span}(\mathbbm{1})$, $\mathrm{Null}(M) = \mathrm{span}(\mathbbm{1})$. It follows that $\mathrm{range}(\v{M}^{T}) = \{\v{x} \in \mathcal{H}^{n}| \sum_{i=1}^{n} x_{i} = 0\}$. Note that $\v{w} + (\I - \v{L}) \overline{\v{x}} \in \mathrm{range}(\v{M}^{T})$, since both terms sum to $0$. Therefore there exists $\overline{\v{z}}$ such that $-\v{M}^{T} \overline{\v{z}} = \v{w} + (\I - \v{L}) \overline{\v{x}}$. Recalling that $\v{w} \in \v{A}(\overline{\v{x}})$, for such a $\v{\overline{z}}$ we have
\begin{align}
-\v{M}^T\overline{\v{z}} + \v{L}\overline{\v{x}} \in& \v{A}(\overline{\v{x}})+\overline{\v{x}}\\
\implies \overline{\v{x}} =& J_{\v{A}}(-\v{M}^T\overline{\v{z}} + \v{L}\overline{\v{x}})
\end{align}
Finally, we note that, because $\overline{\v{x}} = \1 \otimes \bar{x}$, $\overline{\v{x}} \in \mathrm{Null}(\v{M})$, so $\overline{\v{z}} = T_{\v{A}}(\overline{\v{z}})$ and $\overline{\v{z}} \in \mathrm{Fix}(T_{\v{A}})$. 

Since $T_{\v{A}}$ has a fixed point and is $\gamma$-averaged nonexpansive, \cite[Proposition 5.15]{bauschke_combettes} gives that for any starting point $\v{z}^0$ and sequence $(\v{z}^k)$ defined by $\v{z}^{k+1} = T_{\v{A}}(\v{z}^k)$, $\v{z}^{k+1} - \v{z}^{k}\to 0$, and $(\v{z}^k)$ converges weakly to some $\bar{\v{z}} \in \mathrm{Fix}(T_{\v{A}})$. Recall that weak convergence implies boundedness \cite[Proposition 2.40]{bauschke_combettes}, so that $(\v{z}^{k})$ is bounded. 

We next show by induction on $n$ that $(\v{x}^{k}) \in \mathcal{H}^{n}$ is bounded on each of its components $(x_{i}^{k}) \in \mathcal{H}$, so that the entire expression is bounded. Take as the base case $i = 1$. Then recalling that $L$ is lower triangular and that the resolvent is nonexpansive, we have
\begin{align}
\|x_{1}^{k}\| = \left\|J_{A_{1}} \left(- \sum_{j=1}^{d} M_{j1} z_{j}^{k} + L_{11} x_{1}^{k}\right)\right\| &\leq \left\|-\sum_{j=1}^{d} M_{j1} z_{j}^{k} + L_{11} x_{1}^{k}\right\| \\
&\leq \left\|-\sum_{j=1}^{d} M_{j1} z_{j}^{k} \right\| + |L_{11}| \|x_{1}^{k}\|.
\end{align}
Therefore
$$(1 - |L_{11}|) \|x_{1}^{k}\| \leq \left\|-\sum_{j=1}^{d} M_{j1} z_{j}^{k} \right\|$$
so when $|L_{11}| < 1$ we have
$$\|x_{1}^{k}\| \leq \frac{1}{1 - |L_{11}|} \left\|-\sum_{j=1}^{d} M_{j1} z_{j}^{k} \right\|.$$
which is bounded because $(\v{z}^{k})$ is bounded. This concludes the base case.

In the induction step, 
\begin{align}
\|x_{i}^{k}\| &= \left\|J_{A_{i}} \left(-\sum_{j=1}^{d} M_{ji} z_{j}^{k} + \sum_{j<i} L_{ij} x_{j}^{k} + L_{ii} x_{i}^{k}\right)\right\| \\
&\leq \left\|-\sum_{j=1}^{d} M_{ji} z_{j}^{k} + \sum_{j<i} L_{ij} x_{j}^{k} + L_{ii} x_{i}^{k}\right\|\\
& \leq \left\|-\sum_{j=1}^{d} M_{j1-i} z_{j}^{k} \right\| + \left\| \sum_{j<i} L_{ij} x_{j}^{k} \right\| + |L_{ii}| \left\|x_{i}^{k}\right\|
\end{align}
where $\sum_{j<i} L_{ij} x_{j}^{k}$ is bounded by the induction hypothesis. Similar to the base case, the fact that $|L_{ii}| < 1$ allows us to conclude that $x_{i}^{k}$ is bounded. Since $||\v{x}^{k}||^{2} = \sum_{i=1}^{n} ||x_{i}^{k}||^{2}$, each of which are bounded, we conclude that $(\v{x}^{k})$ is bounded. The boundedness of $(\v{x}^k)$ implies the existence of a weak sequential cluster point $\v{\tilde{x}}$ for $(\v{x}^k)$ \cite[Fact 2.27]{bauschke_combettes}. Abusing notation, let $(\v{x}^{k})$ be a subsequence that weakly converges to $\v{\tilde{x}}$.

We next reason that the cluster point $\v{\tilde{x}}$ is unique, which gives that $(\v{x}^{k}) \rightharpoonup \v{\tilde{x}}$. To do so we require two facts: that $\v{\tilde{x}}$ is of the form $\1 \otimes \tilde{x}$ for some $\tilde{x} \in \mathcal{H}$ and that $\v{\tilde{x}} = J_{\v{A}}(-\v{M}^{T} \bar{\v{z}} + \v{L} \v{\tilde{x}})$. The first of these is easy to establish---because $\v{z}^{k+1} - \v{z}^{k} \to 0$ and $\v{z}^{k+1} = \v{z}^{k} + \gamma \v{M} \v{x}^{k}$, we know $\v{M} \v{\tilde{x}} = \lim_{k \to \infty} \v{M} \v{x}^{k} = 0$. Since $\v{\tilde{x}} \in \mathrm{Null}(\v{M})$ all of its components are equal, so $\v{\tilde{x}} = \1 \otimes \tilde{x}$ for some $\tilde{x}$.

The second required fact, that 
\begin{equation}\label{resolvent_fixedpoint}
\v{\tilde{x}} = J_{\v{A}}(-\v{M}^{T} \bar{\v{z}} + \v{L} \v{\tilde{x}}),
\end{equation}
 requires more effort. It suffices to show that $-\v{M}^{T} \bar{\v{z}} + (\v{L} - \I) \v{\tilde{x}} \in \v{A}(\v{\tilde{x}})$. So, letting $\v{v} \in \v{A}(\v{u})$ for $\v{u} \in \mathrm{Dom}(\v{A})$, we want to show that
\begin{equation}
\left\langle \v{\tilde{x}} - \v{u}, -\v{M}^{T} \bar{\v{z}} + (\v{L} - \I)(\v{\tilde{x}}) - \v{v} \right\rangle \geq 0.
\end{equation}
This can be shown by applying limiting arguments to the expression
\begin{equation} \label{limit}
\left\langle \v{x}^{k} - \v{u}, -\v{M}^{T} \v{z}^{k} + (\v{L} - \I)(\v{x}^{k}) - \v{v} \right\rangle \geq 0
\end{equation}
using that $\v{M} \v{x}^{k} \to 0$, $\v{z}^{k} \rightharpoonup \bar{\v{z}}$ and $\v{x}^{k} \rightharpoonup \v{\tilde{x}}$. The only complicated term in \eqref{limit} is $\langle \v{x}^{k}, (\v{L} - \I) \v{x}^{k} \rangle$, which can be shown to converge to $0$ using an $\epsilon/3$ argument by adding and subtracting $\1 \otimes x^{k}_{1}$, the first component of $\v{x}^{k}$ lifted into $\mathcal{H}^{n}$.

Having established that $\v{\tilde{x}} = \1 \otimes \tilde{x}$ and $\v{\tilde{x}} = J_{\v{A}}(-\v{M}^{T} \bar{\v{z}} + \v{L} \v{\tilde{x}})$, we now reason that $\tilde{x}$ must be unique. Consider the first component of the equation \eqref{resolvent_fixedpoint}, 
\begin{equation}
\tilde{x}_{1} = J_{A_{1}}\left(-(\v{M}^{T} \bar{\v{z}})_{1} + L_{11} \tilde{x}_{1}\right).
\end{equation}
Expanding the definition of the resolvent and simplifying, this implies
\begin{equation}\label{expanded_resolvent}
\frac{-(\v{M}^{T} \bar{\v{z}})_{1}}{1-L_{11}} - \tilde{x}_{1} \in \frac{1}{1-L_{11}}A_{1}(\tilde{x}_{1}).
\end{equation}
Because $\abs{L_{11}} < 1$, $\frac{1}{1-L_{11}}A_{1}$ is maximal monotone. So $\tilde{x}_{1} = J_{\frac{1}{1-L_{11}}A_{1}}\left(\frac{-1}{1-L_{11}}(\v{M}^{T} \bar{\v{z}})_{1}\right)$ is therefore unique. The fact that all components of $\v{\tilde{x}}$ are equal then gives that $\v{\tilde{x}}$ is unique. Since $\v{x}^{k}$ has a unique weak cluster point $\v{\tilde{x}}$, it weakly converges to that point. Summing across the components of the containment $-\v{M}^{T} \bar{\v{z}} + (\v{L} - \I) \v{\tilde{x}} \in \v{A}(\v{\tilde{x}})$ yields $0 \in \sum_{i=1}^{n} A_{i}(\tilde{x})$, so $\v{\tilde{x}}$ solves \eqref{zero_in_monotone}. This proves the claimed property of \eqref{d_iteration}.

Finally, since iteration \eqref{n_iteration} corresponds directly to \eqref{d_iteration} with change of variable $\v{v} = -\v{M}^T \v{z}$, the weak convergence of $(\v{z}^k)$ and $(\v{x}^k)$ also implies the convergence of \eqref{n_iteration}. 
\end{proof}

%% file: operator_dual.tex
In addition to the convergence of the $\v{x}$ iterates in \eqref{n_iteration}, the $\v{v}$ iterates can be used to compute the Attouch-Th\'era dual solution of a lifted version of the problem \eqref{zero_in_monotone}. In the following theorem, we denote by $\Delta$ the subspace
$$\Delta = \{\v{x} \in \mathcal{H}^{n}|\;\v{x} = \1 \otimes x \text{ for some } x \in \mathcal{H}\}.$$
\begin{theoremrep}\label{attouch-thera}
Let $\v{v}^{*}$ and $\v{x}^{*}$ be limits of the algorithm \eqref{n_iteration}. Define $\v{u}^{*} = \v{v}^{*} + (\v{L} - \I) \v{x}^{*}$. Then $\v{u}^{*}$ is the solution to the Attouch-Th\'era dual for the problem
\begin{equation} \label{at_primal}
 \find_{\v{x} \in \HH^n} 0 \in \left(\v{A} + \partial \iota_\Delta\right)\v{x},
\end{equation}
which is,
\begin{equation} \label{at_dual}
\find_{\v{u} \in \HH^n} 0 \in \left(\v{A}^{-1} + \left(\partial \iota_\Delta\right)^{-\ovee} \right)\v{u}.
\end{equation}
\end{theoremrep}
\begin{proof}

To show that $\v{u}^{*} = \v{v}^{*} - (\I - \v{L}) \v{x}^{*}$ solves \eqref{at_dual}, we need to show that
\begin{equation}\label{at_to_show}
0 \in \left(\v{A}^{-1} + \left(\partial \iota_\Delta\right)^{-\ovee} \right)\v{u}^*.
\end{equation}
Phrased differently, we need to show that there are $\v{y} \in \v{A}^{-1}(\v{u}^{*})$ and $\v{s} \in \partial \iota_{\Delta}^{-\ovee}$ such that $0 = \v{y} + \v{s}$. The subdifferential $\partial \iota_\Delta$ is the set $\Delta^{\perp}$ since the subdifferential of an indicator function is the normal cone to the set defining the indicator function. Since $\Delta$ is a linear subspace, its normal cone is $\Delta^{\perp}$. Therefore $\left(\partial \iota_\Delta\right)^{-\ovee}$, which is $\left(\partial \iota_\Delta\right)^{-1}$ in this setting, has domain $\Delta^{\perp}$ and returns $\Delta$ for every input in $\Delta^{\perp}$. We next show that $\v{x}^{*} \in \v{A}^{-1}(\v{u}^{*})$ and $\v{x}^{*} \in \Delta$, so that \eqref{at_to_show} is solved by taking $\v{y} = \v{x}^{*}$ and $\v{s} = -\v{x}^{*}$.

From Theorem \ref{main_theorem}, we have $\v{x}^{*} = \1 \otimes x^{*}$ for some $x^{*} \in \mathcal{H}$, so $\v{x}^{*} \in \Delta$ and $-\v{x}^{*} \in \Delta$. We also have $\v{v}^{*} \perp \Delta$ since $\v{v}^{*} \in \mathrm{range}(\v{W})$. Theorem \ref{main_theorem} also points out that $\1^T (\v{L} - \I) \v{x}^{*} = 0$, so $(\v{L} - \I) \v{x}^{*} \in \Delta^\perp$. So $\v{u}^* = \v{v}^{*} + (\v{L} - \I) \v{x}^{*} \in \Delta^\perp$, and $-\v{u}^* \in \Delta^\perp$. Therefore $-\v{x}^* \in \left(\partial \iota_\Delta\right)^{-\ovee}(\v{u}^*)$. From \eqref{n_itr1}, we have $\v{x}^{*} = J_{\v{A}}\left(\v{v}^{*} + \v{L} \v{x}^{*}\right)$, so that $\v{u}^{*} = \v{v}^{*} + (\v{L} - \I) \v{x}^{*} \in \v{A}(\v{x}^{*})$ and $\v{x}^* \in \v{A}^{-1}(\v{u}^*)$. This demonstrates the required properties of $\v{x}^{*}$ and $\v{u}^*$ and proves the result.
\end{proof}
The duality result in Theorem \ref{attouch-thera} is especially useful in the context of warmstarting.
With any prior knowledge of values $\v{x}$ or $\v{u}$ close to $\v{x}^*$ or $\v{u}^*$, one can warm start the algorithm \eqref{n_iteration} with $\v{v}^0 = \v{u} + (\I - \v{L})\v{x}$.
If $A_i = \partial f_i$ for some set of closed, convex, and proper functions $f_i$, a similar result holds for Fenchel and Lagrangian duality in iteration \eqref{d_iteration}. In this case, one can show, for $M \in \R^{n-1 \times n}$ and $\v{u}^* \in \HH^{n-1}$, that $\v{u}^* = \v{z}^* - (\v{M}^T)^\dagger (\v{L}- \v{\I}) \v{x}^*$ as the dual solution of the Lagrangian $g(\v{x}, \v{u}) = \sum_{i=1}^n f_i(x_i) + \left\langle \v{u} , \v{M}\v{x}\right\rangle$, where $(\v{M}^{T})^{\dagger}$ denotes the lifted Moore-Penrose pseudoinverse of $M^{T}$.

%% file: M_from_W.tex
The iteration \eqref{d_iteration} relies on the construction of a matrix $M \in \mathbb{R}^{d \times n}$ from $W \in \mathbb{R}^{n \times n}$ such that $M^{T} M = W$. Theorem \ref{main_theorem} shows that any such $M$ produces an iteration \eqref{d_iteration} that converges to the solution of \eqref{zero_in_monotone}. How should one construct the matrix $M$? Different choices produce different values of $d$ and different sparsity patterns in $M$, which may be beneficial in different scenarios. In this section, we propose three different methods for constructing $M$, each of which provides a different tradeoff between sparsity and leading dimension of $M$.

Our first method prioritizes sparsity. Assume that $W$ is a \emph{Stieltjes matrix}, a symmetric positive semidefinite matrix with nonpositive off-diagonal entries. This can be guaranteed by adding additional nonpositivity constraints for the off-diagonal entries of $W$ to $\mathcal{C}$, which preserves any existing convexity structure in the problem. Empirically, we note that solutions to \eqref{main_prob} often return a Stieltjes $W$ without introducing any additional constraints enforcing the condition. Let $\mathcal{N} = \left((i,j): j > i \text{ and } W_{ij} \neq 0\right)$ be a tuple containing the nonzero entries in the strictly upper triangular part of $W$, ordered arbitrarily. Let $d = |\mathcal{N}|$. Define $B \in \mathbb{R}^{d \times n}$ as
\begin{equation}
B_{ik} = 
\begin{cases} 
-1  & \text{ if } \mathcal{N}_{k} = (i,j) \text{ for some } i \in [n]\\ 
1  & \text{ if } \mathcal{N}_{k} = (j,i) \text{ for some } j \in [n]\\
0 & \text{ otherwise}.
\end{cases}
\end{equation}
Define $V \in \mathbb{R}^{d \times d}$ to be a diagonal matrix with $V_{kk} = -W_{ij}$ where $\mathcal{N}_{k} = (i,j)$. Note that the Stieltjes assumption on $W$ gives that $V$ has nonnegative entries. We claim that $B V B^{T} = W$, which implies that setting $M = V^{1/2} B^{T}$ gives $M^{T} M = W$.

To show $M^{T} M = W$, consider off-diagonal and diagonal terms of $W$ separately. For off-diagonal entries, the inner product of columns $i$ and $j$ in $M$ is equal to $W_{ij}$ because the index $k$ is the only nonzero entry in both vectors. For diagonal terms, the $i$th entry along the diagonal of $M^{T} M$ is the sum of squares of the entries of column $i$ in $M$, $-\sum_{j \neq i} W_{ij}$. Constraint \eqref{con1} gives that $-\sum_{j \neq i} W_{ij} = W_{ii}$, so diagonal entries of $M^{T} M$ also equal diagonal entries of $W$. This method aggressively prioritizes sparsity at the cost of increasing $d$ to the number of nonzero entries in the strictly upper triangular part of $W$. Additionally, it imposes an additional restriction (the Stieltjes condition) on the set of feasible $W$. 

Another option, which prioritizes both sparsity and small value of $d$, is to construct the matrix $M$ via Cholesky decomposition, where $W = M^{T} M$ and $M$ is upper triangular. Additional sparsity can be imposed by performing a sparse Cholesky decomposition which results in matrices $P$, $L$ and $D$ such that
$$W = P L D L^{T} P^{T}.$$
The matrix $P \in \mathbb{R}^{n \times n}$ is a permutation matrix chosen to minimize fill in, $L$ is lower triangular with unit diagonal, and $D$ is a diagonal matrix. Since $W$ is positive semidefinite, the diagonal matrix $D$ has nonnegative entries. Additionally, since $\mathrm{rank}(W) = n-1$, exactly one of $D$'s diagonal entries is zero. If the zero occurs in the $i$th entry of the diagonal, denote by $\tilde{D}$ the $n-1 \times n -1$ matrix with the $i$th row and column removed. Likewise, denote by $B$ the matrix $P L $ and by $\tilde{B}$ the matrix $B$ with its $i$th column removed. Then
$$W = \tilde{B} \tilde{D} \tilde{B}^{T},$$
so taking $M = \tilde{D}^{1/2} \tilde{B}^{T}$ results in a sparse $M \in \mathbb{R}^{n-1 \times n}$. Note that, in this sparse Cholesky decomposition, the matrix $M$ is upper triangularizable via a permutation of its columns, but may not actually be upper triangular. In this construction the dimension $d$ of the iteration \eqref{d_iteration} is $n-1$. Relative to the decomposition in the previous paragraph, the sparse Cholesky decomposition has a smaller $d$ value and does not require the Stieltjes condition, but may be less sparse.

A third method for constructing $M$ uses the eigendecomposition of $W$. Let $U \Lambda U^{T}$ be an eigendecomposition of $W$, where $\Lambda$ is a diagonal matrix containing the (nonnegative) eigenvalues of $W$. Since $\mathrm{Null}(W) = \mathrm{span}\{\1\}$, exactly one eigenvalue is equal to $0$. Let $\tilde{\Lambda} \in \mathbb{R}^{(n-1) \times (n-1)}$ be the matrix $\Lambda$ with the zero eigenvalue removed, and $\tilde{U} \in \mathbb{R}^{n \times (n-1)}$ be the submatrix of $U$ which removes the column corresponding to the zero eigenvalue. Then $W = \tilde{U}\, \tilde{\Lambda} \,\tilde{U}^{T}$, so taking $M = \tilde{\Lambda}^{1/2}\, \tilde{U}^{T}$ yields an $M$ with $d = n-1$. Though this method for constructing $M$ has $d = n-1$ as in the sparse Cholesky method, it does not prioritize sparsity.

Regarding the dimension $d$ of the iteration \eqref{d_iteration}, \cite[Theorem 1]{malitsky2023resolvent} shows that $d \geq n-1$ is necessary for convergence of the algorithm \eqref{d_iteration}. For any $W$, we have provided two different methods (the Cholesky and eigendecomposition methods) which attain that lower bound. This allows us to prove that every sequence of possible $x$ values generated by an iteration of the form \eqref{d_iteration} can be generated by an iteration of minimal dimension. We formalize this result in the following theorem, deferring its proof to the appendix.

\begin{theoremrep} \label{min_lifting}
    Every frugal resolvent splitting given by iteration \eqref{d_iteration} has an equivalent frugal resolvent splitting with minimal lifting. That is, for any $M \in \mathbb{R}^{d \times n}$ and $L \in \mathbb{R}^{n \times n}$ for which $W = M^{T} M$ and $Z = 2I - L - L^{T}$ are feasible in \eqref{main_prob} for some constants $c$ and $\epsilon$ and set $\mathcal{C}$, there exists $\tilde{M} \in \mathbb{R}^{n-1 \times n}$ such that for any initial point $\v{z}^{0}$ there is an initial point $\tilde{\v{z}}^{0}$ for which the iterations
\begin{minipage}{.4\textwidth}
\begin{align*}
\mk{x} &= J_{\v{A}}\left(-\mk{M}^{T} \mk{z}^{k} + \mk{L} \mk{x}\right)\\
\mk{z}^{k+1} &= \mk{z}^{k} + \gamma \mk{M} \mk{x}.
\end{align*}
\end{minipage}%
\begin{minipage}{.2\textwidth}
\begin{center}
and 
\end{center}
\end{minipage}%
\begin{minipage}{.4\textwidth}
\begin{align*}
\mk{x} &= J_{\v{A}}\left(-\tilde{\mk{M}}^{T} \tilde{\mk{z}}^{k} + \mk{L} \mk{x}\right)\\
\tilde{\mk{z}}^{k+1} &= \tilde{\mk{z}}^{k} + \gamma \tilde{\mk{M}} \mk{x}.
\end{align*}
\end{minipage}
produce the same sequence of $\v{x}$ values.
\end{theoremrep} 
\begin{proof}
Let $M$ and $L$ be as in the theorem statement, and let $\v{z}^{0} \in \HH^d$ be an initial starting point. Construct $\tilde{M}$ via one of the methods (either Cholesky or eigendecomposition) proposed in Section \ref{Sec:MfromW} such that $\tilde{M}^{T} \tilde{M} = W$ and $\tilde{M} \in \mathbb{R}^{(n-1) \times n}$. Note that any point $\tilde{\v{z}}^{0} \in \HH^{n-1}$ such that 
\begin{equation}
\v{M}^{T} \v{z}^{0} = \tilde{\v{M}}^{T} \tilde{\v{z}}^{0} \label{need_to_show}
\end{equation}
will generate the same sequence of $\v{x}$ values in the two iterations.

To show that the system \eqref{need_to_show} has a solution $\tilde{\v{z}}^{0}$, we will show that $\v{M}$ and $\tilde{\v{M}}$ have the same row space. Recall that the row space of a matrix is the orthogonal complement of its null space. But the two matrices $\v{M}$ and $\tilde{\v{M}}$ have the same null space, $\mathrm{span}(\mathbbm{1})$. This is because, for any $\v{v}$, 
$$\|\v{M} \v{v}\|^{2} = \langle \v{M}\v{v}, \v{M}\v{v}\rangle = \langle \v{v}, \v{W}\v{v}\rangle$$
which by \eqref{con1}-\eqref{con2} is equal to $0$ if and only if $\v{v} \in \mathrm{span}(\mathbbm{1})$. The same result holds for $\tilde{\v{M}}$ since $\tilde{\v{M}}^{T} \tilde{\v{M}} = \v{W}$ as well. Therefore $\v{M}$ and $\tilde{\v{M}}$ have the same null space, which implies that their row spaces, the orthogonal complements to their null spaces, are equal. We conclude that the system \eqref{need_to_show} has a solution $\tilde{\v{z}}^{0}$ for every $\v{z}^{0}$, which proves the result.
\end{proof}

%% file: graph_theory.tex
The language of graph theory provides a helpful description of the sparsity of the matrices in problem \eqref{main_prob} and information exchange between resolvent operators. In this section, we recall some concepts from graph theory that will be useful in the remainder. For further details we refer the reader to \cite{chung1997spectral}.

A graph $G = (\mathcal{N}, \mathcal{E})$ is a set of nodes $\mathcal{N}$ and a set of edges $\mathcal{E}$ connecting those nodes. The degree of a node is the number of edges containing that node, and the degree matrix of a graph is an $|\mathcal{N}| \times |\mathcal{N}|$ diagonal matrix with degree of each node on the diagonal. The adjacency matrix $A$ of a graph $G$ is an $|\mathcal{N}| \times | \mathcal{N}|$ matrix with $A_{ij} = 1$ if there is an edge between node $i$ and node $j$ and $0$ otherwise. The Laplacian matrix $K$ of a graph $G$ is $K = D - A$, where $D$ is the degree matrix and $A$ the adjacency matrix of $G$. An orientation of a graph assigns a direction to each edge; given an oriented graph its edge-incidence matrix is an $|\mathcal{N}| \times |\mathcal{E}|$ matrix $B$ such that 
\begin{align*}
B_{ve} = \begin{cases}
1 & \text{ if node } v \text{ is the head of edge } e \\
-1 & \text{ if node } v \text{ is the tail of edge } e \\
0 & \text{ otherwise}.
\end{cases}
\end{align*}
The oriented edge-incidence matrix of a connected graph has $\mathrm{Null}(B^{T}) = \text{span}\{\mathbbm{1}\}$. The oriented edge-incidence matrix is connected to the Laplacian matrix $K$ because $K = B \, B^{T}$. Every Laplacian matrix also necessarily satisfies $K \mathbbm{1} = 0$, and for $K \in \Sp^n$, $\lambda_{2}(K) > 0$ implies the graph is connected. 

The Laplacian matrix of a connected graph satisfies the conditions on $W$ in \eqref{con1} and \eqref{con2} for some $c > 0$, and the connection between the oriented edge-incidence matrix and the Laplacian allows one to take $M = B^{T}$ in \eqref{d_iteration} if one can pair it with a feasible $Z$. In general, finding a feasible $Z$ for a given $W$ is a difficult problem and motivates use of the SDP in \eqref{main_prob}, though \cite{tam2023frugal} notes in a more limited setting that taking $W = Z$, where $W$ is a scalar multiple of the Laplacian of a connected regular graph, yields a convergent resolvent splitting algorithm.

Extending the graph theoretic interpretation to the setting of directed graphs with weighted edges allows a graph characterization of iteration \eqref{n_iteration} for any $W$ and $Z$.
If $V$ is an $|\mathcal{E}| \times |\mathcal{E}|$ diagonal matrix with edge weights on the diagonal, the weighted graph Laplacian is defined as $K = B V B^{T}$ where $B$ is the (unweighted) oriented edge-incidence matrix. In the remainder, we will denote by $G(K)$ the weighted graph induced by $K$ by interpreting $K$ as a weighted graph Laplacian, i.e. the graph with the edge weight between nodes $i$ and $j$, where $i \neq j$, given by $-K_{ij}$. Interpreting the $W$ returned by \eqref{main_prob} as a weighted graph Laplacian motivates the first $M$ construction in Section \ref{Sec:MfromW}.
Additionally, any feasible $Z$ in \eqref{main_prob} can be shown to be a weighted graph Laplacian, a result which we prove in the appendix's Proposition \ref{lap_full}.

\begin{toappendix}
\begin{propositionrep} \label{lap_full}
    Any feasible $Z$ in \eqref{main_prob} is a weighted graph Laplacian.
\end{propositionrep}

\begin{proof}   
We will show that $Z$ has $Z \1 = 0$, which permits the decomposition $Z = B V B^{T}$, where $B$ is an oriented edge-incidence matrix and $V$ a diagonal matrix of corresponding edge weights, following the steps in section \ref{Sec:MfromW}. 

First, we note some important properties of $Z$ that we will utilize. Constraint \eqref{con3} implies that $Z$ is positive semidefinite. Furthermore, note that we have $Z_{11} > 0$ since $\epsilon < 2$.

Suppose for contradiction $Z \in \Sp^n$ is a matrix satisfying \eqref{con4} and \eqref{con5} for which $Z \1 \ne 0$. Since the total matrix sums to 0, we have at least one column summing to a positive value and another summing to a negative value. Without loss of generality, assume that the first column sums to a negative value, e.g. $\sum_{j=1}^n Z_{1j} = -p$ where $p > 0$. Looking at the off-diagonal sums in the lower triangle, we have:
\begin{equation}
    \sum_{i=1}^n\sum_{j>i} Z_{ij} = -\frac{n\, Z_{11}}{2}
\end{equation}
\begin{equation}
    \sum_{j > 1} Z_{1j} = -Z_{11}-p \label{lap_p1}
\end{equation}
\begin{equation}
    \sum_{i=2}^n\sum_{j>i} Z_{ij} = -\frac{n\,Z_{11}}{2} + Z_{11} + p \label{lap_tri}
\end{equation}
Choose some positive $\delta$ such that $0 < \delta < \frac{2p}{Z_{11}}$. Define $x \in \R$ such that $x_1 = 1 + \delta$ and $x_j=1$ otherwise. We will show that $x^T Z x < 0$.
Using symmetry and the expansion $x^T Z x = \sum_{i=1}^n \sum_{j=1}^n Z_{ij} x_i x_j$, we have
\begin{align}   
x^T Z x &=\sum_{i=1}^n Z_{ii}x_i^2 + 2\sum_{i=1}^n\sum_{j>i}Z_{ij}x_i x_j \\
&=Z_{11}(1+\delta)^2 + Z_{11}(n-1) + 2\left[\sum_{j>1}(1+\delta)Z_{1j} + \sum_{i=2}^n\sum_{j>i}Z_{ij}\right] \\
&=n\, Z_{11} + 2 Z_{11} \delta + Z_{11} \delta^2 + 2\left[(1+\delta)(-Z_{11}-p) + \left(-\frac{n Z_{11}}{2}+Z_{11}+p\right)\right] \\
&=n\, Z_{11} + 2 Z_{11} \delta + Z_{11} \delta^2 + 2\left[- Z_{11} \delta - \delta p - \frac{n Z_{11}}{2} \right] \\
&=Z_{11} \delta^2 - 2p \delta \label{lap_result}
\end{align}
For any $0 < \delta <  \frac{2p}{Z_{11}}$, expression \eqref{lap_result} is negative. Thus we have a contradiction to $Z \succeq 0$, having shown the existence of a vector $x$ such that $x^T Z x < 0$. We conclude that each row of $Z$ sums to $0$, which proves the result.
\end{proof}
\end{toappendix}

A graph theoretic interpretation of $W$ provides insight into constraint \eqref{con2}. Since $W \mathbbm{1} = 0$ and $W \succeq 0$, we have $\lambda_{1}(W) = 0$. $\lambda_{2}(W)$ is called the \emph{Fiedler value} of the graph $G(W)$. For unweighted graphs, the Fiedler value provides a lower bound on the minimum number of edges which must be removed to form a disconnected graph, and in general graphs with larger Fiedler values have greater connectivity. For an unweighted connected graph on $n$ nodes, the minimal Fiedler value is given by $2\left(1-\cos{\frac{\pi}{n}}\right)$ \citep{fiedler1973algebraic}. We use this value as the default choice of the parameter $c$ in \eqref{con2}, since all connected unweighted graphs, i.e. connected graphs with unit edge weights, are included in the set $\left\{G(W): \lambda_{1}(W) + \lambda_{2}(W) \geq 2\left(1-\cos{\frac{\pi}{n}}\right)\right\}$. 

Viewing $W$ as the Laplacian of a weighted graph also provides an intuitive interpretation of Malitsky and Tam's Minimal Lifting Theorem, \cite[Theorem 1]{malitsky2023resolvent}, when the matrix $W$ is Stieltjes. This theorem states that $d$, the dimension of $\v{z}$ in \eqref{d_iteration}, must be $\geq n-1$. If $d < n-1$, then $W = M^{T} M$ has rank $< n-1$ since $M \in \mathbb{R}^{d \times n}$, so $\lambda_{2}(W) = 0$ and $G(W)$ is disconnected \cite{fiedler1975property}. Clearly, any algorithm with a disconnected graph purporting to reach consensus among all of its nodes may fail to converge.

We next provide a collection of necessary conditions for $W$ and $Z$ to satisfy the constraints \eqref{main_prob}. These conditions are naturally stated in graph theoretic language, and they provide general guidelines for constructing a nonempty constraint set $\mathcal{C}$.

\begin{lemmarep}\label{Lem:Limitations}
  For any matrices $W, Z$ which satisfy constraints \eqref{main_prob} we have the following:
  \begin{enumerate}[\normalfont(a)]
    \item All entries in $Z$ and $W$ have magnitude bounded above by $Z_{11}$. Phrased graph-theoretically, edge weights in $G(Z)$ and $G(W)$ have magnitude bounded above by the (constant) degree of the nodes in $G(Z)$. \label{Lem:lap_bound}
    \item $G(W)$ is connected.\label{Lem:wconnect}
    \item $G(W)$ must have at least one edge connected to each node.\label{Lem:w1}
    \item $G(W)$ must have at least $n-1$ edges. \label{Lem:wn1}
    \item $G(Z)$ is connected.\label{Lem:zconnect}
    \item If $n > 2$, $G(Z)$ must have at least two edges connected to each node.\label{Lem:z2}
    \item If $n > 2$, $G(Z)$ must have at least $n$ edges. \label{Lem:zn}
    \item For any subset $\mathcal{S}$ of nodes in $G(Z)$, let $\mathcal{E}_S$ be the set of edges between nodes in $\mathcal{S}$. It then holds that $\abs{|\mathcal{S}|-|\mathcal{S}^C|} \leq 2\left(\abs{\mathcal{E}_S} + \abs{\mathcal{E}_{S^C}}\right)$. \label{Lem:cutset}
  \end{enumerate}
\end{lemmarep}
\begin{proof}
  \begin{enumerate}[(a)]
   \item By constraints \eqref{con3} and \eqref{con8}, $Z \succeq W \succeq 0$, $Z \in \Sp^n$, and we know its principal minors have non-negative determinant. We also know by constraint \eqref{con5} that $Z_{ii} = Z_{11}$ for all $i$. Therefore, for any $i$, $j$, the determinant of its principle minor is
\begin{equation}\label{principal_minor}
\begin{vmatrix}
      Z_{11} & Z_{ij} \\
      Z_{ij} & Z_{11} \\
 \end{vmatrix} \geq 0
 \end{equation}
 This implies $Z_{ij}^2 \leq Z_{11}^2$, and $\abs{Z_{ij}} \leq Z_{11}$. Since $Z-W \succeq 0$, $e_{i}^{T} (Z - W) e_{i} \geq 0$ for standard basis vector $e_{i}$ gives $W_{ii} \leq Z_{ii}$. Using a similar principal minor argument as \eqref{principal_minor}, we see that $W_{ij}^2 \leq W_{ii}W_{jj} \leq Z_{11}^2$ and thus $\abs{W_{ij}} \leq Z_{11}$. Therefore all matrix entries $Z_{ij}$ and $W_{ij}$ have magnitude less than or equal to $Z_{11}$.

    \item If $G(W)$ has $\geq 2$ connected components, then $\text{Null}(W)$ is at least two-dimensional, since both $\1$ and the vector with 1 on the vertices of the first connected component and 0 otherwise are linearly independent and in $\text{Null}(W)$. Since $W \in \Sp^n$ and $\lambda_2(W) > c > 0$, the dimension of $\text{Null}(W)$ is 1. Therefore $G(W)$ cannot have $\geq 2$ connected components, and we conclude that the number of connected components in $G(W)$ is 1, so the graph is connected.
    \item Follows directly from part \ref{Lem:wconnect}
    \item Follows directly from part \ref{Lem:wconnect}.
    \item $Z \succeq W$, so, by Weyl's Monotonicity Theorem, $\lambda_2(Z) \geq \lambda_2(W) > c > 0$. Therefore the number of connected components in $G(Z)$ is 1, and the graph is connected.
    \item Suppose some node $i$ in $G(Z)$ has zero edges. Then $G(Z)$ is disconnected, violating part \ref{Lem:zconnect}. Suppose instead that node $i$ has a single edge, $(i,j')$. Since $Z\1=0$, $Z_{ii} = -\sum_{j \ne i}Z_{ij} = -Z_{ij'}$, so that edge has value $Z_{ii} > 0$. If node $j'$ has no other edges, the two nodes are disconnected from the rest of the graph. Since $n > 2$, we then have a disconnected graph, violating part \ref{Lem:zconnect}. 
    Suppose instead that $j'$ does have other edges. Since $Z_{j'j'} = -\sum_{j \ne j'} Z_{jj'}= Z_{ii} - \sum_{j \not\in \{ i, j'\}} Z_{jj'}$, $\sum_{j \not \in \{ i, j'\}} Z_{jj'}=0$. Since $j'$ has other edges, and their weight sums to zero, at least one edge $(j',j'')$ must have $Z_{j'j''} > 0$.
    Consider the determinant of the principle minor in $i$, $j'$, and $j''$:
    \begin{align}
      \begin{vmatrix}
      Z_{11}  & -Z_{11} & 0\\
      -Z_{11} & Z_{11}  & Z_{j'j''}\\
      0       & Z_{j'j''} & Z_{11} \\
 \end{vmatrix} &= Z_{11}^3 - Z_{11} Z_{j'j''}^2 - Z_{11}^3 \\
 &= - Z_{11} Z_{j'j''}^2 < 0.
\end{align}
This is a contradiction, however, since $Z \succeq 0$ and its principle minors must all be nonnegative. Therefore each node in $G(Z)$ has at least two edges.
    \item Follows directly from parts \ref{Lem:zconnect} and \ref{Lem:z2}.
    \item We note first that by \eqref{main_prob}, $Z$ is symmetric and $\diag{(Z)}=Z_{11} > 0$. 
    We also note the result in Proposition \ref{lap_full}, which establishes $Z\1=0$. 
    Therefore, in any row of $Z$, $-\sum_{j \ne i}Z_{ij} = Z_{11}$. We write $\delta_S$ as cutset of $\mathcal{S}$, that is, the set of edges with one node in $\mathcal{S}$ and the other in $\mathcal{S^C}$, and write $\mathcal{T} = \mathcal{S^C}$.
    Over the rows for nodes in $\mathcal{S}$ and $\mathcal{T}$, we have:
    \begin{align*}
      \abs{\mathcal{S}}Z_{11} &= -2\sum_{i,j \in \mathcal{E}_S}Z_{ij} - \sum_{i,j \in \delta_S}Z_{ij}\\
    \abs{\mathcal{T}}Z_{11} &= -2\sum_{i,j \in \mathcal{E}_T}Z_{ij} - \sum_{i,j \in \delta_T}Z_{ij}
\end{align*}
    Since $\mathcal{T} = \mathcal{S^C}$, 
    \begin{align*}
      &\sum_{i,j \in \delta_S}Z_{ij} = \sum_{i,j \in \delta_T}Z_{ij}\\
      \implies &\abs{\mathcal{T}}Z_{11} = -2\sum_{i,j \in \mathcal{E}_T}Z_{ij} + \abs{\mathcal{S}}Z_{11} + 2\sum_{i,j \in \mathcal{E}_S}Z_{ij}\\
      \implies & Z_{11}\left(\abs{\mathcal{S}}-\abs{\mathcal{T}}\right) = 2\left(\sum_{i,j \in \mathcal{E}_T}Z_{ij} - \sum_{i,j \in \mathcal{E}_S}Z_{ij}\right)
    \end{align*}
    
    We note the result in part \ref{Lem:lap_bound}, which says $\abs{Z_{ij}} \leq Z_{11}$. Therefore
    \begin{align*}
      & 2\abs*{\sum_{i,j \in \mathcal{E}_T}Z_{ij} - \sum_{i,j \in \mathcal{E}_S}Z_{ij}} \leq 2\left(\abs{\mathcal{E}_S} + \abs{\mathcal{E}_T}\right) Z_{11}\\
\implies &\abs*{Z_{11} \left(\abs{\mathcal{S}}-\abs{\mathcal{T}}\right)} \leq 2\left(\abs{\mathcal{E}_S} + \abs{\mathcal{E}_T}\right)Z_{11}\\
\implies &\abs*{\abs*{\mathcal{S}}-\abs*{\mathcal{T}}} \leq 2\left(\abs{\mathcal{E}_S} + \abs{\mathcal{E}_T}\right)
    \end{align*}
\end{enumerate}
\end{proof}

%% file: examples.tex
One crucial aspect of problem \eqref{main_prob} is the trade-off implied by increased sparsity in $W$ and $Z$. Higher levels of sparsity in the matrices imply less exchange of resolvent calculations. This has two important advantages. First, when structured correctly, it permits parallelization among resolvent calculations, reducing the time necessary to complete one iteration of either \eqref{d_iteration} or \eqref{n_iteration}. Second, in applications where the resolvent calculations are decentralized, exchange of data can be a computational bottleneck, so minimizing the number of exchanges required reduces the time necessary to complete each iteration. In problem \eqref{main_prob}, constraints on the sparsity patterns of $W$ and $Z$ can be incorporated via the constraint set $\mathcal{C}$ while preserving the convexity of the problem.
Increased density, on the other hand, increases the rate of information exchange between resolvents.

We illustrate the benefits of controlling for sparsity with the following example, where we design a customized algorithm for solving a convex program $\min_{x} \sum_{i=1}^{6} f_{i}(x)$. Assume we have access to a compute cluster organized as in Figure \ref{fig:cluster}, where each resolvent calculation $J_{\partial f_{i}}$ is computed by the $i$th machine. Two groups of three machines have low latency connections among themselves, with one of the machines in each group possessing a high latency connection to a machine in the other group. Communications between computers which are not directly connected to one another are not permitted. In this setting, a successful algorithm will allow unrestricted access to the high latency connections depicted with green arrows, while minimizing use of the low latency connection depicted with red.  Such an algorithm can be realized in the context of problem \eqref{main_prob} by using $\mathcal{C}$ to constrain the sparsity structure of $W$ and $L$ such that $W_{ij} = 0$ and $Z_{ij}$ equals zero whenever there is no connection between machines $i$ and $j$. We also take the objective $\phi$ to be the number of nonzero entries in $W_{14}$, $W_{41}$, $Z_{14}$, $Z_{41}$. Though $\phi$ is nonconvex in this setting, solutions can still be easily computed using the techniques described in Section \ref{Sec:mip_formulations}, yielding the following as one solution.

\begin{figure}
\begin{center}
\includesvg[width=.5\textwidth]{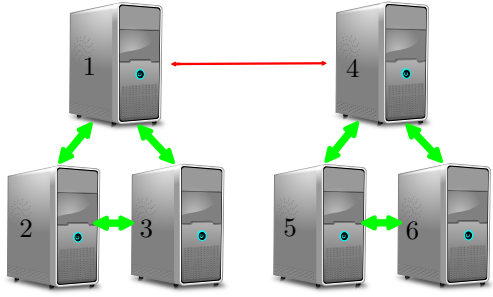}
\end{center}
\caption{A hypothetical cluster for which we design a custom splitting algorithm. Red arrows give high latency connections, which should be minimized, and green arrows give low latency connections.}
\label{fig:cluster}
\end{figure}

{\footnotesize
\begin{equation}\label{sparse_sol}
W =  
\begin{pmatrix}
    1.86 &  -0.52 &  -0.52 &  -0.83 &    0.00 &    0.00\\
  -0.52 &    1.33 &  -0.81 &    0.00 &    0.00 &    0.00\\
  -0.52 &  -0.81 &    1.33 &    0.00 &    0.00 &    0.00\\
  -0.83 &    0.00 &    0.00 &    1.86 &  -0.52 &  -0.52\\
    0.00 &    0.00 &    0.00 &  -0.52 &    1.33 &  -0.81\\
    0.00 &    0.00 &    0.00 &  -0.52 &  -0.81 &    1.33
\end{pmatrix}
\quad 
Z = 
\begin{pmatrix}
    2.00 &  -0.56 &  -0.56 &  -0.88 &    0.00 &    0.00\\
  -0.56 &    2.00 &  -1.44 &    0.00 &    0.00 &    0.00\\
  -0.56 &  -1.44 &    2.00 &    0.00 &    0.00 &    0.00\\
  -0.88 &    0.00 &    0.00 &    2.00 &  -0.56 &  -0.56\\
    0.00 &    0.00 &    0.00 &  -0.56 &    2.00 &  -1.44\\
    0.00 &    0.00 &    0.00 &  -0.56 &  -1.44 &    2.00
\end{pmatrix}.
\end{equation}
}
The matrices in \eqref{sparse_sol} yield convergent algorithms in iteration \eqref{n_iteration} and (when $W$ is factored as $M^{T} M$) iteration \eqref{d_iteration}. By constraining the sparsity patterns on these matrices, we construct a decentralized optimization algorithm that respects the communication structure among machines dedicated to computing the proximal values of each function $f_{i}$. Note that these matrices are not able to avoid communication along the low latency connection, as evidenced by the nonzero values in $W_{14}$ and $Z_{14}$. Removing these communications is not possible because doing so partitions the machines into two separate groups that do not communicate with each other via application of $\v{W}$ or $\v{Z}$, thereby violating Lemma \ref{Lem:Limitations}.

\begin{figure}
  \begin{center}
    \includesvg[width=.5\textwidth]{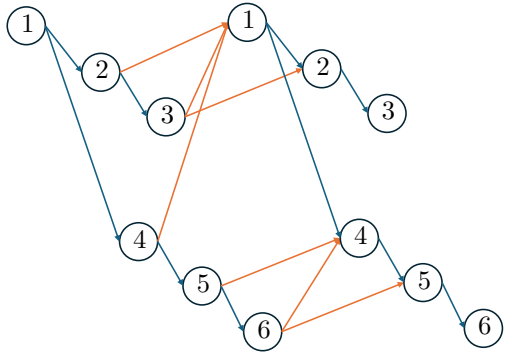}
    \end{center}
    \caption{An activity network depicting the dependencies among resolvent calculations in example \eqref{sparse_sol}. Blue lines are within-iteration dependencies from \eqref{n_itr1}, and orange lines are between-iteration dependencies from \eqref{n_itr2}.}
    \label{fig:activity}
  \end{figure}

In general, any off-diagonal nonzero value in $Z$ and $W$ implies a required communication between the resolvents corresponding to its row and column. For any $Z_{ij} \neq 0$ where $i < j$ (and $Z_{ji} \neq 0$ by symmetry), resolvent operator $i$ must provide its output to resolvent $j$ within the same iteration. If $W_{ij} \neq 0$, resolvent calculations must be provided by $j$ to $i$ and $i$ to $j$ to supply the information for the subsequent iteration.

By choosing entries of $W$ and $Z$ to be 0, we can design algorithms that allow resolvent calculations to run in parallel, both within a given iteration and across iterations. 
Within an iteration, if resolvent $j$ does not rely on information from resolvent $i$ ($Z_{ji} = 0$),
these resolvents can be computed independently of one another. So, if each has received the required inputs from the other resolvents, resolvents $i$ and $j$ can be executed in parallel. Across iterations, since a given calculation may only depend on a subset of other resolvents to calculate $v_i$ or $\left(\v{M}^T \v{z}\right)_i$, resolvent $i$ can begin the next iteration while resolvent $j$ is still finishing the current iteration if $W_{ij}=0$.

Figure \ref{fig:activity} depicts the activity network of resolvent calculations for the example in \eqref{sparse_sol}, where the dependencies between resolvents are those implied by $W$ and $Z$. Blue edges denote dependencies in \eqref{n_itr1} and orange edges the dependencies in \eqref{n_itr2}. When resolvent computation and communication times can be estimated, these can be modeled in the activity network by using appropriate edge weights. Figure \ref{fig:parallel} depicts a Gantt chart which demonstrates the parallel resolvent computation and the execution timeline for the algorithm in \eqref{sparse_sol}. For this Gantt chart, we assumed that each resolvent completes in 16 milliseconds, low latency communications (green edges in Figure \ref{fig:cluster}) take .25 milliseconds, and high latency communications (red edges in Figure \ref{fig:cluster}) take 10 milliseconds. In this algorithm, resolvent 4 only depends on resolvent 1 for its input, so as long as it has also received necessary information from the previous iteration, it can execute its computation in parallel with resolvent 2 or 3. Looking across iterations, resolvent 1 only relies on inputs from resolvents 2, 3, and 4, since $W_{15}$ and $W_{16}$ are 0. Once those resolvents have passed their outputs, resolvent 1 can begin the next iteration, even if resolvents 5 and 6 have not completed their computation.

\begin{figure}
  \begin{center}
\includegraphics[width=.8\textwidth]{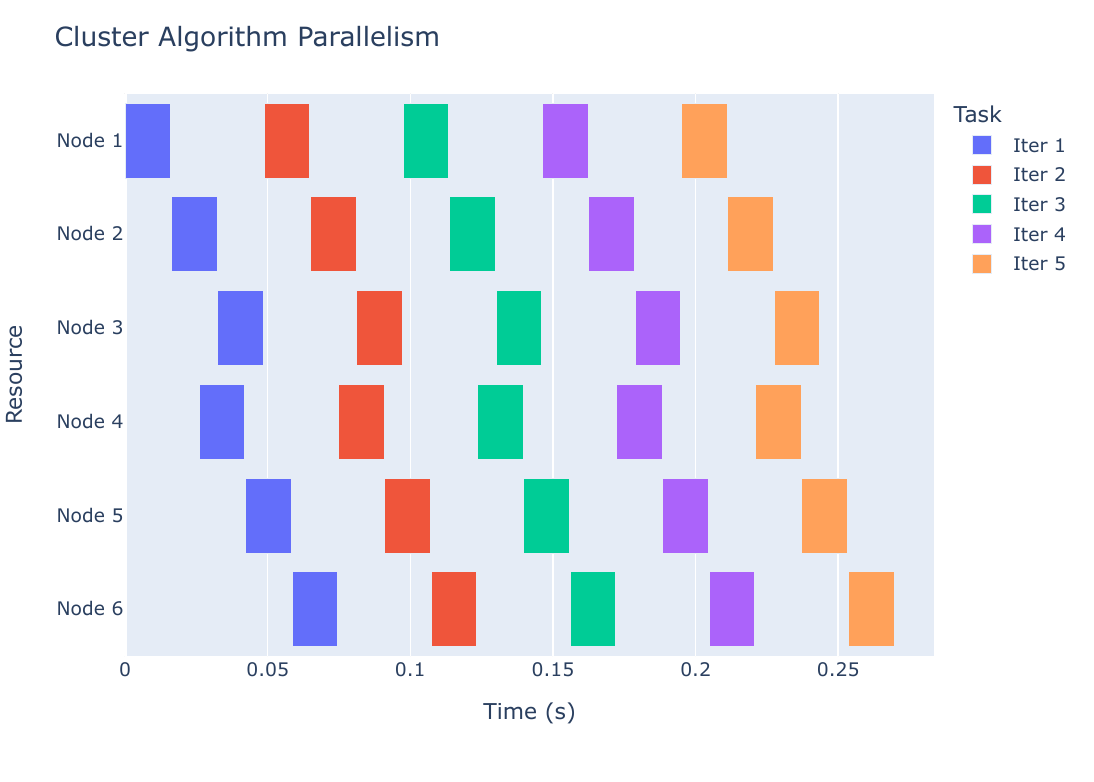}
\end{center}
\caption{Execution timeline showing parallelism within and across iterations for the cluster given in Figure \ref{fig:cluster} using the algorithm specified by \eqref{sparse_sol} and communication and computation times as given in the text.}
\label{fig:parallel}
\end{figure}

%% file: max_dense_sparse.tex
Two extreme examples of sparsity in $W$ and $Z$ are the maximally dense and maximally sparse settings. The maximally dense setting refers to the situation where $W$ and $Z$ have no zero entries. We refer to this as the \textit{fully connected} design, and demonstrate empirically in Section \ref{Sec:Convergence} that it frequently offers the best convergence rate for a given problem. The maximally sparse setting, on the other hand, minimizes the number of nonzero entries in $W$ and $Z$.

In the context of maximally sparse splittings, \cite{malitsky2023resolvent} shows that the minimal number of nonzero entries in $W$ is exactly $n-1$, though their optimal placement given a particular constraint set $\mathcal{C}$ (and the placement of a minimal number of edges in $Z$) is an important open question. To solve it for $n>2$, we propose the following mixed integer semidefinite program (MISDP), which determines the sparsity structure of $W$ and $Z$ through the use of binary variables which track the nonzero entries of these matrices.
\begin{subequations}\label{min_edge_sdp}
\begin{align} 
    \min_{x, y, Z, W} \quad & \sum_{i=1}^{n} \sum_{j<i}  x_{ij} + y_{ij} \nonumber \\
      \text{s.t.}  \quad
    & (2 + \epsilon) x_{ij} \geq \abs{Z_{ij}} \quad \forall i,j \in [n] \label{misdp_zx}\\
    & (2 + \epsilon) y_{ij} \geq \abs{W_{ij}} \quad \forall i,j \in [n] \label{misdp_wy}\\
    & \sum_{i=1}^{n} \sum_{j<i} x_{ij} \geq n  \label{misdp_opt_xn}\\
    & \sum_{i=1}^{n} \sum_{j<i} y_{ij}\geq n-1 \label{misdp_opt_yn}\\
    & \sum_{j\ne i} x_{ij} \geq 2 \quad \forall i \in [n]\label{misdp_opt_x2}\\
    & \sum_{j\ne i} y_{ij} \geq 1 \quad \forall i \in [n] \label{misdp_opt_y1}\\
    & x \in \mathbb{S}^{n}, \quad x_{ij} \in \{0,1\} \quad \forall i,j \in [n]\\
    & y \in \mathbb{S}^{n}, \quad y_{ij} \in \{0,1\} \quad \forall i,j \in [n]\\
    & \text{SDP constraints \eqref{con1}-\eqref{con8}}
    \end{align}
\end{subequations}

In \eqref{min_edge_sdp}, the objective counts the number of nonzero entries in strictly lower triangular parts of $W$ and $Z$. Constraints \eqref{misdp_zx} and \eqref{misdp_wy} require weights which are not indicated as nonzero to be zero, relying on Lemma \ref{Lem:Limitations}\ref{Lem:lap_bound} to bound the value of each entry. Constraints \eqref{misdp_opt_xn}-\eqref{misdp_opt_y1} are constraints which utilize Lemma \ref{Lem:Limitations} to tighten the formulation; they are optional but reduce the time required by MISDP solvers, e.g. SCIP-SDP \cite{gally2018framework}, to solve the problem. Constraint \eqref{misdp_opt_xn} implements Lemma \ref{Lem:Limitations}\ref{Lem:zn}, requiring the existence of at least $n$ edges in $G(Z)$. Constraint \eqref{misdp_opt_yn} implements \ref{Lem:Limitations}\ref{Lem:wn1}, requiring the existence of at least $n-1$ edges in $G(W)$. Constraint \eqref{misdp_opt_x2} ensures each node in $G(Z)$ has at least two edges (required by Lemma \ref{Lem:Limitations}\ref{Lem:z2}). Lemma \ref{Lem:Limitations}\ref{Lem:w1} gives that there must be at least one edge to each node in $G(W)$, which is captured in Constraint \eqref{misdp_opt_y1}.

%% file: general_min_cycle.tex
The parallelism demonstrated in Figure \ref{fig:parallel} raises the question, can one generate algorithms which achieve maximal parallelism, and therefore minimize the time required to execute some fixed number of iterations? 
In this subsection, we introduce terminology and characterizations of maximally parallel algorithms, while focusing on computation in Section \ref{Sec:Objectives}. We focus our attention on \eqref{n_iteration} since it permits control of sparsity structure directly via constraints on $W$.

We assume in the remainder that all resolvent computation times for a given resolvent are fixed, meaning that they do not vary across iteration or input value. These are given by $t_{i} > 0$ for each resolvent $i \in [n]$. We similarly assume fixed communication times $l_{ij} = l_{ji} > 0$ for all $i<j \in [n]$. All other calculations in \eqref{n_iteration} involve scalar arithmetic on co-located data and are assumed negligible. 
Let $s_{ki}$ be the computation start time of resolvent $i$ in iteration $k$, $s_k$ be the earliest computation start time across all resolvents in iteration $k$ (the iteration start time), and $e_k$ be the completion time of the final communication in iteration $k$ (the iteration end time).
We then have:
\begin{align}
s_{11} &= 0 \\
s_{kj} &= \max_{i | L_{ji} \ne 0}{s_{ki} + t_i + l_{ij}} \quad &\forall k \in [r], j > i \in [n]  \\
s_{k+1,i} &= \max_{j | W_{ij} \ne 0}{s_{kj} + t_j + l_{ji}} \quad &\forall k \in [r], i \ne j \in [n] \\
s_{k} &= \min_{i \in [n]}{s_{ki}}\quad &\forall k \in [r]\\
e_k &= \max_{i \in [n]}\{s_{ki} + t_i + \max_{j|{W_{ij} \ne 0}} l_{ij}\}\quad &\forall k \in [r]
\end{align}   
We can then define the $k$-iteration average iteration time as $c^k = e^k/k$. $c^1$ is the completion time of the first iteration. 
The following proposition establishes a lower bound on any single iteration time when communication times are constant, which is useful for establishing whether a given design attains this minimum.
\begin{propositionrep}\label{mincycle}
The minimum single iteration time for algorithm \eqref{n_iteration} with computation time $t_i$ for resolvent $i \in [n]$ and constant communication time $l$ between all resolvents has a lower bound of $\max_{i \in [n]}{t_i}+\min_{i \in [n]}{t_i}+2l$.
\end{propositionrep}

\begin{proof}
Assume each resolvent begins its resolvent computation in iteration $k$ with access to $\v{v}^{k-1}$ (otherwise the time will not be the minimum). We want to find a lower bound on the time required to complete the computations in \eqref{n_itr1} and \eqref{n_itr2}. We claim that computing all elements of $\v{v}^{k}$ requires at least two rounds of resolvent computation and at least two rounds of communication between resolvents, all of which must be performed in serial. Indeed, if only one round of parallelized resolvent calculation was required, then all resolvents can be computed in parallel. But this contradicts the structure of $L$, which must have at least one off-diagonal nonzero because $G(Z)$ is connected. Hence there must be at least two rounds of resolvent computation, and at least one round of communication between them due to the communication implied by $\v{L} \v{x}$ in \eqref{n_itr1}. Computing $\v{v}^{k}$ in \eqref{n_itr2} requires an additional round of communication because of the computation of $\v{W}\v{x}$. This additional round of communication cannot be performed in parallel with the final round of resolvent calculation because it requires at least one resolvent computed in the final round to communicate to a resolvent computed in a previous round. If this were not the case then the partition of nodes induced by the set of resolvents which can be computed in parallel in the final round of computation and its complement would partition $G(W)$ into two disconnected components, contradicting the connectedness of $G(W)$ guaranteed by Lemma \ref{Lem:Limitations}. In summary, at least two rounds of resolvent computation and at least two rounds of communication must be performed in serial.

The resolvent which takes the maximum time to compute must be included in one of the rounds of computation. The other round of computation must be nonempty, so it takes at least $\min_{i \in [n]} t_{i}$ time. Because communication times are uniform the two rounds of communication take at least $2 l $ time. Therefore one iteration of algorithm \eqref{n_iteration} takes at least
$$\max_{i \in [n]} t_{i} + \min_{i \in [n]} t_{i} + 2l.$$
\end{proof}

For large number of iterations, the limiting behavior of $c^{k}$ is a useful characterization of algorithm's performance. We define the \textit{algorithm iteration time}, denoted $c^\infty$, as $c^{\infty} = \lim_{k \to \infty} c^k$. One promising approach to attaining minimal $c^{1}$ and $c^{\infty}$ is to use the constraint set $\mathcal{C}$ in \eqref{main_prob} to generate matrices which are known to allow maximally parallel execution. We do so by dividing the resolvents into blocks within which all resolvents execute in parallel.

We define a \textit{$d$-Block} design over $n$ resolvents as follows. Select $d \in \{2, \dots, n\}$. We construct a partition of the resolvents of size $d$ such that each of the $d$ elements of the partition, which we call a \emph{block}, contains consecutive resolvents. That is, if the blocks are of size $m_1, m_2, \dots, m_d$, block 1 contains resolvents 1 to $m_1$, block 2 contains $m_1+1$ to $m_1+m_2$, and so on. For a $d$-Block design, we prohibit edges in $G(Z)$ connecting two resolvents in the same block and in $G(W)$ we only permit edges between resolvent $i$ in block $k$ and resolvent $j$ in block $l$ if $|k-l| \leq 1$.

Care must be taken in the choice of the size of each block to avoid infeasibility in \eqref{main_prob}. Lemma \ref{Lem:Limitations}\ref{Lem:cutset} is particularly helpful for this. Note that the definition of a $d$-Block design means that for any block $\mathcal{S} \subseteq \{1, \dots, n\}$, $\mathcal{E}_S = \emptyset$ in $G(Z)$. An immediate corollary of Lemma \ref{Lem:Limitations}\ref{Lem:cutset} is that a 2-Block design must have $|\mathcal{S}_1| = |\mathcal{S}_2|$. Furthermore, selecting $d$ which divides $n$ and letting the block size $|\mathcal{S}_i| = \frac{n}{d}$ will always satisfy Lemma \ref{Lem:Limitations}\ref{Lem:cutset}. Unless otherwise noted, $d$-Block designs in the remainder of this work will have constant block size of $m=\frac{n}{d}$. 
A 2-Block design is therefore only feasible for even $n$, and partitions the resolvents into two blocks of size $\frac{n}{2}$. The following theorem establishes the minimality of the $2$-Block design algorithm iteration time in the case of constant computation and communication times.

\begin{theoremrep}
  The 2-Block algorithm design across $n$ resolvents with constant computation and communication times has an algorithm iteration time which attains the minimum single iteration time. It is unique among designs which attain the minimum single iteration time in the first iteration. 
\end{theoremrep}
\begin{proof}
In the 2-Block design with uniform computation and communication times, the first block completes its computation and communication in time $t+l$. It is immediately followed by the second block, which also completes its computation and communication in time $t+l$. Its first iteration time is therefore $2t+ 2l$, and it attains the minimum single iteration time in its first iteration. At the conclusion of the first iteration, all required information is available for the second iteration, and the same is true in each subsequent iteration. It therefore has $e^k = k(2t+2l)$, $c^k = 2t+2l$, and $c^\infty = 2t+2l$, so its algorithm iteration time is minimum single iteration time. 

Furthermore, any algorithm which attains the minimum single iteration time in the first iteration with uniform computation and communication times  must have at least two blocks of resolvents by the argument in Proposition \ref{mincycle}. It must also have no more than two blocks, since the minimum single iteration time is $2t+2l$, and there is no opportunity for between-iteration parallelism in a single iteration. Let the blocks be given by index sets $S_1$ and $S_2$. For any block operating in parallel, $Z_{ij} = 0$ for all $i \ne j$ in the block, so $\mathcal{E}_{S_1} = \mathcal{E}_{S_2} = \emptyset$. By Lemma \ref{Lem:Limitations}, we must have $|\mathcal{S}_1| = |\mathcal{S}_2|$. Suppose there is an index $i \in S_1$ such that $i \notin [\frac{n}{2}]$ (i.e. $S_1$ and $S_2$ are not the partition in the 2-Block design). Since $i$ operates in parallel with the other resolvents in $S_1$, $L_{ji} = 0$ for all $j \ne i \in S_1$. Furthermore, since $i$ is in the first block, it must not require resolvent information from $S_2$ for its resolvent computation. That is, $L_{ij} = 0$ for all $j \in S_2$. $i$ would therefore be disconnected in $G(Z)$, violating Lemma \ref{Lem:Limitations}. Therefore $S_1 = \{1 \dots \frac{n}{2}\}$ and $S_2 = \{\frac{n}{2}+1 \dots n\}$, and we have the 2-Block design, which is therefore unique in attaining the minimum single iteration time in the first iteration.
\end{proof}

Although the general $d$-Block design does not attain the minimum single iteration time in its first iteration for $d > 2$, it does attain the bound in the limit. It is also worth noting that the minimum single iteration bound in Proposition \ref{mincycle} can be attained by $d$-Block designs in many cases beyond constant resolvent times. 
Consider, for example, the case where communication times are constant, but the resolvent computation times are not. If it is possible to choose the resolvent ordering and block sizes so that the maximum resolvent computation time across one of the two sets of parallel blocks equals $\min_{i \in [n]} t_i$, the minimum iteration time bound in Proposition \ref{mincycle} is attained by this design. One can show that this condition holds in the case of constant communication time among all resolvents and resolvent computation times which are constant except for a single maximal outlier. In this situation, the minimum iteration time bound is attained regardless of the ordering of the resolvents whenever the $d$-Block design is feasible in \eqref{main_prob}.

Designs which are $d$-Block impose block matrix structure onto $Z$ and $W$. Matrices $Z$ and $W$ which form a $d$-Block design with constant block size have the block matrix form in \eqref{block_design_mat}, where each entry is a matrix in $\Rm$ where $m = n/d$. Each matrix $D^i \in \mathbb{S}^m$ has diagonals which ensure $W\1=0$ and off-diagonal elements permitted to be non-zero. 0 and $\I$ are the zero and identity matrices, and $*$ can be varied as long as symmetry is maintained. We present the results corresponding to $\epsilon = 0$ in \eqref{main_prob}, so each entry on the diagonal of $Z$ equals 2.
\begin{equation}\label{block_design_mat}
W = 
\begin{pmatrix}
D^1 & * & 0 & \cdots &   0    & 0     \\
* & D^2 & * & \cdots &    0    &  0       \\
0 & * & D^{3} & \ddots & 0 & 0 \\
\vdots & \vdots & \ddots &  \ddots  & \vdots & \vdots\\
0   &   0     & 0 & \cdots & D^{n-1}   & *  \\
0   &    0   & 0 & \cdots & *         & D^n
\end{pmatrix}, \quad
Z = 
\begin{pmatrix}
2\I & *   & *   & \cdots  & *   & *  \\
*   & 2\I & *   & \cdots   &   *   & *  \\
*    &  *   &  2\I  &  \ddots  & *  &  *  \\
\vdots & \vdots & \ddots & \ddots & \vdots & \vdots \\
*   & *   & *   & \cdots  & 2\I & *  \\
*   & *   & *   & \cdots  & *   & 2\I
\end{pmatrix}
\end{equation}

In the $2$-Block setting with $\epsilon = 0$, we observe an interesting block generalization of Douglas-Rachford splitting. In the Douglas-Rachford splitting for finding a zero in the sum of $2$ monotone operators we have 
\begin{equation}
W = \begin{pmatrix}
  1 & -1 \\
  -1 & 1 \\
\end{pmatrix}, \quad Z = \begin{pmatrix}
  2&-2\\
  -2&2\\
\end{pmatrix}.
\end{equation}
The following $2$-Block design generalizes Douglas-Rachford for finding a zero in the sum of $n$ operators, where $n$ is divisible by $2$ and the blocks are of equal size,

\begin{equation} \label{2block_fiedler}
W = \begin{pmatrix}
    2\I & -\frac{2}{m} \1 \1^T \\
    -\frac{2}{m} \1 \1^T & 2\I \\
  \end{pmatrix}, \quad Z = \begin{pmatrix}
    2\I&-\frac{2}{m} \1 \1^T \\
    -\frac{2}{m} \1 \1^T &2\I\\
  \end{pmatrix}.
\end{equation}
We also show in the appendix's Proposition \ref{prop:maxfiedler} that the design in \eqref{2block_fiedler} maximizes the sum of second-smallest eigenvalues--the Fiedler values--of $Z$ and $W$ among all $2$-block designs. The Douglas-Rachford splitting was extended to $n$ operators in the Malitsky-Tam algorithm \citep{malitsky2023resolvent}, and the connection to $d$-Block designs also extends. When $n$ is divisible by $d$ and blocks are of constant size $m=n/d$, the design
$$
Z = 
\begin{pmatrix}
2\I                  & -\frac{1}{m} \1 \1^T & 0                    & \dots  & 0 & -\frac{1}{m} \1 \1^T   \\
-\frac{1}{m} \1 \1^T & 2\I                  & -\frac{1}{m} \1 \1^T &        &   & 0                      \\
0                    & -\frac{1}{m} \1 \1^T & 2\I                  &        &   &                        \\
\vdots               &                      &                      & \ddots &   & \vdots                 \\
0                    &                      &                      &        &   & -\frac{1}{m} \1 \1^T   \\
-\frac{1}{m} \1 \1^T & 0                    &                      & \dots  & -\frac{1}{m} \1 \1^T & 2\I \\
\end{pmatrix}  $$

$$
W =   
\begin{pmatrix}
\I                   & -\frac{1}{m} \1 \1^T & 0                    & \dots & 0 & 0                     \\
-\frac{1}{m} \1 \1^T & 2\I                  & -\frac{1}{m} \1 \1^T &       &   & 0                     \\
0                    & -\frac{1}{m} \1 \1^T & 2\I                  &       &   &                       \\
\vdots               &                      &                      & \ddots &  & \vdots                \\
0                    &                      &                      &       &   & -\frac{1}{m} \1 \1^T  \\
0                    & 0                    &                      & \dots & -\frac{1}{m} \1 \1^T & \I \\
\end{pmatrix} 
$$
is a $d$-Block extension of the Malitsky-Tam algorithm. We have observed experimentally that this designs also maximizes the sum of the Fiedler values of $Z$ and $W$ among designs which share this sparsity pattern.

\begin{toappendix}
\begin{proposition}\label{prop:maxfiedler}
Consider the block matrix 
\begin{equation}\label{sparsity_structure}
Z = \begin{pmatrix} 2I & X \\ X^{T} & 2I \end{pmatrix}
\end{equation}
where the matrix $X$ is of size $m_{1} \times m_{2}$. If $m_{1} = m_{2}$ and $c \leq 2$, taking the matrix $X = -\frac{2}{\sqrt{m_{1} m_{2}}} \mathbbm{1}_{m_1} \mathbbm{1}_{m_{2}}^{T}$ in \eqref{sparsity_structure} and $W=Z$ maximizes $\lambda_{2}(Z) + \lambda_{2}(W)$, the sum of the Fiedler values of $Z$ and $W$ over all $Z$ and $W$ which are feasible in \eqref{main_prob}. If $c > 2$ then \eqref{main_prob} is infeasible among $Z$ and $W$ of the form \eqref{sparsity_structure}.
\end{proposition}

\begin{proof} 
We begin by characterizing the eigenvalues of \eqref{sparsity_structure}. Note that $\lambda$ is an eigenvalue of $Z$ when $\mathrm{det}(Z - \lambda I ) = 0$. By the $2 \times 2$ block determinant formula \cite[Equation (0.8.5.2)]{horn2012matrix}, whenever $\lambda \neq 2$ 
\begin{equation} \label{det_formula}
\mathrm{det}(Z - \lambda I) = \mathrm{det}\left((2-\lambda) I\right)\, \mathrm{det}\left((2 - \lambda) I - \frac{1}{2 - \lambda} X^{T} X\right).
\end{equation}
Any $\lambda$ for which this determinant is $0$ yields an eigenvalue of $Z$. Since the determinant is the product of the eigenvalues, the determinant in \eqref{det_formula} is $0$ if and only if $(\lambda^{2} - 4 \lambda + 4) I - X^{T} X$ has a zero eigenvalue, i.e. $\lambda^{2} - 4 \lambda + 4 = \gamma$ for some $\gamma$ which is an eigenvalue of $X^{T} X$. More generally, we see that all eigenvalues $\lambda$ of $Z$ satisfying $\lambda \neq 2$ are of the form $\lambda = 2 \pm \sqrt{\gamma}$. By the spectral theorem, $Z$ has $n$ positive eigenvalues, so each of the $n$ positive eigenvalues is either $2$ or a zero of \eqref{det_formula} of the form $\lambda = 2 \pm \sqrt{\gamma}$.

Now our attention shifts to finding the eigenvalues of $X^{T} X$. Since $X^{T} X$ is positive semidefinite, each $\gamma > 0$. For every $\gamma \neq 0$, we get two $\lambda$ values because $\lambda = 2 \pm \sqrt{\gamma}$. From Proposition \ref{lap_full}, the smallest eigenvalue of $Z$ is zero since $\mathbbm{1}$ is in its null space, so there must be a $\gamma = 4$ and the leading eigenvalue of $Z$ is $\lambda = 4$. Since all $\gamma \geq 0$, the Fiedler value $\lambda_{2}(Z) = 2 - \sqrt{\gamma} \leq 2$. But this bound is attained when $X^{T} X$ has all $\gamma = 0$ except for the $\gamma = 4$ required to make $\lambda_{1}(Z) = 0$. We note that $c > 2$ results in infeasibility for $Z$ of the form \eqref{sparsity_structure} because, combined with the constraints in \eqref{main_prob}, it implies $\lambda_{2}(Z) > 2$, which contradicts the bound $\lambda_{2}(Z) \leq 2$ for matrices of this form.

The matrix $X = -\frac{2}{\sqrt{m_{1} m_{2}}} \mathbbm{1}_{m_{1}} \mathbbm{1}_{m_{2}}^{T}$ yields $X^{T} X = \frac{4}{m_{2}} \mathbbm{1}_{m_{2}} \mathbbm{1}_{m_{2}}^{T}$, which has the required eigenvalues $\gamma \in \{0, 0, \dots, 0, 4\}$ producing eigenvalues $\lambda \in \{0, 2, \dots, 2, 4\}$. However, this choice of $X$ yields a $Z$ which satisfies $Z \mathbbm{1} = 1$ if and only if $m_{1} = m_{2}$. In this setting, the chosen $X$ produces an $X$ which satisfies the conditions of the theorem, so we conclude that it maximizes the Fiedler value over all possible choices of $X$.

In the two block case, the sparsity structure of $Z$ and $W$ are equivalent. Moreover, we have $Z \succeq W$ so the Fiedler value of $W$ is less than or equal to that of $Z$. But taking $W = Z$ leaves a feasible $W$ and attains the bound that $\lambda_{2}(W) = \lambda_{2}(Z)$, so we conclude that this $W$ maximizes the Fiedler value. This shows that the proposed $Z$ and $W$ maximize $\lambda_{2}(Z)$ and $\lambda_{2}(W)$ individually, so they also maximize the sum $\lambda_{2}(Z) + \lambda_{2}(W)$.
\end{proof}
\end{toappendix}

If instead of fixing the $*$ block matrices in \eqref{block_design_mat} to zero we allow them to vary, we can generate new algorithms which go beyond the generalized Douglas-Rachford or Malitsky-Tam structure. For example, in Section \ref{Sec:Convergence}, we provide empirical evidence that optimizing for a variety of spectral properties of $Z$ and $W$ across this larger space generates minimum iteration time algorithms with a convergence rate which is invariant to the number of resolvents in some cases. We also find that $d$-Block maximum Fiedler value designs share the between-block fully connected nature of $Z$ found in the extended Ryu algorithm provided by Tam \citep{tam2023frugal}, but with a path graph connecting blocks in $W$ rather than a star.

If we set $d=2$ for even $n$, but (unlike the generalization of Douglas-Rachford in \eqref{2block_fiedler}) relax the constraint $D_{ij} = 0 \quad \forall i \ne j$, thereby allowing the diagonal blocks of $W$ to vary, we can generate a number of different $2$-Block algorithms using objective functions which optimize various spectral properties of $W$. Section \ref{Sec:Convergence} will show that these tend to outperform other block designs with more blocks (and therefore more sparsity). We examine these objective functions further in the next section.

%% file: objectives.tex
The objective function in \eqref{obj} allows us to select among the feasible matrices defined by constraints \eqref{con1}-\eqref{con8}. We introduce a variety of possible objective functions in this section, beginning first with approaches which operate on the spectra of $Z$ and $W$ to target specific graph characteristics which are commonly used as heuristics in the algorithm design process \cite{colla2024optimal}. We then explore a mixed integer framework and a linearization of the problem which allow optimization over the number and placement of edges, rather than edge weights. We use these formulations to design algorithms which minimize the iteration time given any set of computation and communication times.

%% file: spectral.tex
Problem \eqref{main_prob} admits a number of widely-used convex objective functions which use the spectrum of the graph Laplacian to optimize for specific properties. In this section we introduce several of these objectives, including the algebraic connectivity, the second-largest eigenvalue magnitude (SLEM), the total effective resistance, and the spectral norm of the difference between $Z$ and $W$. Exact formulations are available in the appendix's Section \ref{objective_formulations}.

We begin by examining the maximum sum of the Fiedler values for $G(W)$ and $G(Z)$. This can be obtained using the objective function $\phi(W, Z) = -\beta_W\left(\lambda_1(W)+\lambda_2(W)\right)-\beta_Z\left(\lambda_1(Z)+\lambda_2(Z)\right)$, where $\beta_W \geq 0$ and $\beta_Z \geq 0$. In Section \ref{Sec:GraphTheory}, we note that $\lambda_1(W)=0$, so $\lambda_1(W)+\lambda_2(W)$ gives the Fiedler value of $G(W)$ and provides its algebraic connectivity (and likewise with $Z$). We also note that this is objective function yields a convex problem in \eqref{main_prob}.
In graphs with positive edge weights, selecting edge weights which maximize the Fiedler value has been shown to maximize the mixing rate for a continuous time Markov chain defined over those edges \citep{boyd2006convex}. This has led to its selection as a common heuristic for a number of decentralized algorithms \citep{colla2024optimal}.
Figure \ref{fig:maxfiedler} provides an example of matrices $W$ and $Z$ formed with the maximum Fiedler value objective function over a 5-Block design on 10 resolvents.
\begin{figure*}
    \centering
    \begin{subfigure}[b]{0.475\textwidth}
        \centering
        \includegraphics[width=\textwidth]{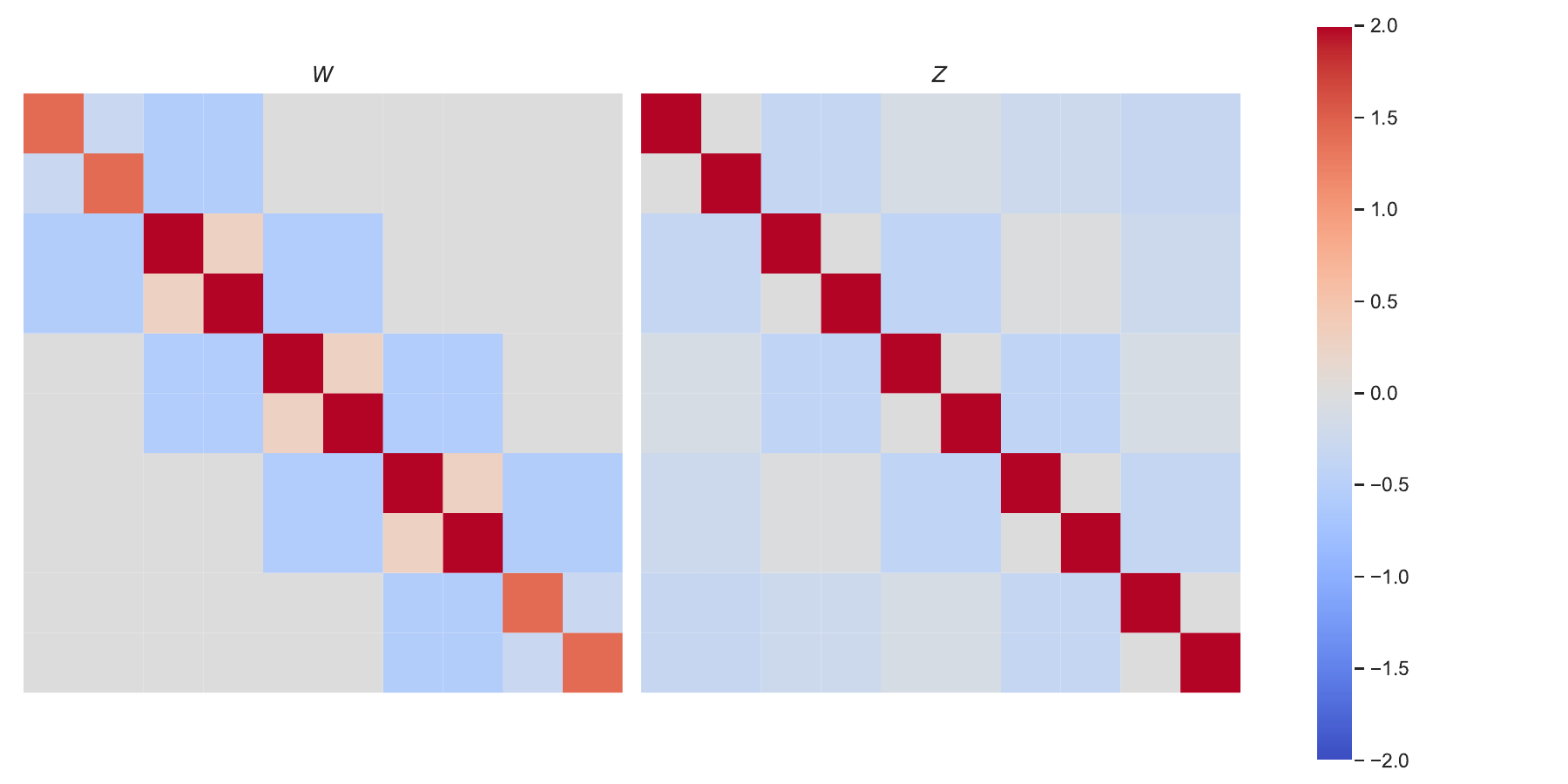}
        \caption[Max Fiedler Value]%
        {{\small Max Fiedler Value}}    
        \label{fig:maxfiedler}
    \end{subfigure}
    \hfill
    \begin{subfigure}[b]{0.475\textwidth}  
        \centering 
        \includegraphics[width=\textwidth]{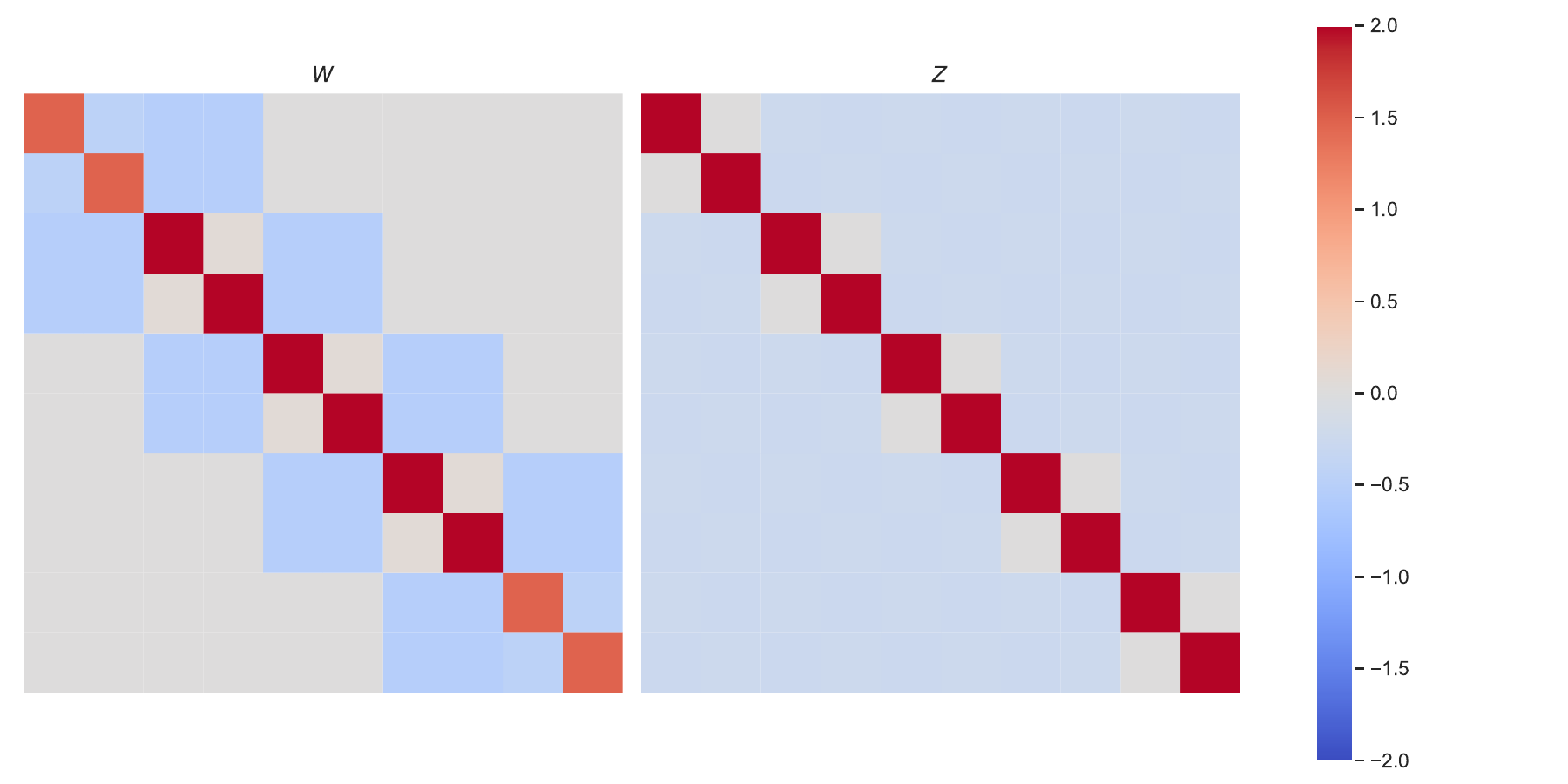}
        \caption[Min SLEM]%
        {{\small Min SLEM}}    
        \label{fig:slem}
    \end{subfigure}
    \vskip\baselineskip
    \begin{subfigure}[b]{0.475\textwidth}   
        \centering 
        \includegraphics[width=\textwidth]{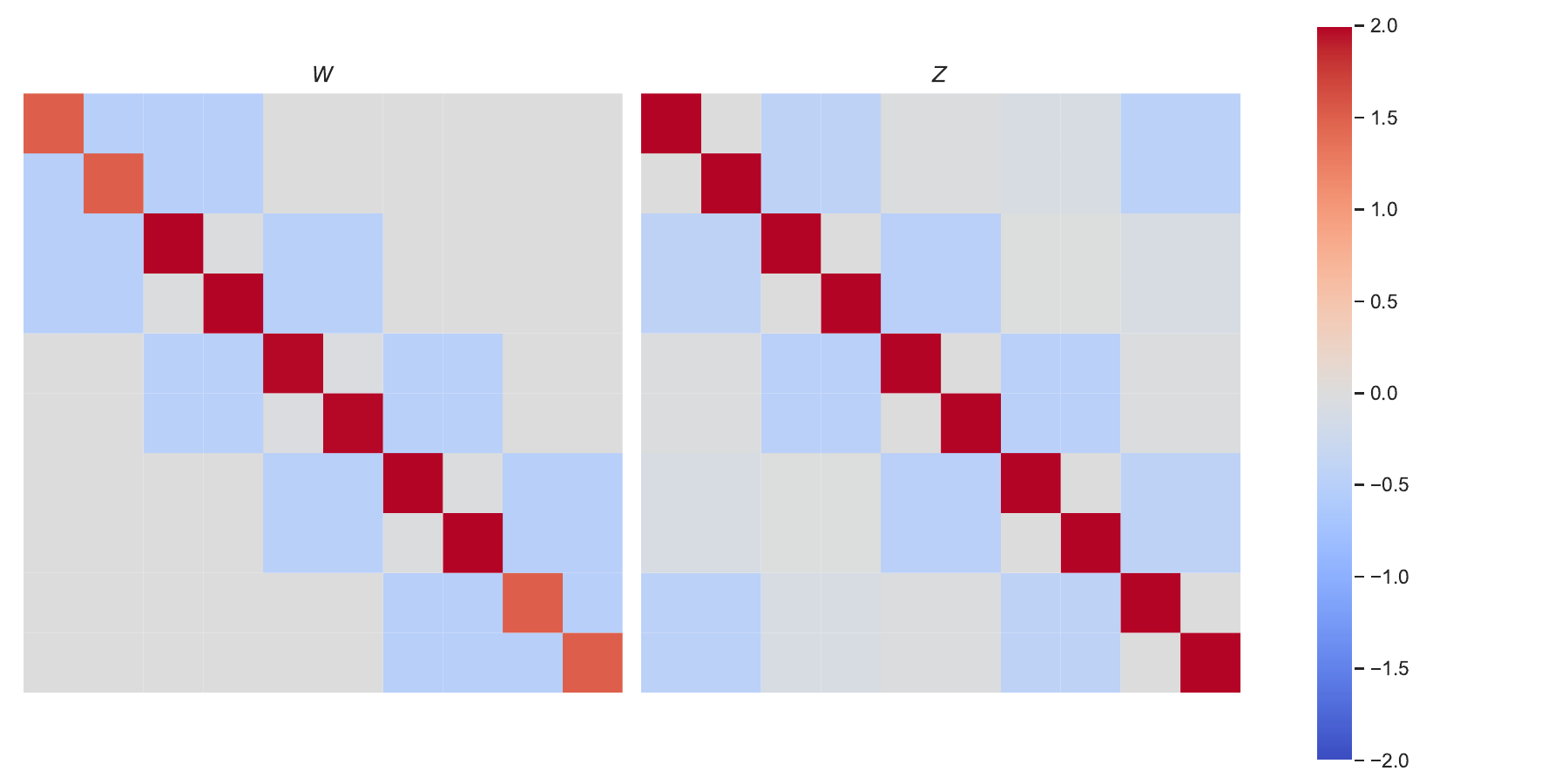}
        \caption[Min Total Effective Resistance]%
        {{\small Min Total Effective Resistance}}    
        \label{fig:ter}
    \end{subfigure}
    \hfill
    \begin{subfigure}[b]{0.475\textwidth}   
        \centering 
        \includegraphics[width=\textwidth]{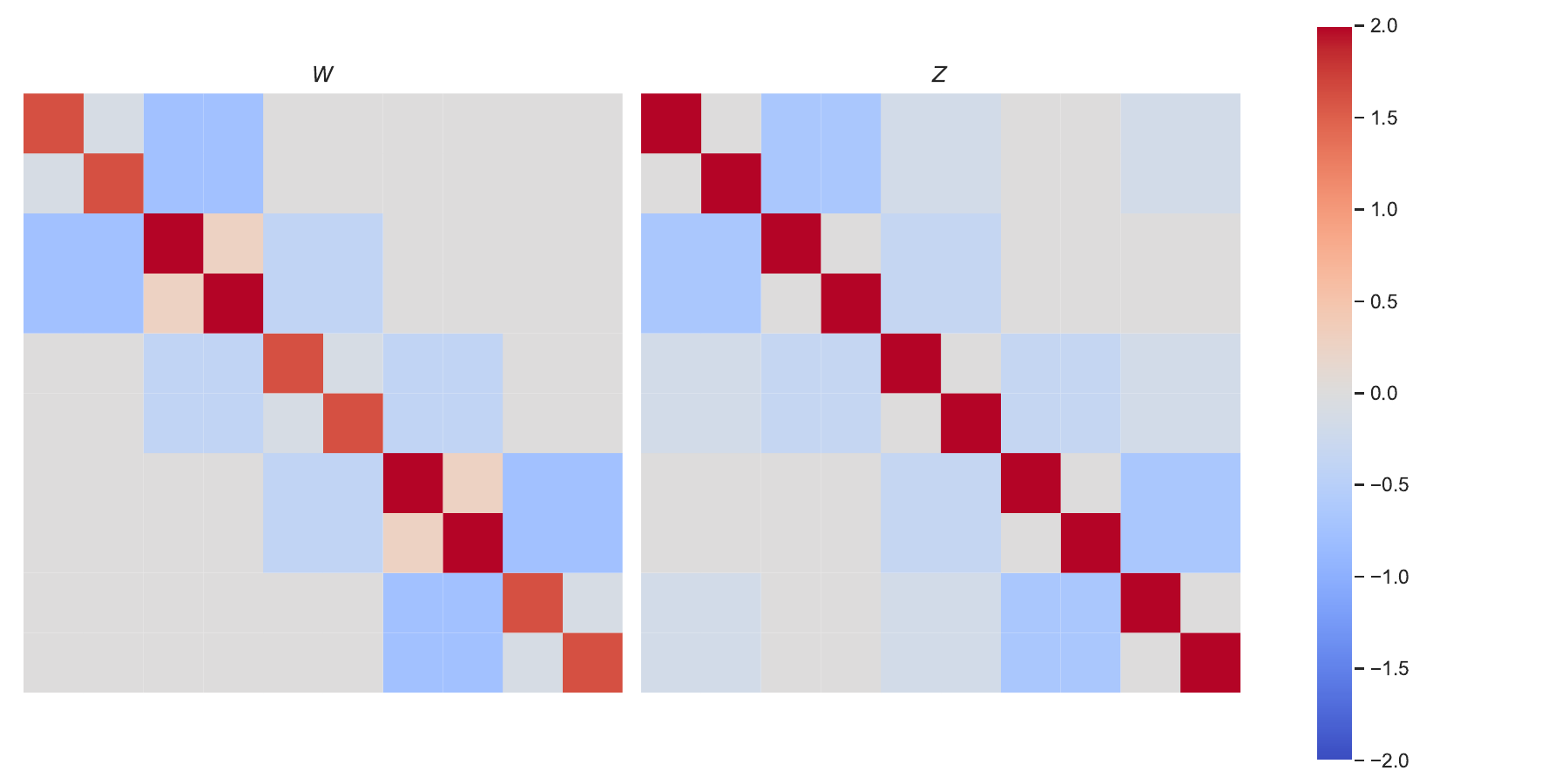}
        \caption[Min $\norm{Z-W}$]%
        {{\small Min $\norm{Z-W}$}}    
        \label{fig:minzw}
    \end{subfigure}
    \caption[5-Block sparsity pattern over 10 resolvents using different spectral objectives]
    {\small 5-Block sparsity pattern over 10 resolvents using different spectral objectives} 
    \label{fig:spectral}
\end{figure*}
  
The SLEM provides another popular heuristic for the design of decentralized algorithm mixing matrices. Given a symmetric Markov Chain, the SLEM is known to maximize its rate of convergence to its equilibrium \cite{boyd2006convex}. If $G(W)$ and $G(Z)$ have only positive edge weights, we can form stochastic matrices $P_w$ and $P_z$ by setting $\epsilon=0$, $P_w = \I - W/2$ and $P_z = \I - Z/2$. The SLEM of $P_w$ is then given by $\max{\left\{\abs{1-\lambda_2(W)}, \lambda_n(W)-1\right\}}$, and similarly for $P_z$, giving us $\phi(W, Z) = \beta_Z\max{\left\{\abs{1-\lambda_2(Z)}, \lambda_n(Z)-1\right\}}+\beta_W\max{\left\{\abs{1-\lambda_2(W)}, \lambda_n(W)-1\right\}}$ \citep{boyd2006convex}. This is a convex function of $W$ and $Z$ which can be minimized for either or both, with a trade-off coefficient determining the weight of each. 
Figure \ref{fig:slem} demonstrates the results of minimizing the SLEM over a 5-Block design on 10 resolvents.
  
Another popular heuristic is the total effective resistance. For any given weighted graph $G(K)$, the total effective resistance of the graph is defined as $\frac{1}{n}\sum_{i=2}^n \frac{1}{\lambda_i(K)}$. For a graph with positive edge weight, the total effective resistance has been shown to minimize the average commute time, over all pairs of vertices, in the random walk on the graph with the transition probability given by the ratio of the edge weights over the node degree for any given node \citep{boyd2006convex}. It is also convex and therefore tractable in \eqref{main_prob} with $\phi(W, Z) = \beta_W\frac{1}{n}\sum_{i=2}^n \frac{1}{\lambda_i(W)} + \beta_Z\frac{1}{n}\sum_{i=2}^n \frac{1}{\lambda_i(Z)}$. Figure \ref{fig:ter} demonstrates the output of this objective. Minimal total effective resistance provides the best convergence rate among the objectives presented here on a wide variety of problems, and demonstrates a convergence rate invariant to problem size in our experiments over any set of $n$ identical $l$-Lipschitz, $\mu$-strongly monotone operators.

The final spectral objective we consider is $\min \norm{Z-W}$ where $\norm{\cdot}$ indicates the spectral norm ($\lambda_n(\cdot)$). This objective has the benefit of balancing the inputs ($\v{v}$ or $-\v{M}^T \v{z}$) to a given resolvent with the resolvent inputs ($\v{x}$). 
For the 2-Block design, $\min \norm{Z-W}$ returns $W=Z$, which also minimizes total resistance subject to these constraints. The minimal spectral norm and minimal total resistance designs are not equal in general, however, as seen in Figures \ref{fig:ter} and \ref{fig:minzw}. 

These objectives highlight the breadth of design options available in \eqref{main_prob}. Many more could be developed. Section \ref{Sec:Convergence} provides a comparison of the various objectives presented here on different problem classes. The maximum Fiedler value and minimum total effective resistance provide particularly promising worst case convergence rate results using the PEP framework.

%% file: edge.tex
A number of valuable algorithmic properties can only be modeled by shifting to a framework which captures whether entries in $Z$ and $W$ are nonzero. 
We therefore introduce binary variables for each of the off-diagonal entries in $Z$ and $W$. 
The sparsity-maximizing formulation we describe in \eqref{min_edge_sdp} provides the most direct application of these binary variables.
The remainder of this subsection describes the use of this mixed integer approach to minimize algorithm iteration time.

%% file: mi_sdp.tex
While the block designs in Section \ref{Sec:block_min} achieve minimal algorithm design iteration time when all resolvent computation times are constant, they do not provide minimal iteration times over arbitrary compute and communication times. 
To minimize the iteration time in this more challenging setting, we find the longest stretch of dependent resolvent calculations (as shown in Figure \ref{fig:activity}) over $r$ iterations of the algorithm. We do so with an MISDP method which is an extension of the critical path method found in project scheduling \cite{kelley1961critical}, though our MISDP approach is more complex than the (MI)LP problems typically found in that literature.
Figure \ref{fig:misdp_gantt} demonstrates the utility of this formulation relative to $d$-Block designs. 
In this example, we show the algorithm execution time for 6-Block, 3-Block, 2-Block, and minimum iteration designs for a problem with a mix of long and short resolvent computation times. In each of the block designs, the alternating sets of adjacent blocks each have a least one long resolvent computation time, resulting in average iteration times which are almost double the minimum possible. The design produced by our MISDP algorithm clearly outperforms the $d$-Block designs, completing 12 iterations in the displayed time interval.

\begin{figure*}
  \centering
  \begin{subfigure}[b]{0.475\textwidth}
      \centering
      \includegraphics[width=\textwidth]{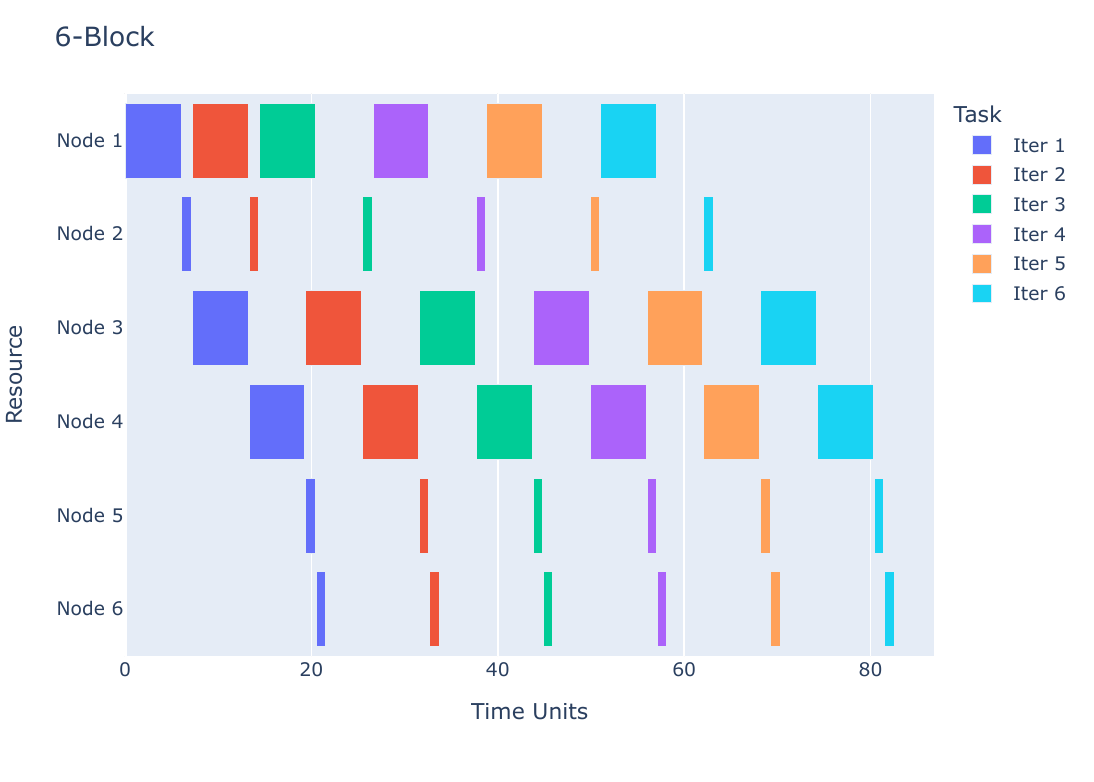}
      \caption[6-Block]%
      {{\small 6-Block}}    
      \label{fig:misdp_6}
  \end{subfigure}
  \hfill
  \begin{subfigure}[b]{0.475\textwidth}  
      \centering 
      \includegraphics[width=\textwidth]{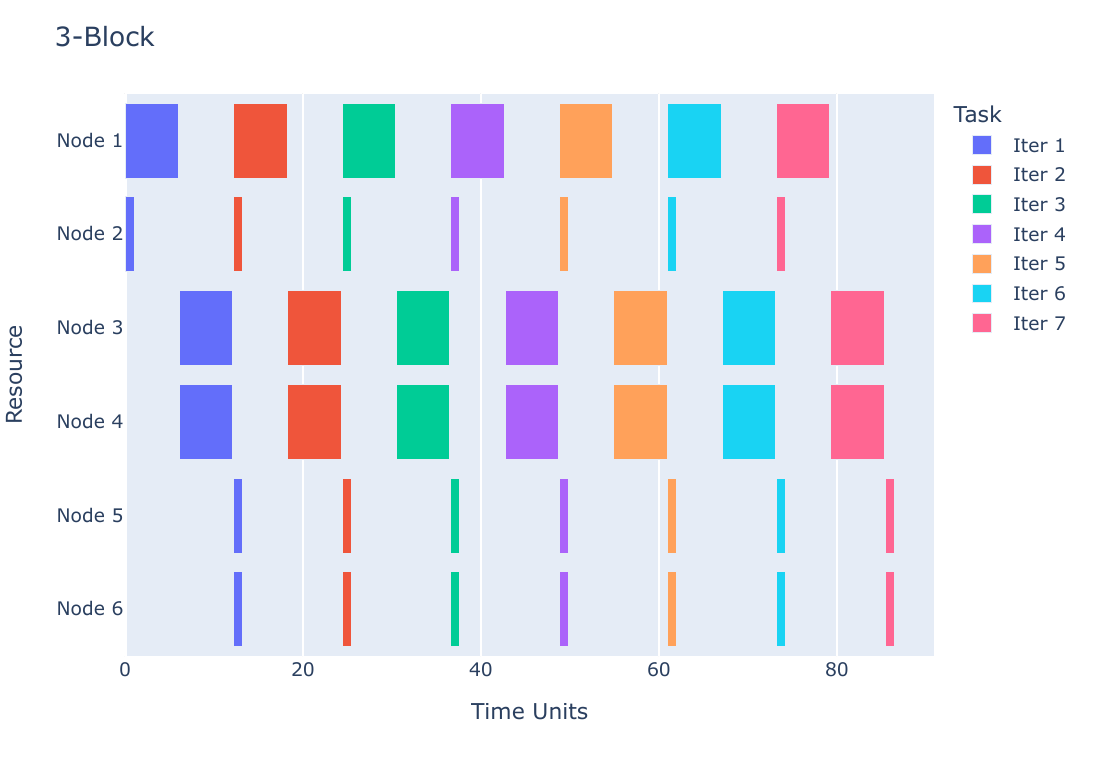}
      \caption[3-Block]%
      {{\small 3-Block}}    
      \label{fig:misdp_3}
  \end{subfigure}
  \vskip\baselineskip
  \begin{subfigure}[b]{0.475\textwidth}   
      \centering 
      \includegraphics[width=\textwidth]{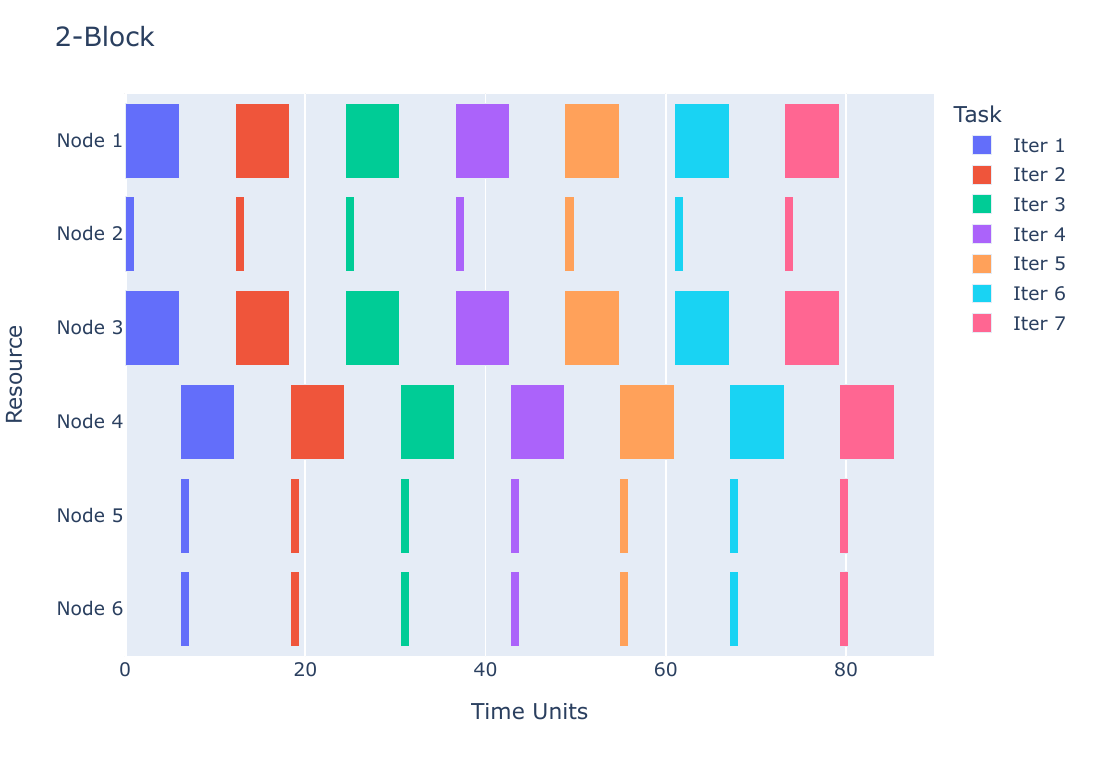}
      \caption[2-Block]%
      {{\small 2-Block}}    
      \label{fig:misdp_2}
  \end{subfigure}
  \hfill
  \begin{subfigure}[b]{0.475\textwidth}   
      \centering 
      \includegraphics[width=\textwidth]{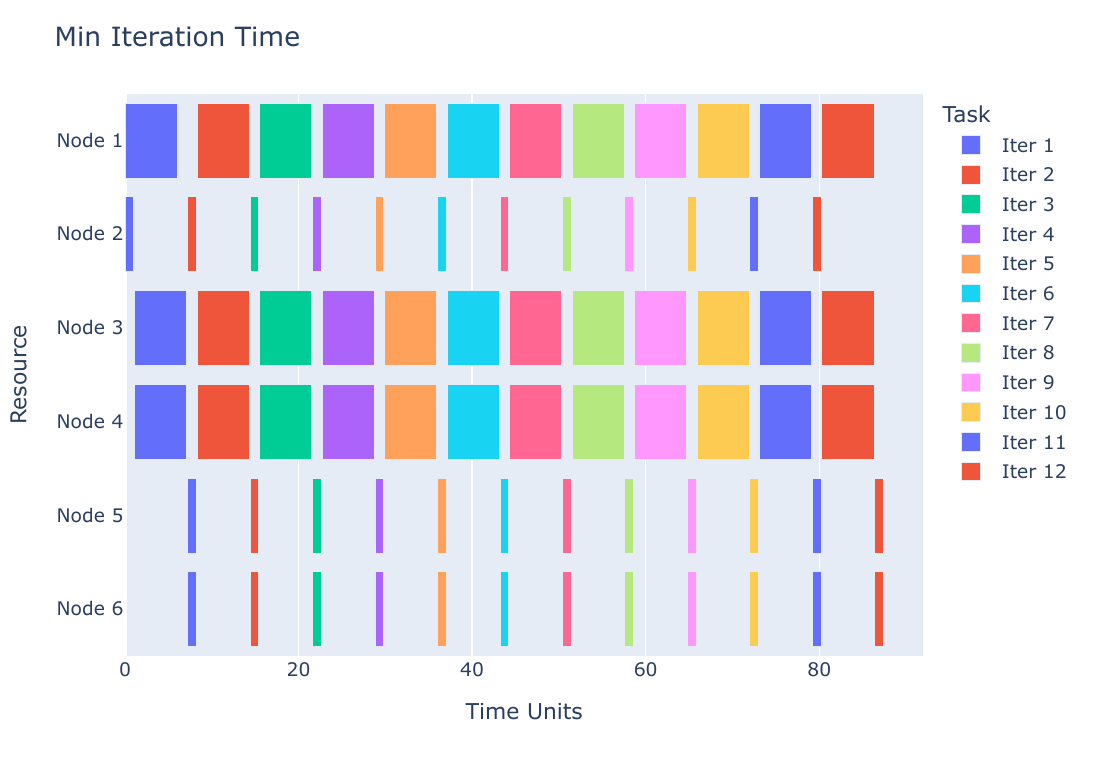}
      \caption[Iteration Time Minimizing \eqref{misdp_mincycle}]%
      {{\small Iteration Time Minimizing \eqref{misdp_mincycle}}}    
      \label{fig:misdp_min_cycle}
  \end{subfigure}
  \caption[$d$-Block and Minimum Iteration Time Design Execution Timelines]
  {\small $d$-Block and Minimum Iteration Time Design Execution Timelines} 
  \label{fig:misdp_gantt}
\end{figure*}

We introduce formalize the MISDP which minimizes the running time of $r$ iterations of \eqref{n_iteration} below. As shown in Figure \ref{fig:misdp_6}, early iterations of the algorithm often have different patterns of completion than later iterations, due to the initial availability of all initialization values, so we recommend selecting $r \geq n$ iterations to model in the critical path. If we have constant communication times, we can also use the lower bound provided in Section \ref{Sec:block_min} to tighten the formulation, which provides a significant performance improvement in the MISDP solution time. Otherwise, we can take this lower bound, which we denote $q$ as $q = \max_i{\{t_i + \min_j{\{l_{ij}\}\}}} + \min_i{\{t_i + \min_j{\{l_{ij}\}}\}}$, which again helps to tighten the formulation.

\textbf{Parameters and Decision Variables}
\begin{notation*}
  r \geq n & integer parameter for number of iterations over which to minimize \\
  t_{i} \geq 0 & parameter for time to compute resolvent $i$\\
  l_{ij} \geq 0 & parameter for time to communicate from $i$ to $j$ \\
  a > 0 & parameter larger than any possible time difference between resolvent start times to relax inactive constraints \\
  q > 0 & parameter for a lower bound on the single iteration time\\
  b \geq 0 & decision variable tracking difference between the run time over $r$ iterations and the lower bound \\
  x_{ij} \in \{0,1\} & decision variable which is 1 if $Z_{ij}$ is nonzero\\
  y_{ij} \in \{0,1\} & decision variable which is 1 if $W_{ij}$ is nonzero\\
  s_{ki} \geq 0 & decision variable for time from algorithm start until completion of resolvent calculation $i$ in iteration $k$ 
\end{notation*}

\begin{subequations}\label{misdp_mincycle}
  \begin{align}
  \min_{x, y, b, s, Z, W} \quad & b \nonumber \\
  \text{s.t.} \quad & s_{kj} - s_{ki} \geq (t_i + l_{ij}+a)x_{ij} - a\quad \forall k \in [r], i < j \in [n] \label{misdp_sin} \\
  & s_{k+1,i} - s_{kj} \geq (t_j + l_{ji}+a)y_{ij} - a \quad \forall k \in [r], i < j \in [n] \label{misdp_soutjk} \\
  & s_{k+1,j} - s_{ki} \geq (t_i + l_{ij}+a)y_{ij} - a \quad \forall k \in [r], i < j \in [n] \label{misdp_soutkj}\\
  & b \geq \max_{i<j\in [n]} \left\{ s_{ri} + t_i + l_{ij}y_{ij} \right\} - r q \label{misdp_cycle_bound} \\
  & \text{SDP constraints \eqref{con1}-\eqref{con7}} \\ %
  & \text{Integer bounding constraints \eqref{misdp_zx}-\eqref{misdp_wy}} \\
  & \text{Optional Cutting Plane Constraints \eqref{misdp_opt_xn}-\eqref{misdp_opt_y1}}
\end{align}
\end{subequations}
Constraint \eqref{misdp_sin} forces the start time $s_{kj}$ of resolvent $j$ in iteration $k$ to fall on or after the completion time ($s_{ki} + t_i + l_{ij}$) of resolvent $i$ in iteration $k$ if $x_{ij} = 1$, and relaxes the constraint otherwise.
Constraints \eqref{misdp_soutjk} and \eqref{misdp_soutkj} ensure that resolvents $i$ and $j$ cannot start iteration $k+1$ until after the completion of each in iteration $k$ when $y_{ij} = 1$, and relaxes the constraint otherwise.
Constraint \eqref{misdp_cycle_bound} sets $b$ as the difference between the run time bound for $r$ iterations ($rq$) and the actual end time for iteration $r$.  
The SDP, integer, and cutting plane constraints are all inherited from \eqref{main_prob} and \eqref{mip_sdp}.
For relatively small problems ($n<10$), this approach can identify iteration time minimizing designs which improve upon the block designs presented in Section \ref{Sec:block_min}.

%% file: mip.tex
We can solve larger minimum iteration time problems by reformulating the MISDP as a Mixed Integer Linear Program (MILP) with a restriction on $W$, $Z$, and $Z-W$ to allow only negative off-diagonal values for each. 
The minimum iteration time edges returned by the MILP may then be used to form a constraint set for the original SDP to achieve specific spectral objectives. The following theorem establishes the connection between the SDP and its restricted linear counterpart.

\begin{theoremrep}\label{linear_sdp}
Any $W$ and $Z$ satisfying the linear constraints \eqref{linear_symmetric}-\eqref{linear_connected} also satisfy the constraints in \eqref{main_prob} for some constant $c > 0$ and $\mathcal{C} \subseteq \Sp^{n} \times \mathbb{S}^{n}$.

\begin{subequations}
\begin{align}
  Z, W &\in \mathcal{S}^n \label{linear_symmetric}\\
  Z_{ij} &\leq W_{ij} \leq 0 \quad \forall i,j \in [n], i\neq j\label{linear_zw_bound}\\
  W \1 &= 0 \label{linear_w}\\
  2 - \epsilon &\leq Z_{11} \leq 2 + \epsilon \label{linear_eps}\\
  \mathrm{diag}(Z) &= Z_{11} \label{linear_z_diag}\\
  Z \1 &= 0 \label{linear_z}\\
   G(W) &\mathrm{ \;is\; connected.}\label{linear_connected}
  \end{align}  
\end{subequations}
\end{theoremrep}
\begin{proof}
  For $i \in [n]$, let $R_i^z = \sum_{i \ne j} \abs{Z_{ij}}$ and $D(Z_{ii},R_i^z) \subseteq \R$ be a closed disc of radius $R_i^z$ centered at $Z_{ii}$. 
  By the Gershgorin Circle Theorem, every eigenvalue of $Z$ lies in at least one such disc. 
  By constraints \eqref{linear_zw_bound}, \eqref{linear_z_diag}, and \eqref{linear_z}, $R_i^z = Z_{11}$ and $D(Z_{ii},R_i^z) = [0, 2Z_{11}]$. Therefore $\lambda_i(Z) \geq 0 \quad \forall i \in [n]$ and $Z \succeq 0$.
Constraints \eqref{linear_zw_bound} and \eqref{linear_w} similarly imply $D(W_{ii},R_i^w) = [0, 2 R_i^w] \quad \forall i \in  \{1 \dots n\}$, and $W \succeq 0$.
  Let $V = Z-W$. Since $Z_{ij} \leq W_{ij}$, $V_{ij} = Z_{ij} - W_{ij} \leq 0$. Constraints \eqref{linear_w}, \eqref{linear_z_diag}, and \eqref{linear_z} set $V_{ii} = Z_{ii} - W_{ii} = -\sum_{i \ne j} Z_{ij} - W_{ij} = -\sum_{i \ne j} V_{ij}$, and $V_{ii} \geq 0$. We also have $R_i^v = -\sum_{i \ne j} V_{ij}$, so $D(V_{ii},R_i^v) = [0, 2 R_i^v] \quad \forall i \in  [n]$, and $V = Z-W \succeq 0$.

  Constraint \eqref{linear_w} directly satisfies $W\1 =0$.

  Constraint \eqref{linear_z} directly satisfies $\1^T Z \1 = 0$.

  If $G(W)$ is a connected graph with positive edge weights, it satisfies $\lambda_2(W) = \lambda_1(W) + \lambda_2(W) > 0$ \cite{fiedler1975property}. We can therefore find a valid value $c = \lambda_2(W) > 0$ for which $W$ is a solution \eqref{main_prob}. 

  Finally, since both $W$ and $Z$ are positive semidefinite, they satisfy the constraints in \eqref{main_prob} for any $\mathcal{C} \subseteq \mathbb{S}_{+}^{n} \times \mathbb{S}^{n}$ which includes $(W, Z)$. 

\end{proof}

Theorem \ref{linear_sdp} means that, along with a constraint on the connectedness of $G(W)$, we can formulate a restricted version of our minimum iteration time MISDP \eqref{misdp_mincycle} as a mixed integer linear program. The restrictions in \eqref{linear_zw_bound} remove from consideration graphs with negative edge weights and graphs in which $G(W)$ has edges which are not present in $G(Z)$. The existing Malitsky-Tam, Ryu and extended Ryu algorithms presented in Section \ref{Sec:Preliminaries} all meet this restriction. 
There are at least two approaches for linearly constraining $W$ to provide connectivity in $G(W)$ when it has positive edge weights. 
The first verifies that each subset $\mathcal{S}$ of the graph's nodes has positive edge weight over its cutset (the set of edges with one node in $\mathcal{S}$ and the other in $\mathcal{S^C}$), so that every subset is connected, and the graph is therefore connected. 
The second verifies a network flow can move from a source node to every other node. The number of constraints in the cutset approach scales exponentially in $n$, whereas the network flow requires a set of additional continuous variables and constraints which scale quadratically in $n$. We provide the flow-based formulation here.

The flow-based formulation uses the additional notation presented below to fix parameters for the source and sink values for the network and define variables for its flow. In the following formulation we eliminate $W$ in the program as a decision variable, defining it after completion as $W_{ij} = y_{ij} Z_{ij}$ for $i \ne j$, and $W_{ii} = - \sum_{j \ne i}W_{ij}$. This restriction on $W$ comes with the benefit of scaling to larger problems.

\begin{notation*}
h \in \R^n & supply/demand parameter placing a supply value of $n-1$ at resolvent 1 ($h_1=n-1$) and a unit demand value at each of the other resolvents ($h_i = -1$)\\ %
f_{ij} \geq 0 \quad \forall i \ne j & decision variable for flow from $i$ to $j$  \\
\end{notation*}

\begin{subequations}\label{min_cycle_fmip}
\begin{align}
  \min_{b, x, y, Z, f, s} \quad & b \nonumber \\
    \text{s.t.} \quad & Z_{ij} \leq 0 \quad \forall i \ne j\label{fmip_z2} \in [n] \\
    & y_{ij} \leq x_{ij} \quad \forall i < j \in [n] \label{fmip_yx}\\
    & -Z_{ij} \leq (2+\epsilon)x_{ij} \quad \forall i < j \in [n] \label{fmip_zlex} \\
    & -Z_{ij} \geq x_{ij}/(n-1) \quad \forall i < j \in [n] \label{fmip_zgex} \\
    & \sum_{j \ne i}f_{ij} - \sum_{j \ne i}f_{ji} = h_i \quad \forall i \in [n]  \label{fmin_flow}\\
    & f_{ij} \leq (n-1) y_{ij} \quad \forall i \ne j  \in [n] \label{fmin_fy} \\
    & \text{Minimum Iteration Time Constraints \eqref{misdp_sin}-\eqref{misdp_cycle_bound}}  \label{fmip_min_cycle} \\
    & \text{Cutting Plane Constraints \eqref{misdp_opt_xn}-\eqref{misdp_opt_y1}} \label{fmip_opt}\\
    & \text{Linear SDP constraints \eqref{con1}, \eqref{con5}} \label{mip_sdp}
  \end{align}
\end{subequations}

Problem \eqref{min_cycle_fmip} uses constraints \eqref{fmin_flow} and \eqref{fmin_fy} to require the existence of a path from resolvent 1 to each of the other resolvents by forcing the flow of the $n-1$ value at 1 only via edges allowed in $y$, thereby requiring connectivity in $y$. 
Constraints \eqref{fmip_yx} and \eqref{fmip_zgex} link the values of $Z_{ij}$ and $x_{ij}$ to $y_{ij}$, forcing them to represent connected graphs when $y$ does. The lower bound in \eqref{fmip_zgex} forces an entry of at least $1/(n-1)$ in $-Z_{ij}$ when $x_{ij}=1$, which is half the value of the weight on the edges of the fully connected design; other small positive values could be used instead.
Setting $W_{ij} = y_{ij}Z_{ij}$ then makes $G(W)$ a connected graph.
Problem \eqref{min_cycle_fmip} therefore satisfies the requirements of Theorem \ref{linear_sdp}, and guarantees that $W$ and $Z$ satisfy the constraints in \eqref{main_prob}. 
We then apply the minimum iteration time constraints in \eqref{misdp_mincycle} for $s$ and $b$ using the integer-valued $x$ and $y$, and the optional cutting planes to tighten the formulation. These modifications allow the MILP in \eqref{min_cycle_fmip} to generate designs for larger problems than \eqref{misdp_mincycle} because it does not require solving difficult MISDPs.

%% file: pep.tex
In this section, we describe the use of the Performance Estimation Problem (PEP) approach described in \citep{drori2014performance} and \citep{ryu2020operator} to optimize the step size $\gamma$ and matrix $W$ with respect to convergence rates of algorithms designed with \eqref{main_prob} across a variety of problem classes. These classes include sums of strongly monotone operators, Lipschitz operators, and the subdifferentials of closed, convex and proper functions, including indicator functions over convex sets, as long as the sum over the operators is guaranteed to have a zero. We note that throughout this section, as in Theorem \ref{main_theorem}, the modulus of strong monotonicity of an operator is permitted to be $0$.

We begin first with an analysis of iteration \eqref{d_iteration}, which provides both the worst-case contraction factor for a fixed $L$, $M$, and $\gamma$ over all initializations and all monotone operators satisfying certain assumptions, as well as the ability to optimize the step size with respect to the contraction factor of $\v{z}$. We then present an analysis of iteration \eqref{n_iteration} which allows us to determine the contraction factor of $\v{v}$ and to optimize $W$ and/or $\gamma$ with respect this contraction factor.

%% file: convergence.tex
In this section we build a PEP which upper bounds the worst-case contraction factor, $\tau = \norm{\v{z}_1^{1} - \v{z}_2^{1}}^2 / \norm{\v{z}_1^{0} - \v{z}_2^{0}}^2$, of \eqref{d_iteration} over all initializations and operators $(A_{i})_{i=1}^{n}$ which are strongly monotone and Lipschitz with respect to some constants $\mu_{i}$ and $l_{i}$, respectively. The first step is to build a PEP which has optimal value which upper bounds $\tau$. When $\mathrm{dim}(\mathcal{H}) \geq d+n$ we show the bound is tight. We then dualize this PEP and note that the primal and dual problems enjoy strong duality. Finally, we note that the dual problem can be parametrized with respect to the step length $\gamma$, so that it can be optimized over in the dual problem, resulting in a value of $\gamma$ which minimizes our bound on the worst-case contraction factor. In what follows, we only introduce and state the dual of the PEP. A detailed derivation of both the primal and dual problems appears in the appendix.

For fixed $M \in \mathbb{R}^{d \times n}$ and $L \in \mathbb{R}^{n \times n}$, define block matrices $K_I$ and (for $i \in [n]$) $K_{\mu_i}$ and $K_{l_i} \in \mathcal{S}^{d+n}$ as:
\begin{align}
K_I &= \begin{bmatrix}
        \I & 0 \\
        0 & 0
    \end{bmatrix} \\
K_{\mu_i} &= \begin{bmatrix}
0 & -\frac{1}{2}M_{\cdot i} e_i^T  \\
-\frac{1}{2}e_i M_{\cdot i}^T & \frac{1}{2}\left(e_i L_{i \cdot} + \left(e_i L_{i \cdot}\right)^T\right)  - (1 + \mu_i)e_{i} e_{i}^{T}  \\
    \end{bmatrix}\\
K_{l_i} &= \begin{bmatrix}
-M_{\cdot i}M_{\cdot i }^T & M_{\cdot i}\left(L_{i\cdot} - e_i^T\right) \\
\left(L_{i\cdot} - e_i^T\right)^T M_{\cdot i}^T & \quad -\left(L_{i\cdot} - e_i^T\right)^T\left(L_{i\cdot} - e_i^T\right) + e_i e_i^T l_i^2
        \end{bmatrix}\\
\end{align}
Also define a function $\tilde{S}: \mathbb{R}^{n} \times \mathbb{R}^{n} \times \mathbb{R} \times \mathbb{R} \to \mathbb{S}^{d+n+d}$ as:
\begin{align}
\tilde{S}(\phi, \lambda, \psi, \gamma) &= \begin{bmatrix}
    - \sum_i \phi_{i} K_{\mu_i} - \sum_i \lambda_{i}K_{l_i} + \psi K_I& \begin{bmatrix}
        \I \\
        \gamma M^T
    \end{bmatrix} \\
    \begin{bmatrix}
        \I & \gamma M
    \end{bmatrix} & \I
\end{bmatrix}
\end{align}

\begin{theoremrep}\label{m_dual_thm}
Let $M \in \mathbb{R}^{d \times n}$, $L \in \mathbb{R}^{n\times n}$, and $(\mu_{i})_{i=1}^{n}$ and $(l_{i})_{i=1}^{n}$ satisfy  $0 \leq \mu_{i} < l_{i}$ for all $i \in [n]$. Consider the problem 
\begin{subequations}\label{pepm_thm} 
    \begin{align} 
    \min_{\phi, \lambda, \psi, \gamma} \quad & \psi \label{pepm_obj}\\
\mathrm{subject\, to} \quad & \tilde{S}\left( \phi, \lambda, \psi, \gamma \right) \succeq 0  \label{pepm_S}\\
& \phi \geq 0, \; \lambda \geq 0\label{pepm_lam}\\
& \phi, \lambda \in \mathbb{R}^{n}\\
& \psi, \gamma \in \mathbb{R}.
\end{align}
\end{subequations}
\begin{enumerate}[(i)]
\item In algorithm \eqref{d_iteration}, if the value of $\gamma$ is provided, then fixing $\gamma$ in \eqref{pepm_thm} and optimizing over the remaining variables provides an optimal value which upper bounds the worst-case contraction factor
$$\tau = \frac{\|\v{z}_{1}^{1} - \v{z}_{2}^{1}\|^{2}}{\|\v{z}_{1}^{0} - \v{z}_{2}^{0}\|^{2}}$$
of \eqref{d_iteration} over all initial values $\v{z}_{1}^{0}$ and $\v{z}_{2}^{0}$ and all possible     $\mu_{i}$-strongly monotone $\l_{i}$-Lipschitz operators $(A_{i})_{i\in[n]}$. When                      $\mathrm{dim}(\mathcal{H}) \geq d+n$, this bound is tight. \label{thm_d_pt1}
\item If $\gamma$ is optimized over, along with the remaining variables in \eqref{pepm_thm}, then this is the choice of $\gamma$ which minimizes the upper bound on the worst-case contraction factor in \eqref{thm_d_pt1}.
\end{enumerate}
\end{theoremrep}
\begin{proof}
    \input{M_PEP.tex}

\end{proof}

We note that $\gamma > 0$ is required in the proof of Theorem \ref{main_theorem} for fixed points of \eqref{d_iteration} to yield a solution of \eqref{zero_in_monotone}. One can enforce this by adding the constraint $\gamma \geq c$ to \eqref{m_dual_thm} for some small $c > 0$.

%% file: M_PEP.tex
Denote by $\mathcal{Q}_{i}$ the set of $\mu_{i}$-strongly monotone and $l_{i}$ Lipschitz operators. A PEP problem for the contraction factor $\tau$ in algorithm \eqref{d_iteration} is

\begin{subequations}\label{pep1}
    \begin{align}
    \max_{\v{z}_1^0, \v{z}_2^0, \v{z}_1^1, \v{z}_2^1, \v{x}_1, \v{x}_2, \v{y}_{1}, \v{y}_{2}, A_{1}, \dots A_{n}} &\quad  \frac{\norm{\v{z}_1^{1} - \v{z}_2^{1}}^2}{\norm{\v{z}_1^{0} - \v{z}_2^{0}}^2} \label{pep1_obj}\\
    \mathrm{subject\, to} &\quad \v{z}_1^{1} = \v{z}_1^{0} + \gamma \v{M} \v{x}_1 \label{pep1z1}\\
     &\quad \v{z}_2^{1} = \v{z}_2^{0} + \gamma \v{M} \v{x}_2 \label{pep1z2}\\
     &\quad x_{1i} = J_{A_i}(y_{1i}) \quad \forall i \in [n]\label{pep1x1}\\
     &\quad y_{1i} = \sum_{j=1}^d -M_{ji} z_{1j}^0 + \sum_{j=1}^n L_{ij} x_{1j} \quad \forall i \in [n]\label{pep1y1}\\
     &\quad x_{2i} = J_{A_i}(y_{2i}) \quad \forall i \in [n]\label{pep1x2}\\
     &\quad y_{2i} = \sum_{j=1}^d -M_{ji} z_{2j}^0 + \sum_{j=1}^n L_{ij} x_{2j} \quad \forall i \in [n]\label{pep1y2}\\
     &\quad \v{z}_{1}^{0}, \v{z}_2^0, \v{z}_1^1, \v{z}_2^1 \in \mathcal{H}^{d}\\
     &\quad \v{x}_1, \v{x}_2, \v{y}_{1}, \v{y}_{2} \in \mathcal{H}^{n}\\
     &\quad A_i \in \mathcal{Q}_i \quad \forall i \in [n].\label{pepclass}
    \end{align}
\end{subequations}

Our first step in the reformulation of \eqref{pep1} is to modify the resolvent evaluation constraints \eqref{pep1x1} and \eqref{pep1x2}. The resolvent calculations \eqref{pep1x1} and \eqref{pep1x2} can be written as constraints requiring that certain points are in graphs of the operators $A_{i}$. In general, a set of points is said to be \emph{interpolable} by a class of operators if there is an operator in the class which has the points in its graph. Proposition 2.4 in \cite{ryu2020operator} gives that the points $\{(x_1, q_1), (x_2, q_2)\}$ are interpolable by the class of $\mu$-strongly monotone and $l$-Lipschitz operators if and only if
\begin{align}
    \left\langle q_1 - q_2, x_1-x_2 \right\rangle \geq \mu \norm{x_1 -x_2}^2 \\
    \norm{q_1 - q_2}^2 \leq l^2 \norm{x_1 -x_2}^2.
\end{align}

We can therefore use this result to write \eqref{pep1} as:
\begin{subequations}\label{pep2}
\begin{align}
\max_{\v{z}_1^0, \v{z}_2^0, \v{z}_1^1, \v{z}_2^1, \v{x}_1, \v{x}_2, \v{y}_{1}, \v{y}_{2}} \quad &
\frac{\norm{\v{z}_1^{1} - \v{z}_2^{1}}^2}{\norm{\v{z}_1^{0} - \v{z}_2^{0}}^2} \label{pep2_obj}\\
\mathrm{subject\, to} \quad &\v{z}_1^{1} = \v{z}_1^{0} + \gamma \v{M} \v{x}_1 \label{pep2z1}\\
& \v{z}_2^{1} = \v{z}_2^{0} + \gamma \v{M} \v{x}_2 \label{pep2z2}\\
& y_{1i} = \sum_{j=1}^d -M_{ji} z_{1j}^0 + \sum_{j=1}^n L_{ij} x_{1j} \quad \forall i \in [n]\label{pep2y1}\\
& y_{2i} = \sum_{j=1}^d -M_{ji} z_{2j}^0 + \sum_{j=1}^n L_{ij} x_{2j} \quad \forall i \in [n]\label{pep2y2}\\
& \left\langle x_{1i} - x_{2i}, y_{1i} - y_{2i} \right\rangle \geq \left(1+\mu_i\right)\norm{x_{1i} - x_{2i}}^2 \quad \forall i \in [n]\label{pep2strong}\\
& l_i^2 \norm{x_{1i} - x_{2i}}^2 \geq \norm{y_{1i} - x_{1i} - \left(y_{2i} - x_{2i} \right)}^2 \quad \forall i \in [n]\label{pep2lip}\\
&\quad \v{z}_{1}^{0}, \v{z}_2^0, \v{z}_1^1, \v{z}_2^1 \in \mathcal{H}^{d}\\
&\quad \v{x}_1, \v{x}_2, \v{y}_{1}, \v{y}_{2} \in \mathcal{H}^{n}.
\end{align}
\end{subequations}

Letting $\v{z} = \v{z}^{0}_{1} - \v{z}^{0}_{2}$, $\v{x} = \v{x}_{1} - \v{x}_{2}$, and $\v{y} = \v{y}_{1} - \v{y}_{2}$, we have: 
\begin{subequations}\label{pep3}
\begin{align}
\max_{\v{z} \in \HH^d, \v{x},\v{y} \in \HH^n} \quad & \frac{\norm{\v{z} + \gamma \v{M} \v{x}}^2}{\norm{\v{z}}^{2}} \label{pep3_obj}\\
\mathrm{subject\, to} \quad & y_{i} = \sum_{j=1}^d -M_{ji} z_j + \sum_{j=1}^n L_{ij} x_{j} \quad \forall i \in [n]\label{pep3y}\\
& \left\langle x_i, y_i \right\rangle \geq \left(1+\mu_i\right)\norm{x_i}^2 \quad \forall i \in [n]\label{pep3strong}\\
& l_i^2 \norm{x_i}^2 \geq \norm{y_i - x_i}^2 \quad \forall i \in [n] \label{pep3lip}
\end{align}
\end{subequations}

Perform the change of variables $\v{z} \to \v{z}/\|\v{z}\|$, $\v{x} \to \v{x}/\|\v{z}\|$, and $\v{y} \to \v{y}/\|\v{z}\|$, and substituting out $y$, this further reduces to:

\begin{subequations}\label{pep4}
\begin{align}
\max_{\v{z} \in \HH^d, \v{x} \in \HH^n} \quad & \norm{\v{z} + \gamma \v{M} \v{x}}^2 \label{pep4_obj}\\
\mathrm{subject\, to} \quad &\norm{\v{z}}^2 = 1 \\
& \left\langle x_i, \sum_{j=1}^d -M_{ji} z_j + \sum_{j=1}^n L_{ij} x_j \right\rangle \geq \left(1+\mu_i\right)\norm{x_i}^2 \quad \forall i \in [n]\label{pep4strong}\\
& l_i^2 \norm{x_i}^2 \geq \norm{\sum_{j=1}^d -M_{ji} z_j + \sum_{j=1}^n L_{ij} x_j - x_i}^2 \quad \forall i \in [n]. \label{pep4lip} 
    \end{align}
\end{subequations}

We then form the Grammian matrix $G \in \Sp^{d+n}$, where
\begin{equation}\label{grammian_def_d}
    G = \begin{bmatrix}
        \norm{z_1}^2 & \langle z_1, z_2 \rangle & \dots & \langle z_1, z_d \rangle & \langle z_1, x_1 \rangle & \langle z_1, x_2 \rangle & \dots & \langle z_1, x_n \rangle \\
        \langle z_1, z_2 \rangle & \norm{z_2}^2 & \dots & \langle z_2, z_d \rangle & \langle z_2, x_1 \rangle & \langle z_2, x_2 \rangle & \dots & \langle z_2, x_n \rangle \\
        \vdots                   &              &       &                          &        \vdots            &                          &       & \vdots \\
        \langle z_1, z_d \rangle & \langle z_2, z_d \rangle  & \dots & \norm{z_d}^2 & \langle z_d, x_1 \rangle & \langle z_d, x_2 \rangle & \dots & \langle z_d, x_n \rangle \\
        \langle z_1, x_1 \rangle & \langle z_2, x_1 \rangle  & \dots & \langle z_d, x_1 \rangle & \norm{x_1}^2 & \langle x_1, x_2 \rangle & \dots & \langle x_1, x_n \rangle \\
        \vdots                   &              &       &                          &       \vdots             &                          &       & \vdots \\
        \langle z_1, x_n \rangle & \langle z_2, x_n \rangle  & \dots & \langle z_d, x_n \rangle & \langle x_1, x_n \rangle & \langle x_2, x_n \rangle & \dots & \norm{x_n}^2
    \end{bmatrix}.
\end{equation}
In what follows, we require $G \succeq 0$. We note that a straightforward extension of \cite[Lemma 3.1]{ryu2020operator} to $n+d$ dimensions gives that, when $\mathrm{dim}(\HH) \geq n+d$, every $G \in \Sp^{n+d}$ is of the form \eqref{grammian_def_d} for some $\v{z}$ and $\v{x}$, and every $G$ of the form \eqref{grammian_def_d} is PSD. It follows that in the sequel when we relax $G$ from the form \eqref{grammian_def_d} to $G \in \Sp^{n+d}$, the relaxation is tight when $\mathrm{dim}(\HH) \geq n+d$. 

For $i \in [n]$, define block matrices $K_O, K_I, K_{\mu_i}, K_{l_i} \in \mathbb{S}^{d+n}$ as follows:
\begin{align}
    K_O &= \begin{bmatrix} \I & \gamma M \\ \gamma M^T & \gamma^2 M^T M \end{bmatrix}\\
K_I &= \begin{bmatrix}
        \I & 0 \\
        0 & 0
    \end{bmatrix} \\
K_{\mu_i} &= \begin{bmatrix}
0 & -\frac{1}{2} M_{\cdot i} e_{i}^{T} \\
-\frac{1}{2} e_{i} M_{\cdot i}^{T} & \frac{1}{2} \left( e_{i} L_{i \cdot} + L_{i \cdot}^{T} e_{i}^{T}\right) - (1 + \mu_{i}) e_{i} e_{i}^{T}
\end{bmatrix}\\
K_{l_i} &= \begin{bmatrix}
-M_{\cdot i}M_{\cdot i }^T & M_{\cdot i}\left(L_{i\cdot} - e_i\right) \\
\left(L_{i\cdot} - e_i\right)^T M_{\cdot i}^T & \quad -\left(L_{i\cdot} - e_i\right)^T\left(L_{i\cdot} - e_i\right) + e_i e_i^T l_i^2
        \end{bmatrix},
\end{align}

We then have the following convex program in $G$, which is equivalent to \eqref{pep4} when $\mathrm{dim}(\mathcal{H}) \geq n+d$ and otherwise is a relaxation,
\begin{subequations}\label{pep5}
    \begin{align}
    \max_{G \in \Sp^{d+n}} \quad & \tr\left(K_O G\right) \label{pep5_obj}\\
    \mathrm{subject\, to} \quad &\tr\left(K_{\mu_i} G\right) \geq 0 \quad \forall i \in [n]\label{pep5mu}\\
    & \tr\left(K_{l_i} G\right) \geq 0 \quad \forall i\in [n] \label{pep5L}\\
    & \tr\left(K_{I} G\right) = 1. \label{pep5I}
    \end{align}
\end{subequations}
Define $S: \mathbb{R}^{n} \times \mathbb{R}^{n} \times \mathbb{R}: \to \mathbb{S}^{d+n}$ as
$$S(\phi, \lambda, \psi) = - K_O - \sum_i \phi_{i} K_{\mu_i} - \sum_i \lambda_{i}K_{l_i} + \psi K_I.$$
For fixed $\gamma$, the dual for problem \eqref{pep5} is:
\begin{subequations}\label{pep6}
    \begin{align}
    \min_{\phi, \lambda, \psi} \quad & \psi \label{pep6_obj}\\
    \mathrm{subject\, to} \quad & S\left(\phi, \lambda, \psi \right) \succeq 0  \label{pep6S}\\
 \quad &\phi \geq 0,\; \lambda \geq 0\\
  \quad & \phi, \lambda \in \mathbb{R}^{n}\\
  \quad & \psi \in \mathbb{R}.
    \end{align}
\end{subequations}
Note that $S\left(\phi, \lambda, \psi \right)$ is the Schur complement of $\tilde{S}(\phi, \lambda, \psi)$, where $\tilde{S}: \mathbb{R}^{n} \times \mathbb{R}^{n} \times \mathbb{R} \to \mathbb{S}^{d + n + d}$ is defined as
\begin{align}
\tilde{S}(\phi, \lambda, \psi) &= \begin{bmatrix}
    - \sum_i \phi_{i} K_{\mu_i} - \sum_i \lambda_{i}K_{l_i} + \psi K_I& \begin{bmatrix}
        \I \\
        \gamma M^T
    \end{bmatrix} \\
    \begin{bmatrix}
        \I & \gamma M
    \end{bmatrix} & \I
\end{bmatrix}.
\end{align}
For fixed $\gamma$, the following SDP is therefore the dual of \eqref{pep5}
\begin{subequations}\label{pep7}
    \begin{align}
    \min_{\phi, \lambda, \psi} \quad & \psi \label{pep7_obj}\\
    \mathrm{subject\, to} \quad & \tilde{S}\left(\phi, \lambda, \psi \right) \succeq 0  \label{pep7S}\\
 \quad &\phi \geq 0,\; \lambda \geq 0\\
  \quad & \phi, \lambda \in \mathbb{R}^{n}\\
  \quad & \psi \in \mathbb{R}.
    \end{align}
\end{subequations}
We next show that problems \eqref{pep7} and \eqref{pep5} are strongly dual to one another by demonstrating Slater's condition \cite{rockafellar1974conjugate} holds. For each $i \in [n]$, select $\varepsilon_{i} > 0$ such that the set of $\mu_{i} + \varepsilon_{i}$ strongly monotone and $l_{i} - \varepsilon_{i}$ Lipschitz operators is nonempty. Choose operators $(A_{i})_{i=1}^{n}$ from each of these sets. Let $\v{z}_{1}^{0}, \v{z}_{2}^{0} \in \mathcal{H}^{d}$ such that $\v{z}_{1}^{0} \neq \v{z}_{2}^{0}$. Run the algorithm for a single iteration with the provided $M \in \mathbb{R}^{d \times n}$, $L \in \mathbb{R}^{n \times n}$, and $\gamma > 0$, and construct $G$ from \eqref{grammian_def_d}, where $\v{z}$ and $\v{x}$ are constructed according to the transformations preceding \eqref{grammian_def_d}. This matrix $G$ is feasible in \eqref{pep5} by construction, and the inequalities \eqref{pep5mu} and \eqref{pep5L} are loose because of the strong monotonicity and Lipschitz constants of the operators $(A_{i})_{i=1}^{n}$. $G$ may not be positive definite, but if it is not then there is a $\delta > 0$ such that 
$$G^{*} = (1- \delta)G  + \frac{\delta}{n} \I$$
is positive definite, feasible with respect to \eqref{pep5I}, and loose with respect to \eqref{pep5mu} and \eqref{pep5L}. This $G^{*}$ is therefore in the relative interior of the feasible set of \eqref{pep5}, so that Slater's condition is satisfied.

Finally, we treat the problem \eqref{pep7} with $\gamma$ as a decision variable, extending the definition of $\tilde{S}$ to treat $\gamma$ as an additional argument. This yields the following, which optimizes the contraction factor bound in problem \eqref{pep5} over $\gamma$,
\begin{subequations}
    \begin{align}
    \min_{\phi, \lambda, \psi, \gamma} \quad & \psi\\
\mathrm{subject\, to} \quad & \tilde{S}\left( \phi, \lambda, \psi, \gamma \right) \succeq 0             \\
& \phi \geq 0, \; \lambda \geq 0\\
& \phi, \lambda \in \mathbb{R}^{n}\\
& \psi, \gamma \in \mathbb{R}.
\end{align}
\end{subequations}

%% file: W_PEP_dual.tex
We now turn to iteration \eqref{n_iteration}. Given a feasible matrix $Z$ in \eqref{main_prob}, we provide a method for finding $\gamma$ and/or $W$ which have the optimal contraction factor $\tau$. This optimal contraction factor is valid over a set of operators defining \eqref{zero_in_monotone} which have certain structural properties that can be characterized using the PEP framework \cite{drori2014performance}.

For fixed $L \in \mathbb{R}^{n \times n}$, we define $K_I$, $K_{\mathbbm{1}}$ and (for $i \in [n]$) $K_{\mu_i}$ and $K_{l_i}$, of each which is in $\mathbb{S}^{2n}$, as:
\begin{align}
    K_I &= \begin{bmatrix}
        \I & 0 \\
        0 & 0
    \end{bmatrix}\\
    K_{\mathbbm{1}} &= \begin{bmatrix}
      \mathbbm{1} \mathbbm{1}^{T} & 0 \\
      0 & 0
    \end{bmatrix}\\
        K_{\mu_i} &= \begin{bmatrix}
    0 & \frac{1}{2} e_{i} e_{i}^{T} \\
    \frac{1}{2} e_{i} e_{i}^{T} & \frac{1}{2} \left( e_{i} L_{i \cdot} + L_{i \cdot}^{T} e_{i}^{T}\right) - \left(1 + \mu_{i}\right) e_{i} e_{i}^{T}
        \end{bmatrix}\\
    K_{l_i} &= \begin{bmatrix}
    -e_i e_i^T & -e_i\left(L_{i\cdot} - e_i\right) \\
    -\left(L_{i\cdot} - e_i\right)^T e_i^T & \quad -\left(L_{i\cdot} - e_i\right)^T\left(L_{i\cdot} - e_i\right) + e_i e_i^T l_i^2
            \end{bmatrix}
\end{align}
We also define a function $\tilde{S}: \mathbb{R}^{n} \times \mathbb{R}^{n} \times \mathbb{R} \times \mathbb{R} \times \mathbb{S}^{n}_{+} \to \mathbb{S}^{3n}$ as
\begin{align}
    \tilde{S}(\phi, \lambda, \psi, \omega, \tilde{W}) &= \begin{bmatrix}
        - \sum_i \phi_{i} K_{\mu_i} - \sum_i \lambda_{i} K_{l_i} + \psi K_I + \omega K_{\mathbbm{1}} & \begin{bmatrix}
            \I \\
            -\tilde{W}
        \end{bmatrix} \\
        \begin{bmatrix}
            \I & -\tilde{W}^T
        \end{bmatrix} & \I
    \end{bmatrix}.
\end{align}

\begin{theoremrep}\label{w_dual_thm}
Let $L \in \mathbb{R}^{n\times n}$, $c > 0$, and $(\mu_{i})_{i=1}^{n}$ and $(l_{i})_{i=1}^{n}$ satisfy $0 \leq \mu_{i} < l_{i}$ for all $i \in [n]$. Consider the problem
\begin{subequations}\label{pepw_thm}
    \begin{align}
    \min_{\phi, \lambda, \psi, \omega, \tilde{W}} \quad & \psi \label{pepw_obj}\\
\mathrm{subject\, to} \quad & \tilde{S}\left( \phi, \lambda, \psi, \omega, \tilde{W} \right) \succeq 0 \label{pepw_S}\\
& \phi \geq 0, \; \lambda \geq 0\label{pepw_lam}\\
& \tilde{W} \1 = 0 \label{null_constraint}\\
& \lambda_{1}(\tilde{W}) + \lambda_{2}(\tilde{W}) \geq c \label{fiedler_constraint}\\
& \phi, \lambda \in \mathbb{R}^{n}\\
& \psi, \omega \in \mathbb{R}\\
& \tilde{W} \in \Sp^n.
\end{align}
\end{subequations}
\begin{enumerate}[(i)]
\item In algorithm \eqref{n_iteration}, if the values of $\gamma$ and $W$ are provided, then fixing $\tilde{W} = \gamma W$ and optimizing over the remaining variables in \eqref{pepw_obj} provides an optimal value of \eqref{pepw_thm} which upper bounds the worst-case contraction factor
$$\tau = \frac{\|\v{v}_{1}^{1} - \v{v}_{1}^{2}\|^{2}}{\|\v{v}_{1}^{1} - \v{v}_{1}^{2}\|^{2}}$$
of \eqref{n_iteration} over all initial values $\v{v}_{1}^{0}$ and $\v{v}_{2}^{0}$ and all possible $\mu_{i}$-strongly monotone $\l_{i}$-Lipschitz operators $(A_{i})_{i\in[n]}$. When $\mathrm{dim}(\mathcal{H}) \geq 2n$, this bound is tight. \label{thm_n_pt1}
\item If $\tilde{W}$ is optimized over, along with the remaining variables in \eqref{pepw_thm}, then setting $\gamma W = \tilde{W}$ in \eqref{n_iteration} is the choice of $\gamma W$ which minimizes the upper bound on the worst-case contraction factor in \eqref{thm_n_pt1} over all matrices $\tilde{W} \in \mathbb{S}_{+}^{n}$ satisfying \eqref{null_constraint} and \eqref{fiedler_constraint}.
\item If $W$ is fixed but $\gamma$ is to be determined, then substituting $\tilde{W} = \gamma W$ in \eqref{pepw_thm} and optimizing over $\gamma$ along with the remaining variables in \eqref{pepw_thm} results in a choice of $\gamma$ which minimizes the upper bound on the worst-case contraction factor in \eqref{thm_n_pt1}.  
\end{enumerate}
\end{theoremrep}
\begin{proof}
Denote by $\mathcal{Q}_{i}$ the set of $\mu_{i}$-strongly monotone and $l_{i}$ Lipschitz operators. Since we arrive at $\v{v} \in \HH^n$ by the change of variable $\v{v} = -\v{M}^T \v{z}$, $\v{v} \in \mathrm{range}(\v{M}^T)$ which gives $\sum_{i=1}^{n} v_{i} = 0$. 
We denote by $\tau = \frac{\norm{\v{v}_1^{1} - \v{v}_2^{1}}^2}{\norm{\v{v}_1^{0} - \v{v}_2^{0}}^2}$. Given the algorithm definition and our performance metric, our worst case $\tau$ is provided by the PEP formulation below:
\begin{subequations}\label{pepw1}
\begin{align}
\max_{\v{v}_1^0, \v{v}_2^0, \v{v}_1^1, \v{v}_2^1, \v{x}_1, \v{x}_2, \v{y}_{1}, \v{y}_{2}, A_{1}, \dots, A_{n}} \quad & \frac{\norm{\v{v}_1^{1} - \v{v}_2^{1}}^2}{\norm{\v{v}_1^{0} - \v{v}_2^{0}}^2} \label{pepw1_obj}\\
\mathrm{subject\, to}\quad & \v{v}_1^{1} = \v{v}_1^{0} - \gamma \v{W} \v{x}_1 \label{pepw1v1}\\
& \v{v}_2^{1} = \v{v}_2^{0} - \gamma \v{W} \v{x}_2 \label{pepw1v2}\\
& x_{1i} = J_{A_i}(y_{1i}) \quad \forall i \in [n]\label{pepw1x1}\\
& y_{1i} = v_{1i}^0 + \sum_{j=1}^n L_{ij} x_{1j} \quad \forall i \in [n]\label{pepw1y1}\\
& x_{2i} = J_{A_i}(y_{2i}) \quad \forall i \in [n]\label{pepw1x2}\\
& y_{2i} = v_{2i}^0 + \sum_{j=1}^n L_{ij} x_{2j} \quad \forall i \in [n]\label{pepw1y2}\\
& \sum_{i=1}^{n}v_{1i}^0 = 0 \label{pepw1v1n}\\
& \sum_{i=1}^{n}v_{2i}^0 = 0 \label{pepw1v2n}\\
& \v{v}_2^0, \v{v}_1^1, \v{v}_2^1, \v{x}_1, \v{x}_2, \v{y}_{1}, \v{y}_{2} \in \mathcal{H}^{n}\\
& A_{i} \in \mathcal{Q}_{i} \quad \forall i \in [n]
\end{align}
\end{subequations}

Similar to the proof of Theorem \ref{m_dual_thm}, we fix $\|\v{v}_{1}^{0} - \v{v}_{2}^{0}\|^{2} = 1$ and rewrite the resolvents using the interpolability constraints for strongly monotone Lipschitz operators.
\begin{subequations}\label{pepw2}
    \begin{align}
    \max_{\v{v}_1^0, \v{v}_2^0, \v{v}_1^1, \v{v}_2^1, \v{x}_1, \v{x}_2, \v{y}_{1}, \v{y}_{2} \in \HH^n} \quad & \norm{\v{v}_1^{1} - \v{v}_2^{1}}^2 \label{pepw2_obj}\\
    \mathrm{subject\, to}\quad & \v{v}_1^{1} = \v{v}_1^{0} - \gamma \v{W} \v{x}_1 \label{pepw2v1}\\
    & \v{v}_2^{1} = \v{v}_2^{0} - \gamma \v{W} \v{x}_2 \label{pepw2v2}\\
    & y_{1i} = v_{1i}^0 + \sum_{j=1}^n L_{ij} x_{1j} \quad \forall i\in[n]\label{pepw2y1}\\
    & y_{2i} = v_{2i}^0 + \sum_{j=1}^n L_{ij} x_{2j} \quad \forall i\in[n]\label{pepw2y2}\\
    & \norm{\v{v}_1^0 - \v{v}_2^0}^2 = 1 \label{pepw2equal}\\
    & \sum_{i=1}^{n}v_{1i}^0 = 0 \label{pepw2v1n}\\
    & \sum_{i=1}^{n}v_{2i}^0 = 0 \label{pepw2v2n}\\
    & \left\langle x_{1i} - x_{2i}, y_{1i} - y_{2i} \right\rangle \geq \left(1+\mu_i\right)\norm{x_{1i} - x_{2i}}^2 \quad \forall i\in [n]\label{pepw2strong}\\
    & l_i^2 \norm{x_{1i} - x_{2i}}^2 \geq \norm{y_{1i} - x_{1i} - \left(y_{2i} - x_{2i} \right)}^2 \quad \forall i\in [n]. \label{pepw2lip} 
    \end{align}
    \end{subequations}
Letting $v_i = v_{1i}^0 - v_{2i}^0$, $x_i = x_{1i} - x_{2i}$, and $y_i = y_{1i} - y_{2i} = v_i + \sum_{j=1}^n L_{ij} x_j$, and substituting to remove $y$, we get the following,
\begin{subequations}\label{pepw4}
    \begin{align}
    \max_{\v{v}, \v{x} \in \HH^n} \quad & \norm{\v{v} - \gamma \v{W} \v{x}}^2 \label{pepw4_obj}\\
    \mathrm{subject\, to} \quad & \norm{\v{v}}^2 = 1 \\
& \left\langle x_i, v_i + \sum_{j=1}^n L_{ij} x_j \right\rangle \geq \left(1+\mu_i\right)\norm{x_i}^2 \quad \forall i \in [n]\label{pepw4strong}\\
& l_i^2 \norm{x_i}^2 \geq \norm{v_i  + \sum_{j=1}^n L_{ij} x_j - x_i}^2 \quad \forall i \in [n] \label{pepw4lip} \\ 
& \sum_{i=1}^{n} v_{i} = 0. \label{pepw4v1n}
    \end{align}
\end{subequations}

Though problem \eqref{pepw4} initially appears to be a relaxation of \eqref{pepw2} because of its elimination of variables and constraints, the problems have the same optimal value. Any feasible point $(\v{v}, \v{x})$ of \eqref{pepw4} yields a feasible point $(\v{v}^{0}_{1}, \v{v}^{0}_{2}, \v{x}^{0}_{1}, \v{x}^{0}_{2})$ of \eqref{pepw2} with the same objective if one takes
\begin{align}
\v{v}_{1}^{0} &= \v{v} \quad \quad \v{v}_{2}^{0} = 0 \\
\v{x}_{1}^{0} &= \v{x} \quad \quad \v{x}_{2}^{0} = 0.
\end{align}

We form the Grammian matrix $G \in \R^{2n \times 2n}$, where
\begin{equation} \label{grammian_def}
    G = \begin{bmatrix}
        \norm{v_1}^2 & \langle v_1, v_2 \rangle & \dots & \langle v_1, v_n \rangle & \langle v_1, x_1 \rangle & \langle v_1, x_2 \rangle & \dots & \langle v_1, x_n \rangle \\
        \langle v_1, v_2 \rangle & \norm{v_2}^2 & \dots & \langle v_2, v_n \rangle & \langle v_2, x_1 \rangle & \langle v_2, x_2 \rangle & \dots & \langle v_2, x_n \rangle \\
        \vdots                   &              &       &                          &        \vdots            &                          &       & \vdots \\
        \langle v_1, v_n \rangle & \langle v_2, v_n \rangle  & \dots & \norm{v_n}^2 & \langle v_n, x_1 \rangle & \langle v_n, x_2 \rangle & \dots & \langle v_n, x_n \rangle \\
        \langle v_1, x_1 \rangle & \langle v_2, x_1 \rangle  & \dots & \langle v_n, x_1 \rangle & \norm{x_1}^2 & \langle x_1, x_2 \rangle & \dots & \langle x_1, x_n \rangle \\
        \vdots                   &              &       &                          &       \vdots             &                          &       & \vdots \\
        \langle v_1, x_n \rangle & \langle v_2, x_n \rangle  & \dots & \langle v_n, x_n \rangle & \langle x_1, x_n \rangle & \langle x_2, x_n \rangle & \dots & \norm{x_n}^2
    \end{bmatrix}.
\end{equation}
We note that a straightforward extension of \cite[Lemma 3.1]{ryu2020operator} to $2n$ dimensions gives that, when $\mathrm{dim}(\HH^{n}) \geq 2n$, every $G \in \mathbb{S}^{2n}$ is of the form \eqref{grammian_def}. It follows that in the sequel when we relax $G$ from the form \eqref{grammian_def} to $G \in \mathbb{S}^{2n}$, the relaxation is tight when $\mathrm{dim}(\HH) \geq 2n$.

We next define $K_I$, $K_O$, $K_{\mathbbm{1}}$ and (for $i \in [n]$) $K_{\mu_i}$ and $K_{l_i}$ as: 
\begin{align}
    K_I &= \begin{bmatrix}
        \I & 0 \\
        0 & 0
    \end{bmatrix}\\
    K_O &= \begin{bmatrix}
        \I & -\gamma W^T \\
        -\gamma W & \gamma^2 W W^T
    \end{bmatrix} = \begin{bmatrix}
        \I \\ -\gamma W 
    \end{bmatrix} \begin{bmatrix}
        \I &
        -\gamma W^T
    \end{bmatrix}\\
    K_{\mathbbm{1}} &= \begin{bmatrix}
      \mathbbm{1} \mathbbm{1}^{T} & 0 \\
      0 & 0
    \end{bmatrix}\\
        K_{\mu_i} &= \begin{bmatrix}
    0 & \frac{1}{2} e_{i} e_{i}^{T} \\
    \frac{1}{2} e_{i} e_{i}^{T} & \frac{1}{2} \left( e_{i} L_{i \cdot} + L_{i \cdot}^{T} e_{i}^{T}\right) - \left(1 + \mu_{i}\right) e_{i} e_{i}^{T}
        \end{bmatrix}\\
    K_{l_i} &= \begin{bmatrix}
    -e_i e_i^T & -e_i\left(L_{i\cdot} - e_i\right) \\
    -\left(L_{i\cdot} - e_i\right)^T e_i^T & \quad -\left(L_{i\cdot} - e_i\right)^T\left(L_{i\cdot} - e_i\right) + e_i e_i^T l_i^2
            \end{bmatrix}.
\end{align}
We then have equivalence between \eqref{pepw2} and the following program when $\mathrm{dim}(\HH) \geq 2n$, and a relaxation (which still provide a valid, if not necessarily tight, bound) when $\mathrm{dim}(\HH) < 2n$:
\setcounter{MaxMatrixCols}{20}
\begin{subequations}\label{pepw5}
    \begin{align}
    \max_{G \in \Sp^{2n}} \quad & \tr\left(K_O G\right) \label{pepw5_obj}\\
    \mathrm{subject\, to} \quad &\tr\left(K_{\mu_i} G\right) \geq 0 \quad \forall i \in [n]\label{pepw5mu}\\
& \tr\left(K_{l_i} G\right) \geq 0 \quad \forall i \in [n]\label{pepw5L}\\
& \tr\left(K_{I} G\right) = 1 \label{pepw5I}\\
& \tr\left(K_{\mathbbm{1}} G \right) = 0 \label{pepw51}.
    \end{align}
\end{subequations}

We next form the dual PEP of problem \eqref{pepw5}. The dual of SDP \eqref{pepw6} is
\begin{subequations}\label{pepw6}
    \begin{align}
    \min_{\phi, \lambda, \psi, \omega} \quad & \psi \label{pepw6_obj}\\
    \mathrm{subject\, to} \quad & S(\phi, \lambda, \psi, \omega) \succeq 0 \label{pepw6S}\\
    &\phi \geq 0, \; \lambda \geq 0\\
    & \phi, \lambda \in \mathbb{R}^{n}\\
    &\psi, \omega \in \mathbb{R}. 
    \end{align}
\end{subequations}
where $S(\phi, \lambda, \psi, \omega) = -K_O - \sum_{i=1}^{n} \phi_{i} K_{\mu_i} - \sum_{i=1}^{n}  \lambda_{i} K_{l_i} + \psi K_I + \omega K_{\mathbbm{1}}$.
We note that given the definition of $K_O$ above, $S$ is the Schur complement of $\tilde{S} \in \Sp^{3n}$, where $\tilde{S}$ is defined as:
\begin{equation}\label{tildeS}
    \tilde{S}(\phi, \lambda, \psi, \omega) = \begin{bmatrix}
        - \sum_i \phi_{i} K_{\mu_i} - \sum_i \lambda_{i} K_{l_i} + \psi K_I + \omega K_{\mathbbm{1}} & \begin{bmatrix}
            \I \\
            -\gamma W
        \end{bmatrix} \\
        \begin{bmatrix}
            \I & -\gamma W^T
        \end{bmatrix} & \I
    \end{bmatrix}.
\end{equation}
For fixed $\gamma W$, the following SDP is therefore the dual of \eqref{pepw1}
\begin{subequations}\label{pepw7}
    \begin{align}
    \min_{\phi, \lambda, \psi, \omega} \quad & \psi \label{pepw7_obj}\\
    \mathrm{subject\, to} \quad & \tilde{S}(\phi, \lambda, \psi, \omega) \succeq 0 \label{pepw7S}\\
    &\phi \geq 0, \; \lambda \geq 0\\
    &\phi, \lambda \in \mathbb{R}^{n}\\
    &\psi, \omega \in \mathbb{R}. 
    \end{align}
\end{subequations}

Strong duality holds between the primal problem \eqref{pepw5} and its dual \eqref{pepw7}. To demonstrate this, we show that Slater's constraint qualification \cite{rockafellar1974conjugate} holds. Since $\mu_{i} < l_{i}$ for each $i \in [n]$, there is an $\varepsilon_{i} > 0$ such that the set of $\mu_{i} + \varepsilon_{i}$-strongly monotone and $l_{i} - \varepsilon_{i}$ Lipschitz operators is nonempty. Choose $A_{i}$ as such an operator for each $i \in [n]$, and choose $\v{v}_{1}^{0}, \v{v}_{2}^{0} \in \mathrm{range}(\v{M}^{T})$  such that $\v{v}_{1}^{0} \neq \v{v}_{2}^{0}$. Run the algorithm \eqref{n_iteration}, constructing the variables $\v{v}$ and $\v{x}$ according to \eqref{pepw7} and the matrix $G$ as in \eqref{grammian_def}. We claim that there is a $\delta > 0$ such that for $(\1 0) \in \R^{2n}$, which is a length $2n$ vectors with $n$ ones followed by $n$ zeroes,
$$G^{*} = (1- \delta) G + \frac{\delta}{n-1}\left(\I-  \frac{1}{n} \begin{pmatrix} \mathbbm{1} \\ 0 \end{pmatrix} \begin{pmatrix} \mathbbm{1}^{T} & 0 \end{pmatrix}  \right)$$
is in the relative interior of \eqref{pepw5}'s feasible region, which we denote by $\mathcal{S}$. By our choice of the $(A_{i})_{i=1}^{n}$ operators, inequalities \eqref{pepw5mu} and \eqref{pepw5L} are loose for $G$ and therefore for $G^{*}$ for small enough $\delta$. Both equality constraints \eqref{pepw5I} and \eqref{pepw51} are satisfied by $G^{*}$. Though $G^{*}$ is positive semidefinite, arguments about the interior of the positive semidefinite cone warrant careful consideration. The constraint \eqref{pepw51} gives that any feasible $G$ is on the boundary of the PSD cone, but we show next that $G^{*}$ is in the relative interior of $\mathcal{S}$. To show that $G^{*} \in \mathrm{relint}(\mathcal{S})$, it suffices to show that there is a $\nu > 0$ such that $G^{*} + M \in \mathcal{S}$ for all $M \in \mathbb{S}^{2n}$ satisfying $\| M \| < \nu$,  $\mathrm{tr}(K_{\mathbbm{1}} M) = 0$, and $\mathrm{tr}(K_{I} M) = 0$. The cyclic property of trace gives that for such an $M$,
$$\mathrm{tr}(K_{\mathbbm{1}} M) = 0 \Rightarrow \begin{pmatrix} \mathbbm{1}^{T} & 0 \end{pmatrix} M \begin{pmatrix} \mathbbm{1} \\ 0 \end{pmatrix} = 0.$$
The spectral theorem then gives that 
\begin{equation}\label{spectral}
M = 0 \begin{pmatrix} \mathbbm{1} \\ 0 \end{pmatrix} \begin{pmatrix} \mathbbm{1} \\ 0 \end{pmatrix}^{T} + \sum_{i=2}^{n} \lambda_{i} u_{i} u_{i}^{T},
\end{equation}
where each of the $u_{i}$ vectors is orthogonal to $\begin{pmatrix} \mathbbm{1} \\ 0 \end{pmatrix}$. We note that the indexing of eigenvalues in \eqref{spectral} does not imply any ordering among them. From this expansion of $M$, we see that $M \begin{pmatrix} \mathbbm{1} \\ 0 \end{pmatrix} = 0$. Finally, to show that $G^{*} + M \in \mathcal{S}$, we note that the equality constraints in \eqref{pepw5} are clearly satisfied and because the inequality constraints are loose for $G^{*}$, they will be loose for $G^{*} + M$ for $\nu$ small enough. To show that $G^{*} + M \succeq 0$, we show that for any $w \in \mathcal{H}^{2n}$, $w^{T}(G^{*} + M) w \geq 0$. For such a $w$, write $w = \alpha \begin{pmatrix} \mathbbm{1} \\ 0 \end{pmatrix} + w_{\perp}$, where $w_{\perp} \perp \begin{pmatrix} \mathbbm{1} \\ 0 \end{pmatrix}$. Then
\begin{align}
w^{T} (G^{*} + M) w = & \alpha^{2} \begin{pmatrix} \mathbbm{1} \\ 0 \end{pmatrix}^{T} \left( (1-\delta) G \right) \begin{pmatrix} \mathbbm{1} \\ 0 \end{pmatrix}\\
& + 2\alpha \begin{pmatrix} \mathbbm{1} \\ 0 \end{pmatrix}^{T} \left((1 - \delta) G \right) w_{\perp} \\
& + w_{\perp}^{T} \left( (1-\delta) G + \frac{\delta}{n-1} \I + M \right) w_{\perp}\\
= & (1-\delta) w^{T} G w + \frac{\delta}{n-1} \|w_{\perp} \|^{2} + w_{\perp}^{T} M w_{\perp}\\
\geq & \|w_{\perp}\|^{2} \left( \frac{\delta}{n-1} + \lambda_{\text{min}}(M)\right).
\end{align}
Therefore there is a $\nu$ small enough such that $G^{*} + M \in \mathcal{S}$. It follows that $G^{*} \in \mathrm{relint}(\mathcal{S})$ and strong duality holds.

Finally, we treat the dual problem \eqref{pepw7} with $\tilde{W}=\gamma W$ as a decision variable. We extend the definition of $\tilde{S}$ to treat $\tilde{W}$ as an argument which yields the following problem. We add the $\tilde{W}\1=0$ and $\lambda_{1}(\tilde{W}) + \lambda_{2}(\tilde{W})$ constraints required in \eqref{main_prob}, noting that these constraints are the properties of $W$ in the proof of Theorem \ref{main_theorem} which guarantee that a fixed point of \eqref{n_iteration} yields a solution of \eqref{zero_in_monotone}.
\begin{subequations}\label{pepw8}
    \begin{align}
    \min_{\phi, \lambda, \psi, \omega, \tilde{W}} \quad & \psi \label{pepw8_obj}\\
    \mathrm{subject\, to} \quad & \tilde{S}(\phi, \lambda, \psi, \omega, \tilde{W}) \succeq 0 \label{pepw8S}\\
    &\phi \geq 0, \; \lambda \geq 0\\
    &\tilde{W} \mathbbm{1} = 0\\
    &\lambda_{1}(\tilde{W}) + \lambda_{2}(\tilde{W}) \geq c\\
    & \phi, \lambda \in \mathbb{R}^{n}\\
    &\psi, \omega \in \mathbb{R}\\
    &\tilde{W} \in \mathbb{S}_{+}^{n}.
    \end{align}
\end{subequations}
\end{proof}

In practice, when optimizing over $\tilde{W}$, we find that when computing the solution of \eqref{pepw_thm} for small values of $c$ the constraint \eqref{fiedler_constraint} is rarely tight, so we recommend solving \eqref{pepw_thm} without \eqref{fiedler_constraint} and only adding it if $\lambda_{2}(\tilde{W})$ is found to be zero (since $\lambda_{1}(\tilde{W}) = 0$ is guaranteed by \eqref{null_constraint}).

%% file: experiments.tex
This section presents numerical experiments which utilize and validate our contributions from the previous sections. We conduct three sets of experiments. The first compares the contraction factor of $d$-Block designs, which compromise between a high degree of parallelism and high resolvent connectivity, with a fully connected design and the Malitsky-Tam algorithm introduced in \cite{malitsky2023resolvent}. We conduct this comparison with and without the optimal step size selection in Theorem \ref{w_dual_thm}, while also examining the impact of resolvent evaluation ordering and the choice of $d$ in $d$-Block designs. Additionally, we numerically investigate the ideal amount of sparsity to balance between parallelism and rapid rate of convergence, demonstrating that $d$-Block designs make the most progress towards a solution during a fixed running time. A final experiment compares $d$-Block designs to minimum iteration time designs from Section \ref{Sec:mip_formulations} when the resolvent computation times and communication times are known. Throughout, in order to compare our contributions to previous results \cite{malitsky2023resolvent, tam2023frugal}, which are formulated in terms of iteration \eqref{d_iteration}, we quantify the contraction factor in iteration \eqref{d_iteration} using Theorem \ref{m_dual_thm}. For $d$-Block designs and minimum iteration time designs from Section \ref{Sec:mip_formulations}, we apply the sparse Cholesky decomposition from Section \ref{Sec:MfromW} to construct an $M$ for use in Theorem \ref{m_dual_thm}. For the sake of reproducibility, all experiments from this section can be found in our accompanying software package \href{https://github.com/peterbarkley/oars/}{github.com/peterbarkley/oars/}.

\subsection{Comparing Contraction Factors of Various Designs}

In our first experiment, we characterize the contraction factor $\tau$ from Theorem \ref{m_dual_thm} for some of the design choices presented in Sections \ref{Sec:Examples} and \ref{Sec:Objectives}. We do so for problems in two problem classes: $n$ identical maximal $l$-Lipschitz $\mu$-strongly monotone operators (which we call Class 1) and $n-1$ identical maximal $l$-Lipschitz $\mu$-strongly monotone operators with one unrestricted maximal monotone operator (which we call Class 2).

%% file: convergence_rate_results.tex
We begin with a comparison of the spectral objective functions in Section \ref{Sec:Objectives} using a Class 1 problem with $l=2$ and $\mu=1$. We restrict the SDP to the 2-Block design. For comparison, we consider the Malitsky-Tam (MT) \cite{malitsky2023resolvent} and a fully connected design. For the fully connected design, we take the matrix $Z$ to have $2$ on its diagonal and $-\frac{2}{n-1}$ in all entries off the diagonal, and we set $W = Z$. For each design, we take the step size $\gamma = 0.5$.

\begin{figure}[h!]
    \begin{center}
        \begin{subfigure}{.49\textwidth}
\centering
\includegraphics[width=\linewidth]{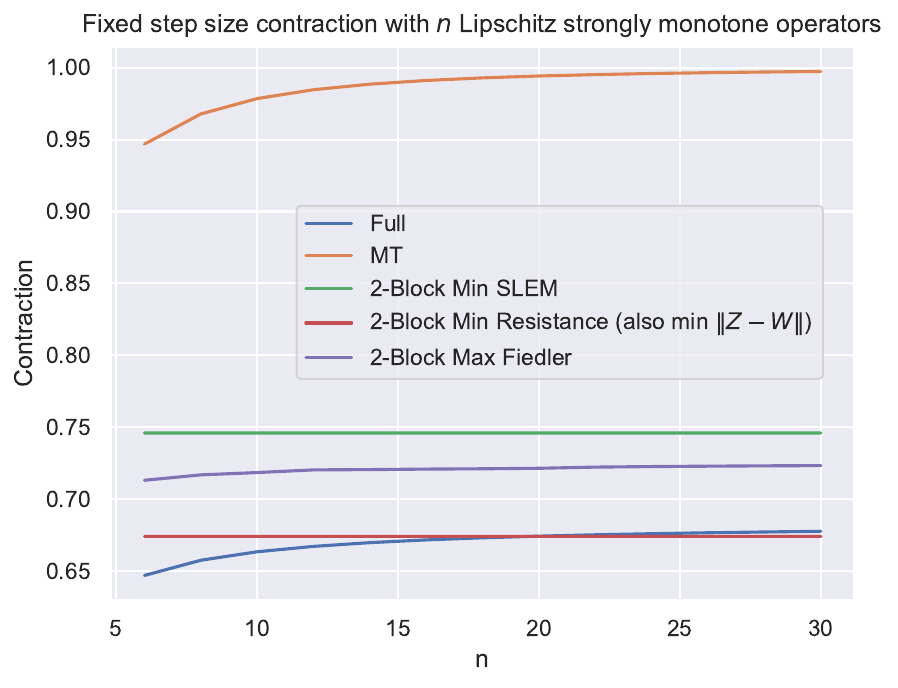}

    \caption{Class 1 problem: All $n$ monotone operators satisfy a $1.0$-strongly monotone and $2.0$-Lipschitz assumption.}
    \label{fig:fixedlipstrong}

\end{subfigure}%
\hspace{.019\textwidth}%
\begin{subfigure}{.49\textwidth}
    \centering
        
\includegraphics[width=\linewidth]{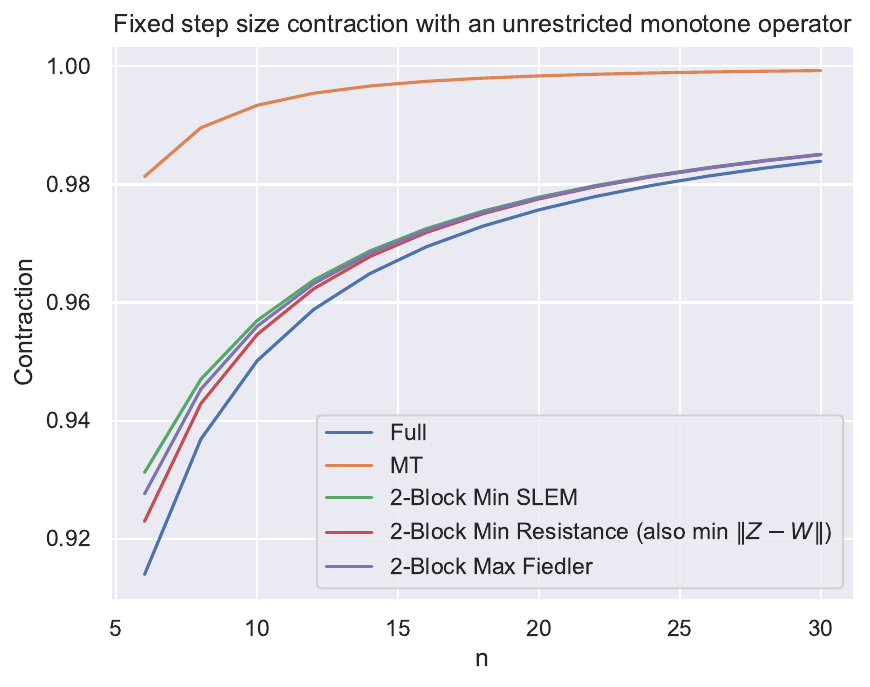}

    \caption{Class 2 problem: $n-1$ monotone operators are $1.0$-strongly monotone and $2.0$-Lipschitz, and one is unconstrained.}
    \label{fig:fixedproj}
\end{subfigure}
\end{center}
\caption{Contraction factors with fixed $\gamma = 0.5$.}
\label{fig:fixed_gamma}
\end{figure}

Figure \ref{fig:fixed_gamma} displays the results. We note first that the maximally dense fully connected design provides the best contraction factor, and the maximally sparse MT design provides the worst. The minimum resistance objective \eqref{resistance_prob} and minimizing $\norm{Z-W}$ return the same matrices and provide the fastest convergence rate among the $2$-Block designs. 

We next conduct a similar test using optimal step sizes. We compare the various objective functions over the 2-Block design with MT and the fully connected design, both for Class 1 (with $l=2$ and $\mu=1$) and for Class 2 with the same parameters. We use the dual PEP formulation in Theorem \ref{m_dual_thm} to calculate the optimal step sizes and contraction factor.

\begin{figure}[h!]
    \begin{center}
        \begin{subfigure}{.49\textwidth}
\centering
\includegraphics[width=\linewidth]{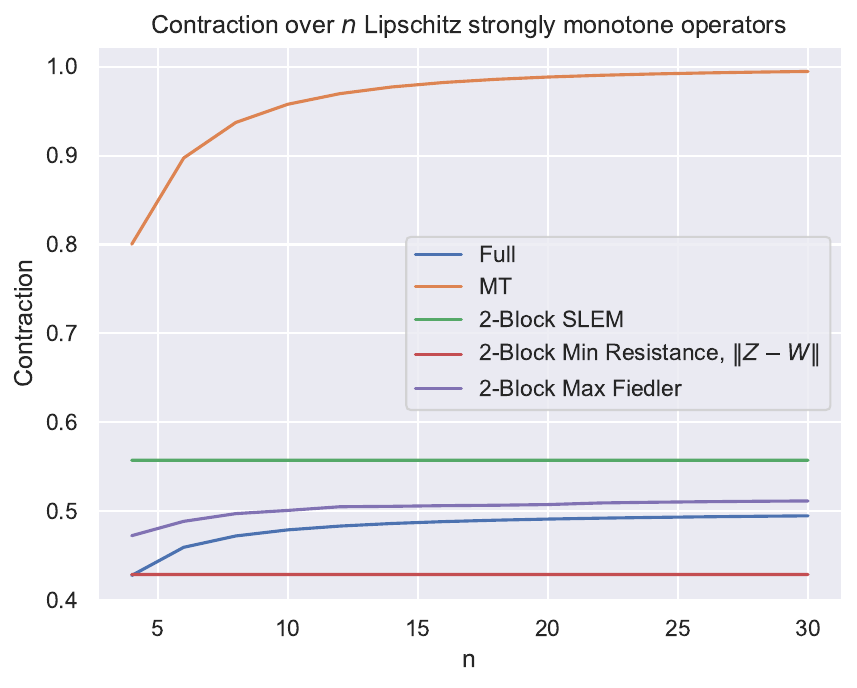}

    \caption{Class 1 problem: All $n$ monotone operators satisfy a $1.0$-strongly monotone and $2.0$-Lipschitz assumption.}
    \label{fig:optlipstrong}

\end{subfigure}%
\hspace{.019\textwidth}%
\begin{subfigure}{.49\textwidth}
    \centering
        
\includegraphics[width=\linewidth]{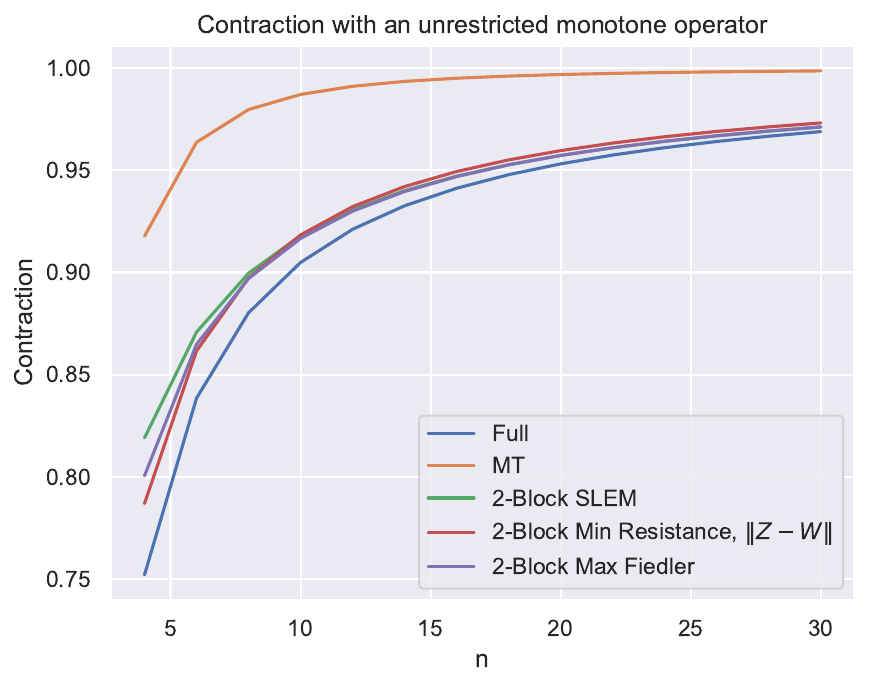}

    \caption{Class 2 problem: $n-1$ monotone operators are $1.0$-strongly monotone and $2.0$-Lipschitz, and one is unconstrained.}
    \label{fig:optproj}
\end{subfigure}
\end{center}
\caption{Contraction factors with optimal $\gamma$ from Theorem \ref{w_dual_thm}.}
\label{fig:opt_gamma}
\end{figure}

Figure \ref{fig:opt_gamma} provides the results. We see a dramatic improvement in convergence rates when using optimal step size in all cases. In the Class 1 case of identical operators, the 2-Block minimum resistance design \eqref{resistance_prob} beats the fully connected algorithm for $n \geq 5$, and the 2-Block maximum Fiedler value algorithm \eqref{fiedler_prob} comes quite close to it as well. Both minimum resistance and minimum SLEM show invariance to the number of operators, $n$, in this scenario.

In Class 2, the invariance disappears, and the number of monotone operators strongly impacts the contraction factor for all designs. The fully connected design provides the best convergence rate, with the spectral objectives approaching the performance of the fully connected design as $n$ increases. Below $n=10$, the minimum resistance design \eqref{resistance_prob} is slightly better than the other spectral designs. Above $n=10$, the maximum Fiedler value \eqref{fiedler_prob} and minimum SLEM designs \eqref{slem_prob} begin to outperform minimum resistance by a small margin.

The matrices returned by the spectral objectives, whether restricted to a $d$-Block design or not, exhibit a high degree of symmetry. Tam's extension of the Ryu Algorithm \cite{tam2023frugal}, however, has a strong lack of symmetry, concentrating significant value in $W$ onto the last resolvent. We therefore examine the impact of operator ordering on the MT, full, and block designs, and Tam's extension of the Ryu algorithm. In this test, we use the optimal step size for each algorithm design, and shift an unrestricted monotone operator between the last position and the first position. The other $n-1$ operators are 2-Lipschitz 1-strongly monotone as before.

\begin{figure}[h!]
    \centering
        
\includegraphics[width=.7\textwidth]{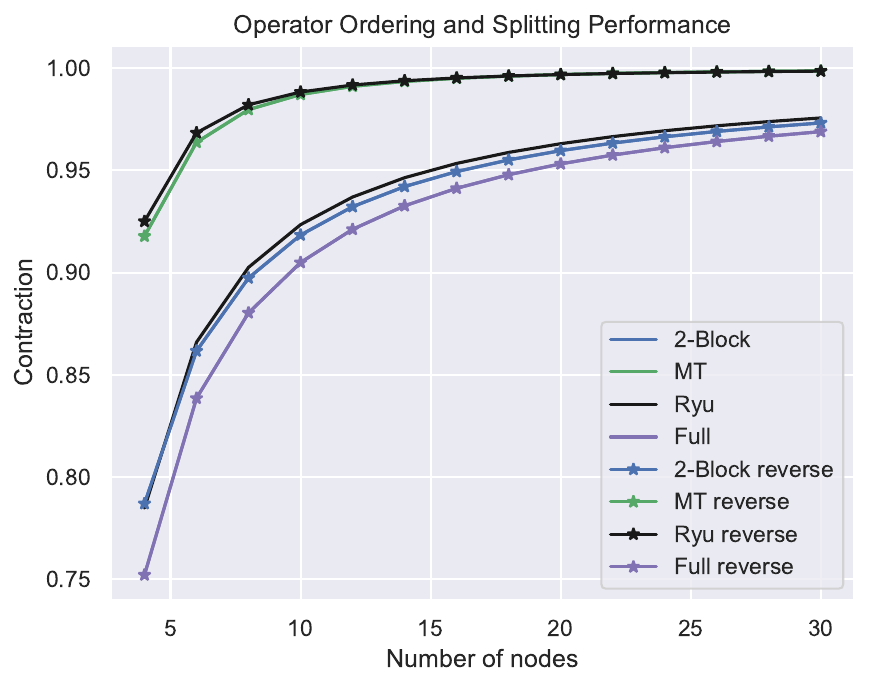}

    \caption{Impact of unrestricted operator placement on convergence rate in Class 2 problems. All designs are invariant to swapping the role of $A_{1}$ and $A_{n}$ except the extended Ryu algorithm, which has degraded performance when the restricted operator is evaluated first.}
    \label{fig:function_order}
\end{figure}

Figure \ref{fig:function_order} provides these results. For most algorithms, the convergence remains the same regardless of ordering, but for the extended Ryu algorithm, the impact is quite strong. With the unrestricted operator in the last position, the extended Ryu design performs similarly to the fully connected algorithm and the 2-Block minimum resistance design \eqref{resistance_prob}, which has the best contraction factor among the $2$-Block designs. But when the unrestricted operator is moved from the last resolvent to the first resolvent becomes poor relative to its competitors, which are invariant to this shift in resolvent evaluation order.

Finally, we test the impact of block size on contraction factor. As above, we test over Class 1 and Class 2 with $l=2$ and $\mu=1$. We use the optimal step size in all cases, and compare the minimum resistance design \eqref{resistance_prob} over all possible constant-size block counts for a given $n$. We include the fully connected and MT designs for reference, which provide $1$ and $n$ blocks, respectively. Figure \ref{fig:block_size} displays the block size results. We observe that increasing to three or more blocks has a strong negative impact on contraction factor, particularly in the identical operator case. For Class 2, the 2-Block design contraction factor approaches that of the 3-Block design near $n=15$.

\begin{figure}[h!]
    \begin{center}
        \begin{subfigure}{.8\textwidth}
\centering
\includegraphics[width=\linewidth]{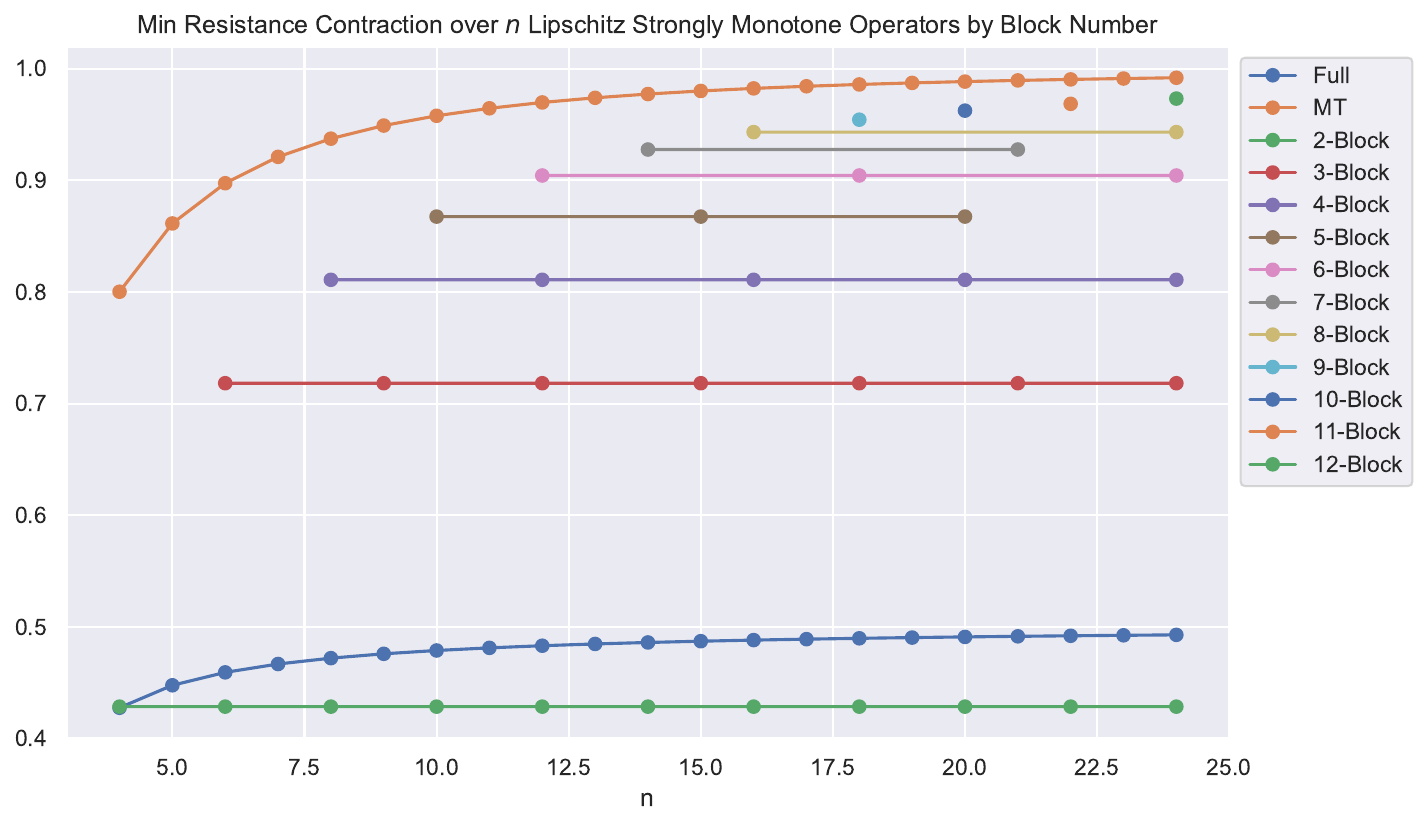}

    \caption{Class 1}
    \label{fig:block_size_lipstrong}

\end{subfigure}
\begin{subfigure}{.8\textwidth}
    \centering
        
\includegraphics[width=\linewidth]{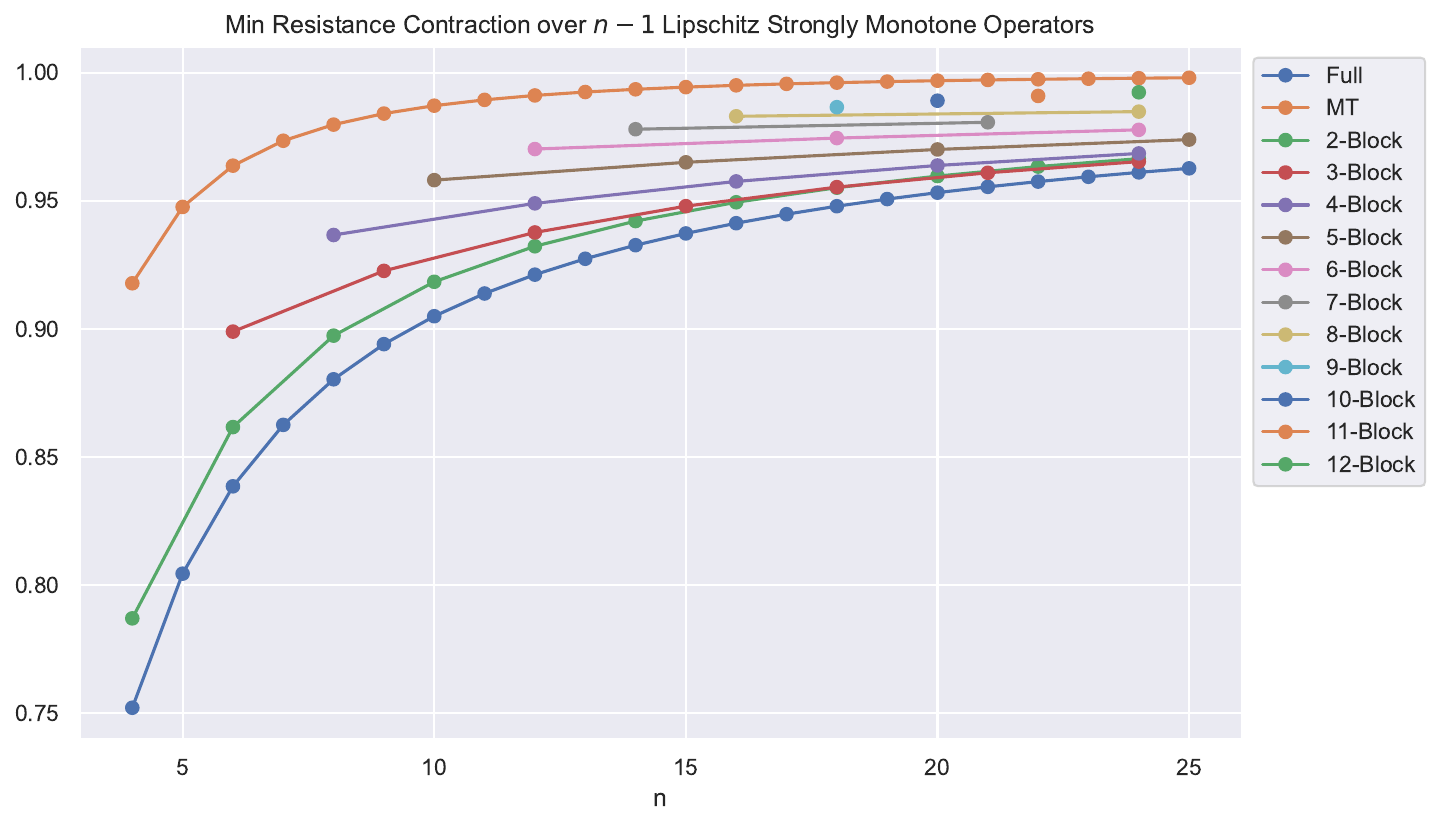}

    \caption{Class 2}
    \label{fig:block_size_proj}
\end{subfigure}
\end{center}
\caption{Convergence rate by block size}
\label{fig:block_size}
\end{figure}

\subsection{The Ideal Amount of Sparsity}

In our next experiment, we analyze the total time required for a contraction of 0.5, given a fixed amount of algorithm running time. This combines the average iteration time of the algorithm designs with their contraction factor in order to determine overall performance. We test across Class 2 with $l=2$ and $\mu=1$, using both fixed and optimal step size, and the maximum Fiedler value design \eqref{fiedler_prob}. We progressively increase matrix sparsity in $Z$ toward the 2-Block design, removing edges alternately from within the first and second block. Figure \ref{fig:edgeremoval} illustrates the edge removal ordering. After reaching the 2-Block design, we remove edges progressively from the remaining block matrices beginning near the diagonal. All resolvent computations and communications are treated as a constant single time unit throughout.

Figure \ref{fig:cycletime} displays the results for $n=24$. In Figure \ref{fig:pareto}, moving from right to left we see the gradual decrease in average iteration time as edges are removed toward the $2$-Block design. Once the 2-Block design is reached, continuing to remove edges shifts the convergence rate upwards without any further decrease in average iteration time. This point is the phase transition of each of the parametrized curves in Figure \ref{fig:cycletime}.
Figure \ref{fig:optimalsparsity} shows the decrease in total convergence time as edges are removed. We observe a strong inflection point at the 2-Block design, with total time rising gradually thereafter as more edges are pruned. These results demonstrate the utility of constructing designs within the $d$-Block constraints.

\begin{figure}[h!]
    \begin{center}
        \begin{subfigure}{.5\textwidth}
            \centering
                
        \includegraphics[width=\linewidth]{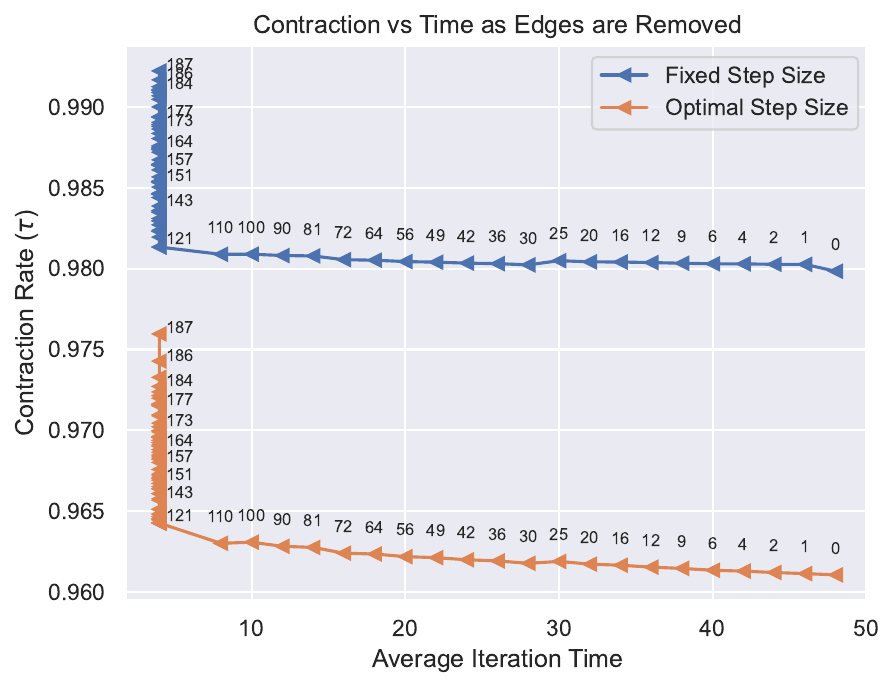}
        
            \caption{Convergence vs iteration time}
            \label{fig:pareto}
        \end{subfigure}%
        \begin{subfigure}{.5\textwidth}
\centering
\includegraphics[width=\linewidth]{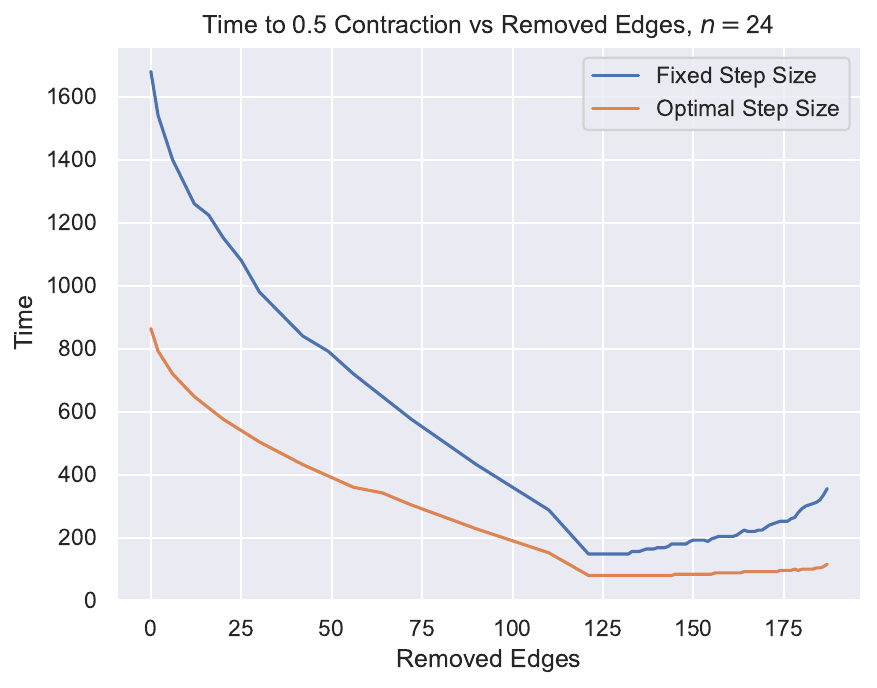}

    \caption{Time to half contraction vs sparsity}
    \label{fig:optimalsparsity}

\end{subfigure}
\end{center}
\caption{Impact of convergence and average iteration time on total algorithm time across $(n-1)$ Lipschitz strongly monotone operators with one monotone operator as sparsity increases to 2-Block design and beyond}
\label{fig:cycletime}
\end{figure}

\begin{figure}[h!]
    \begin{center}

        \includegraphics[width=.7\linewidth]{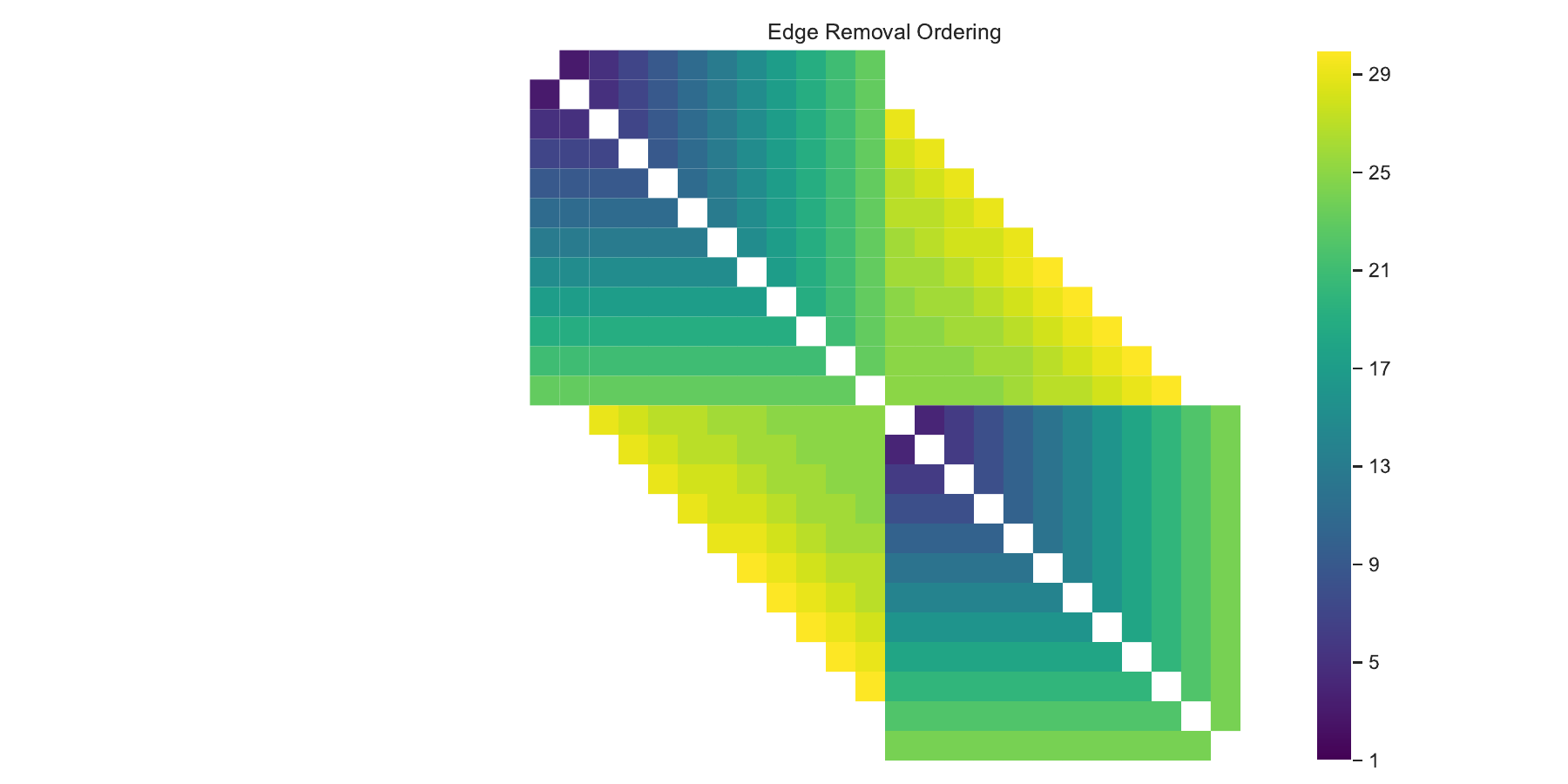}

    \end{center}
    \caption{In Figure \ref{fig:cycletime}, edges are alternately removed from block matrices (1,1) and (2,2) starting in the upper left off-diagonal positions and progressively removing edges to the right and down. Once block matrices $(1,1)$ and $(2,2)$ are restricted to identities, we then prune within the remaining blocks ($(1,2)$ and $(2,1)$, which are equivalent by symmetry), and proceed in the top right of block matrix $(2,1)$ from the diagonal of the entire matrix working outward.}
    \label{fig:edgeremoval}
\end{figure}

%% file: compare_mincycle.tex
\subsection{Minimum Iteration Time and $d$-Block Designs}
We now turn our attention to problems with arbitrary resolvent computation and communication times. Although the $d$-Block designs for small $d$ consistently show strong performance in terms of contraction factor and average iteration time, it is possible in some cases to reduce the total time required for convergence by using the minimum iteration time formulations \eqref{misdp_mincycle} and \eqref{min_cycle_fmip}.

Achieving competitive performance of formulations \eqref{misdp_mincycle} and \eqref{min_cycle_fmip} with $d$-Block designs requires the addition of another constraint---a minimum number of nonzero entries for each row in $W$. From a graph theoretic perspective, this provides a minimum number of edges for each node in $G(W)$. Experiments with various constraint levels have shown that requiring at least $\floor{\frac{n}{2}}$ entries in each row preserves sufficient density in $W$ to keep the convergence rate comparable with the 2-Block design, so we subject our time minimizing designs to an additional constraint requiring $\floor{\frac{n}{2}}$ nonzero entries in each row throughout this experiment. For even $n$, we compare the performance of a time minimizing design with a $2$-Block design, and for odd $n$ we use a $3$-Block design with block sizes $\left(\frac{n-1}{2}, \frac{n-1}{2}, 1\right)$, which we found to perform best across the set of size $3$ partitions. We construct the $d$-Block designs using the minimum total resistance objective function \eqref{resistance_prob}, and we construct the minimum iteration time design using the scalable MILP formulation in problem \eqref{min_cycle_fmip}, which we formulate in Pyomo and solve with Gurobi \cite{bynum2021pyomo,gurobi}.

We assess the performance of each design by computing the total time required to achieve $0.01$ contraction factor using the single-iteration contraction factor $\tau$ from Theorem \ref{m_dual_thm} with optimal step size. We use a Class 2 problem set with $l=2$ and $\mu=1$ as before, and conduct 40 trials each for $n \in \{6,7,8,9\}$, selecting the resolvent computation times uniformly from [0.5,2.0] and the communication times uniformly from [1,11].
Figure \ref{fig:con_mincycle} provides the results. While the 2-Block design remains quite good for the even $n$, for odd $n$ the minimum iteration time results provide consistently better performance. 

\begin{figure}
    \centering        
    \includegraphics[width=.8\textwidth]{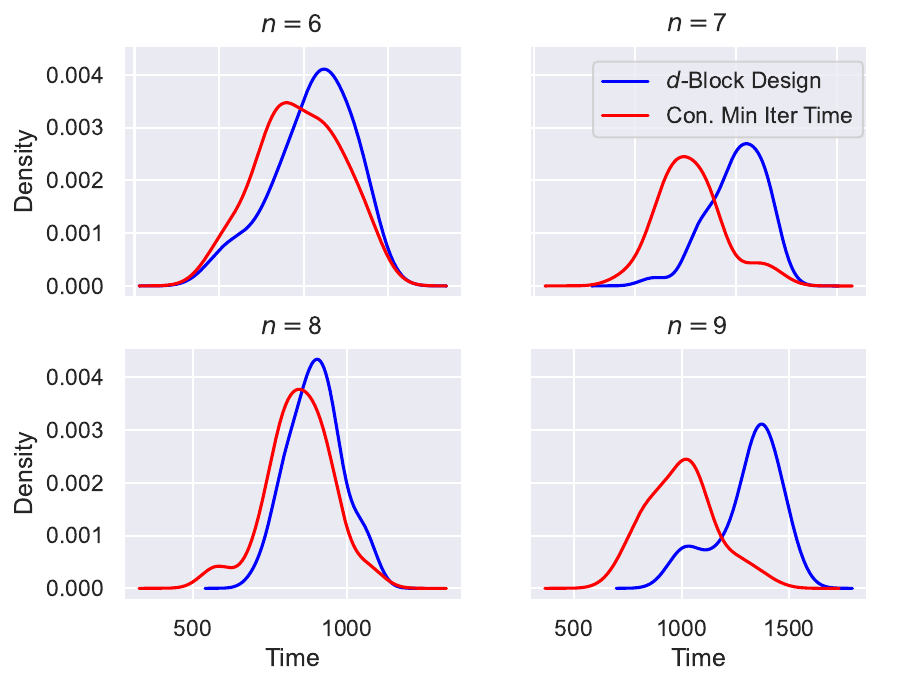}
    \caption{Comparison of $d$-Block and the constrained minimum iteration time design from \eqref{min_cycle_fmip} on randomly generated resolvent computation and communication times. $d$-Block designs are 2-Block for even $n$ and $\left(\frac{n-1}{2}, \frac{n-1}{2}, 1\right)$ for odd $n$ and use the minimum total effective resistance objective \eqref{resistance_prob}. Constrained Minimum Iteration Time designs use the MILP formulation \eqref{min_cycle_fmip} with the additional edge constraint described in the text.}
    \label{fig:con_mincycle}
\end{figure}

%% file: conclusion.tex
This work presents a novel framework for designing frugal resolvent splittings suitable for particular applications. We prove the convergence of these algorithms over $n$ maximal monotone operators whose sum has a zero, offering an enlarged range of valid step sizes and an approach which provides optimized step size given a characterization of the problem class. We establish the equivalence of all algorithms in \eqref{d_iteration} with minimal lifting designs. We also further develop the connection between the algorithm matrices and the weighted graph Laplacian, using the connection to inform the set of feasible constraints for \eqref{main_prob}. 

We demonstrate the utility of the framework with a number of possible constraint sets and objective functions. These allow minimization of iteration time for both uniform and arbitrary computation and communication times, as well as designs which optimize over a wide variety of common heuristics.
We also demonstrate the use of the PEP framework and its dual with these designs, optimizing $W$ and $\gamma$, and providing contraction factors under various structural assumptions on the monotone operators.

Numerically, we validate the performance of our contributions by comparing their performance in settings with different computation times, communication times, and problem classes. We show that $d$-Block designs, for small values of $d$, outperform competitors from the literature. We compare the designs optimized by the various spectral objectives, with and without the optimal step size, and find that minimizing the total effective resistance performs consistently well. We analyze algorithm design performance when the monotone operators are re-ordered, and find consistent performance across most designs, but note that the asymmetric extended Ryu design can perform poorly in this setting. Finally, we demonstrate that when resolvent computation times and communications times are known or can be accurately estimated, the minimal iteration time designs we propose in Section \ref{Sec:mip_formulations} outperform $d$-Block designs, but the marginal improvement suggests that $d$-Block designs still can be expected to exhibit good performance when these times are not available.

%% file: objective_formulations.tex
\section{Objective Function Formulations}\label{objective_formulations}

In this section we provide formulations of each of the spectral objective functions described in Section \ref{Sec:Objectives}. In each of these examples, the objective function is a linear combination of the same spectral objective function applied to $Z$ and $W$. 

\subsubsection*{Maximum Fiedler Value}
\begin{notation*}
  \beta_z \geq 0 & parameter for weighting the objective in $Z$\\
  \beta_w \geq 0 & parameter for weighting the objective in $W$\\
  \gamma_z \geq 0 & decision variable capturing the Fiedler value of $Z$\\
  \gamma_w \geq 0 & decision variable capturing the Fiedler value of $W$\\ 
\end{notation*}

\begin{subequations}\label{fiedler_prob}
  
\begin{align}
    \max_{Z, W, \gamma_z, \gamma_w} \quad & \beta_z\gamma_z + \beta_w\gamma_w \nonumber \\
      \text{s.t.} \quad & \lambda_1(Z) + \lambda_2(Z) \geq \gamma_z \\
      & \lambda_1(W) + \lambda_2(W) \geq \gamma_w \\
      & \text{SDP constraints \eqref{con1}-\eqref{con8}}
\end{align}

\end{subequations}

\subsubsection*{Minimal Second-Largest Eigenvalue Magnitude (SLEM)}

\begin{notation*}
  \beta_z \geq 0 & parameter for weighting the objective in $Z$\\
  \beta_w \geq 0 & parameter for weighting the objective in $W$\\
  \gamma_z \geq 0 & decision variable capturing the SLEM of $Z$\\
  \gamma_w \geq 0 & decision variable capturing the SLEM of $W$\\ 
\end{notation*}
\begin{subequations}\label{slem_prob}
\begin{align}
    \min_{Z, W, \gamma_z, \gamma_w} \quad & \beta_z\gamma_z + \beta_w\gamma_w \nonumber \\
      \text{s.t.}  \quad & -\gamma_z \I \preceq \I - \frac{1}{2+\epsilon}Z - \frac{1}{n}\1 \1^T \preceq \gamma_z \I \\
      & -\gamma_w \I \preceq \I - \frac{1}{2+\epsilon}W - \frac{1}{n}\1 \1^T \preceq \gamma_w \I \\
      & \text{SDP constraints \eqref{con1}-\eqref{con8}}
\end{align}

\end{subequations}

\subsubsection*{Minimal Total Effective Resistance}

\begin{notation*}
  \beta_z \geq 0 & parameter for weighting the objective in $Z$\\
  \beta_w \geq 0 & parameter for weighting the objective in $W$\\
  Y_z \in \Sp^n & supporting decision variable for capturing $\sum_{i=2}^n\frac{1}{\lambda_i(Z)}$\\
  Y_w \in \Sp^n & supporting decision variable for capturing $\sum_{i=2}^n\frac{1}{\lambda_i(W)}$\\ 
\end{notation*}
\begin{subequations}\label{resistance_prob}
\begin{align}
    \min_{Z, W, Y_z, Y_w} \quad & \beta_z\text{Tr}\left(Y_z\right) + \beta_w\text{Tr}\left(Y_w\right) \nonumber \\
      \text{s.t.}  \quad & \begin{bmatrix}
        Z + \1 \1^T/n & \I \\
        \I    & Y_z      \\
      \end{bmatrix} \succeq 0  \\
      & \begin{bmatrix}
        W + \1 \1^T/n & \I \\
        \I    & Y_w      \\
      \end{bmatrix} \succeq 0  \\
      & \text{SDP constraints \eqref{con1}-\eqref{con8}}
\end{align}

\end{subequations}

%% file: paper.bbl
\begin{thebibliography}{31}
\providecommand{\natexlab}[1]{#1}
\providecommand{\url}[1]{{#1}}
\providecommand{\urlprefix}{URL }
\providecommand{\doi}[1]{\url{https://doi.org/#1}}
\providecommand{\eprint}[2][]{\url{#2}}
 \bibcommenthead

\bibitem[{Arag{\'o}n-Artacho et~al(2024)Arag{\'o}n-Artacho, Campoy, and
  L{\'o}pez-Pastor}]{aragon2024forward}
Arag{\'o}n-Artacho FJ, Campoy R, L{\'o}pez-Pastor C (2024) Forward-backward
  algorithms devised by graphs. arXiv preprint arXiv:240603309

\bibitem[{Bauschke et~al(2011)Bauschke, Combettes et~al}]{bauschke_combettes}
Bauschke HH, Combettes PL, et~al (2011) Convex analysis and monotone operator
  theory in Hilbert spaces, vol 408. Springer

\bibitem[{Ben-Tal and Nemirovski(2001)}]{ben2001lectures}
Ben-Tal A, Nemirovski A (2001) Lectures on modern convex optimization:
  analysis, algorithms, and engineering applications. SIAM

\bibitem[{Boyd(2006)}]{boyd2006convex}
Boyd S (2006) Convex optimization of graph laplacian eigenvalues. In:
  Proceedings of the International Congress of Mathematicians, Citeseer, pp
  1311--1319

\bibitem[{Bredies et~al(2022)Bredies, Chenchene, Lorenz, and
  Naldi}]{bredies2022degenerate}
Bredies K, Chenchene E, Lorenz DA, et~al (2022) Degenerate preconditioned
  proximal point algorithms. SIAM Journal on Optimization 32(3):2376--2401

\bibitem[{Bredies et~al(2024)Bredies, Chenchene, and Naldi}]{bredies2024graph}
Bredies K, Chenchene E, Naldi E (2024) Graph and distributed extensions of the
  douglas--rachford method. SIAM Journal on Optimization 34(2):1569--1594

\bibitem[{Bynum et~al(2021)Bynum, Hackebeil, Hart, Laird, Nicholson, Siirola,
  Watson, Woodruff et~al}]{bynum2021pyomo}
Bynum ML, Hackebeil GA, Hart WE, et~al (2021) Pyomo-optimization modeling in
  python, vol~67. Springer

\bibitem[{Campoy(2022)}]{campoy2022product}
Campoy R (2022) A product space reformulation with reduced dimension for
  splitting algorithms. Computational Optimization and Applications
  83(1):319--348

\bibitem[{Chen et~al(2016)Chen, He, Ye, and Yuan}]{chen2016direct}
Chen C, He B, Ye Y, et~al (2016) The direct extension of admm for multi-block
  convex minimization problems is not necessarily convergent. Mathematical
  Programming 155(1-2):57--79

\bibitem[{Chung(1997)}]{chung1997spectral}
Chung FR (1997) Spectral graph theory, vol~92. American Mathematical Soc.

\bibitem[{Colla and Hendrickx(2024)}]{colla2024optimal}
Colla S, Hendrickx JM (2024) On the optimal communication weights in
  distributed optimization algorithms. arXiv preprint arXiv:240205705

\bibitem[{Condat et~al(2023)Condat, Kitahara, Contreras, and
  Hirabayashi}]{condat2023proximal}
Condat L, Kitahara D, Contreras A, et~al (2023) Proximal splitting algorithms
  for convex optimization: A tour of recent advances, with new twists. SIAM
  Review 65(2):375--435

\bibitem[{Douglas and Rachford(1956)}]{douglas1956numerical}
Douglas J, Rachford HH (1956) On the numerical solution of heat conduction
  problems in two and three space variables. Transactions of the American
  mathematical Society 82(2):421--439

\bibitem[{Drori and Teboulle(2014)}]{drori2014performance}
Drori Y, Teboulle M (2014) Performance of first-order methods for smooth convex
  minimization: a novel approach. Mathematical Programming 145(1):451--482

\bibitem[{Eckstein and Bertsekas(1992)}]{eckstein1992douglas}
Eckstein J, Bertsekas DP (1992) On the douglas—rachford splitting method and
  the proximal point algorithm for maximal monotone operators. Mathematical
  programming 55:293--318

\bibitem[{Eckstein and Svaiter(2009)}]{eckstein2009general}
Eckstein J, Svaiter BF (2009) General projective splitting methods for sums of
  maximal monotone operators. SIAM Journal on Control and Optimization
  48(2):787--811

\bibitem[{Fiedler(1973)}]{fiedler1973algebraic}
Fiedler M (1973) Algebraic connectivity of graphs. Czechoslovak mathematical
  journal 23(2):298--305

\bibitem[{Fiedler(1975)}]{fiedler1975property}
Fiedler M (1975) A property of eigenvectors of nonnegative symmetric matrices
  and its application to graph theory. Czechoslovak mathematical journal
  25(4):619--633

\bibitem[{Gally et~al(2018)Gally, Pfetsch, and Ulbrich}]{gally2018framework}
Gally T, Pfetsch ME, Ulbrich S (2018) A framework for solving mixed-integer
  semidefinite programs. Optimization Methods and Software 33(3):594--632

\bibitem[{Goujaud et~al(2022)Goujaud, Moucer, Glineur, Hendrickx, Taylor, and
  Dieuleveut}]{pepit2022}
Goujaud B, Moucer C, Glineur F, et~al (2022) {PEPit}: computer-assisted
  worst-case analyses of first-order optimization methods in {P}ython. arXiv
  preprint arXiv:220104040

\bibitem[{{Gurobi Optimization, LLC}(2023)}]{gurobi}
{Gurobi Optimization, LLC} (2023) {Gurobi Optimizer Reference Manual}.
  \urlprefix\url{https://www.gurobi.com}

\bibitem[{Kelley~Jr(1961)}]{kelley1961critical}
Kelley~Jr JE (1961) Critical-path planning and scheduling: Mathematical basis.
  Operations research 9(3):296--320

\bibitem[{Malitsky and Tam(2023)}]{malitsky2023resolvent}
Malitsky Y, Tam MK (2023) Resolvent splitting for sums of monotone operators
  with minimal lifting. Mathematical Programming 201(1-2):231--262

\bibitem[{Rockafellar(1970{\natexlab{a}})}]{rockafellar1970convex}
Rockafellar R (1970{\natexlab{a}}) Convex analysis. Princeton Math Series 28

\bibitem[{Rockafellar(1970{\natexlab{b}})}]{rockafellar1970maximal}
Rockafellar R (1970{\natexlab{b}}) On the maximal monotonicity of
  subdifferential mappings. Pacific Journal of Mathematics 33(1):209--216

\bibitem[{Ryu(2020)}]{ryu2020uniqueness}
Ryu EK (2020) Uniqueness of drs as the 2 operator resolvent-splitting and
  impossibility of 3 operator resolvent-splitting. Mathematical Programming
  182(1-2):233--273

\bibitem[{Ryu and Yin(2019)}]{JCM-37-778}
Ryu EK, Yin W (2019) Proximal-proximal-gradient method. Journal of
  Computational Mathematics 37(6):778--812.
  \doi{https://doi.org/10.4208/jcm.1906-m2018-0282},
  \urlprefix\url{http://global-sci.org/intro/article_detail/jcm/13374.html}

\bibitem[{Ryu et~al(2020)Ryu, Taylor, Bergeling, and
  Giselsson}]{ryu2020operator}
Ryu EK, Taylor AB, Bergeling C, et~al (2020) Operator splitting performance
  estimation: Tight contraction factors and optimal parameter selection. SIAM
  Journal on Optimization 30(3):2251--2271

\bibitem[{Tam(2023)}]{tam2023frugal}
Tam MK (2023) Frugal and decentralised resolvent splittings defined by
  nonexpansive operators. Optimization Letters pp 1--19

\bibitem[{Vandenberghe and Boyd(1996)}]{vandenberghe1996semidefinite}
Vandenberghe L, Boyd S (1996) Semidefinite programming. SIAM review
  38(1):49--95

\bibitem[{Yurtsever et~al(2021)Yurtsever, Tropp, Fercoq, Udell, and
  Cevher}]{yurtsever2021scalable}
Yurtsever A, Tropp JA, Fercoq O, et~al (2021) Scalable semidefinite
  programming. SIAM Journal on Mathematics of Data Science 3(1):171--200

\end{thebibliography}
